\documentclass[11pt]{amsart}
\usepackage[margin=1.25in]{geometry}
\usepackage[utf8]{inputenc}
\usepackage{latexsym,amsxtra,amscd,ifthen,amsmath,color, multicol,mathdots,fancyvrb,url}
\usepackage{amsfonts}
\usepackage{amssymb}
\usepackage{verbatim}
\usepackage{amsmath}
\usepackage{amsthm}
\usepackage{mathtools}
\usepackage{enumerate}
\usepackage{enumitem}
\usepackage{tikz-cd}
\usepackage[all]{xy}
 \usepackage{xcolor}
 \usepackage{adjustbox}
 \usepackage{mathrsfs}
 \usepackage{comment}
 \usepackage{todonotes}
 \usepackage{quiver}
 \usepackage{hyperref}
 \usepackage{multicol}

\definecolor{cyan(process)}{rgb}{0.0, 0.6, 1.0}
\definecolor{blue-violet}{rgb}{0.54, 0.17, 0.89}

\newcommand{\ku}{ \Bbbk}
\newcommand{\N}{\mathbb{N}}
\newcommand{\cA}{\mathcal{A}}
\newcommand{\cB}{\mathcal{B}}
\newcommand{\cC}{\mathcal{C}}

\newcommand{\cJ}{\mathcal{J}}

\newcommand{\Z}{\mathbb{Z}}
\newcommand{\I}{\mathbb I}
\newcommand{\bF}{\mathbf{F}}

\newcommand{\bL}{\mathbf{L}}
\newcommand{\bR}{\mathbf{R}}

\newcommand{\ad}{\operatorname{ad}}
\newcommand{\pt}{\operatorname{pt}}
\newcommand{\degf}[1]{\left|#1\right|}
\newcommand{\vect}{\operatorname{vec}}
\newcommand{\Irr}{\operatorname{Irr}}
\newcommand{\ot}{\otimes}

\newcommand{\fpdim}{\operatorname{FPdim}}
\newcommand{\uno}{\mathbf{1}}
\newcommand{\id}{\operatorname{id}}
\newcommand{\Hom}{\operatorname{Hom}}
\newcommand{\Aut}{\operatorname{Aut}}

 \newcommand{\trid}{\triangleright}
\newcommand{\fiz}{\triangleleft}
\newcommand{\tridb}{\blacktriangleright}
\newcommand{\fizb}{\blacktriangleleft}
\newcommand{\sgn}{\operatorname{Sgn}}

\newcommand{\stkout}[1]{\ifmmode\text{\sout{\ensuremath{#1}}}\else\sout{#1}\fi}

\newcommand{\objpentagonOne}{\left((A\ot \degf{C}\trid A')\ot \left((\degf{C}\fizb \degf{A'}) \degf{C'}\right)\trid A''\right) \ot \left(((\degf{C}\fizb \degf{A'})\degf{C'})\fizb \degf{A''}\right) \degf{C''}\trid A'''}
\newcommand{\objpentagonTwo}{\left(A\ot ((\degf{C}\trid A')\ot((\degf{C}\fizb \degf{A'}) \degf{C'})\trid A'' )\right)\ot ((((\degf{C}\fizb \degf{A'})\degf{C'})\fizb \degf{A''})\degf{C''})\trid A'''}
\newcommand{\objpentagonThree}{A\ot ((\degf{C}\trid A' \ot (\degf{C}\fizb \degf{A'})\degf{C'}\trid A'') \ot ((((\degf{C}\fizb\degf{A'})\degf{C'})\fizb\degf{A''})\degf{C'})\trid A''' )}
\newcommand{\objpentagonFour}{(A\ot \degf{C}\trid A')\ot (((\degf{C}\fizb \degf{A'})\degf{C'})\trid A''\ot ((((\degf{C}\fizb\degf{A'})\degf{C'})\fizb\degf{A''})\degf{C''})\trid A''')}
\newcommand{\objpentagonFive}{A\ot (\degf{C}\trid A' \ot ((\degf{C}\fizb \degf{A'})\degf{C'}\trid A''\ot ((((\degf{C}\fizb\degf{A'})\degf{C'})\fizb\degf{A''})\degf{C'})\trid A''' ))}
\newcommand{\objpentagonSix}{\left(A\ot ((\degf{C}\trid A')\ot(\degf{C}\fizb \degf{A'})\trid(\degf{C'}\trid A'') )\right)\ot ((((\degf{C}\fizb \degf{A'})\degf{C'})\fizb \degf{A''})\degf{C''})\trid A'''}
\newcommand{\objpentagonSeven}{A\ot \left(((\degf{C}\trid A')\ot(\degf{C}\fizb \degf{A'})\trid(\degf{C'}\trid A'') )\ot ((((\degf{C}\fizb \degf{A'})\degf{C'})\fizb \degf{A''})\degf{C''})\trid A'''\right)}
\newcommand{\objpentagonEight}{A\ot \left((\degf{C}\trid A')\ot((\degf{C}\fizb \degf{A'})\trid(\degf{C'}\trid A'') \ot ((((\degf{C}\fizb \degf{A'})\degf{C'})\fizb \degf{A''})\degf{C''})\trid A'''\right))}
\newcommand{\objpentagonNine}{A\ot (\degf{C}\trid A' \ot ((\degf{C}\fizb \degf{A'})\degf{C'}\trid A''\ot ((((\degf{C}\fizb\degf{A'})\degf{C'})\fizb\degf{A''}))\trid(\degf{C'}\trid A''' )))}
\newcommand{\objpentagonTen}{(A\ot \degf{C}\trid A')\ot (((\degf{C}\fizb \degf{A'})\degf{C'})\trid A'' \ot (((\degf{C}\fizb\degf{A'})\degf{C'})\fizb\degf{A''})\trid (\degf{C''}\trid A''') )}
\newcommand{\objpentagonEleven}{(A\ot \degf{C}\trid A')\ot(((\degf{C}\fizb \degf{A'})\degf{C'})\trid (A''\ot \degf{C''}\trid A'''))}
\newcommand{\objpentagonTwelve}{A\ot (\degf{C}\trid A'\ot(((\degf{C}\fizb \degf{A'})\degf{C'})\trid (A''\ot \degf{C''}\trid A''')))}
\newcommand{\objpentagonThirteen}{A\ot (\degf{C}\trid A'\ot(((\degf{C}\fizb \degf{A'}))\trid(\degf{C'}\trid (A''\ot \degf{C''}\trid A'''))))}
\newcommand{\objpentagonFourteen}{A\ot (\degf{C}\trid(A'\ot \degf{C'}\trid A'')\ot ((\degf{C}\fizb (\degf{A'}(\degf{C'}\tridb\degf{A''})))(\degf{C'}\fizb\degf{A''})\degf{C''})\trid A''')}
\newcommand{\objpentagonFifteen}{A\ot (\degf{C}\trid(A'\ot \degf{C'}\trid A'')\ot (\degf{C}\fizb (\degf{A'}(\degf{C'}\tridb\degf{A''})))\trid(((\degf{C'}\fizb\degf{A''})\degf{C''})\trid A'''))}
\newcommand{\objpentagonSixteen}{A\ot ((\degf{C}\trid A'\ot (\degf{C}\fizb \degf{A'})\trid(\degf{C'}\trid A''))\ot ((\degf{C}\fizb \degf{A'})\fizb(\degf{C'}\fizb \degf{A''}))\trid (((\degf{C'}\fizb \degf{A''})\degf{C''})\trid A'''))}
\newcommand{\objpentagonSeventeen}{A\ot \degf{C}\trid (A'\ot \degf{C'}\trid (A''\ot \degf{C''}\trid A'''))}
\newcommand{\objpentagonEighteen}{A\ot \degf{C}\trid (A'\ot (\degf{C'}\trid A''\ot(\degf{C'}\fizb\degf{A''})\trid(\degf{C''}\trid A''')))}
\newcommand{\objpentagonNineteen}{A\ot (\degf{C}\trid A' \ot (\degf{C}\fizb \degf{A'})\trid (\degf{C'}\trid A'' \ot (\degf{C'}\fizb \degf{A''})\trid(\degf{C''}\trid A''')))}
\newcommand{\objpentagonTwenty}{A\ot \degf{C}\trid((A'\ot \degf{C'}\trid A'')\ot((\degf{C'}\fizb \degf{A''})\degf{C''})\trid A''')}
\newcommand{\objpentagonTwentyOne}{A\ot \degf{C}\trid(A'\ot (\degf{C'}\trid A''\ot ((\degf{C'}\fizb\degf{A''})\degf{C''})\trid A'''))}
\newcommand{\objpentagonTwentyTwo}{A\ot (\degf{C}\trid A'\ot (\degf{C}\fizb\degf{A'})\trid(\degf{C'}\trid A''\ot ((\degf{C'}\fizb\degf{A''})\degf{C''})\trid A'''))}
\newcommand{\objpentagonTwentyThree}{A\ot (\degf{C}\trid A'\ot ((\degf{C}\fizb \degf{A'})\trid(\degf{C'}\trid A'')\ot ((\degf{C}\fizb \degf{A'})\fizb(\degf{C'}\fizb \degf{A''}))\trid (((\degf{C'}\fizb \degf{A''})\degf{C''})\trid A''')))}
\newcommand{\objpentagonTwentyFour}{A\ot (\degf{C}\trid A'\ot((\degf{C}\fizb\degf{A'})\trid(\degf{C'}\trid A'')\ot(\degf{C}\fizb(\degf{A'}(\degf{C'}\tridb\degf{A''})))\trid((\degf{C'}\fizb\degf{A''})\trid(\degf{C''}\trid A'''))))}
\newcommand{\objpentagonTwentyFive}{A\ot(\degf{C}\trid A'\ot ((\degf{C}\fizb\degf{A'})\trid(\degf{C'}\trid A'')\ot(((\degf{C}\fizb\degf{A'})\degf{C'})\fizb\degf{A''})\trid(\degf{C''}\trid A''')))}
\newcommand{\objpentagonTwentySix}{(A\ot \degf{C}\trid(A'\ot \degf{C'}\trid A''))\ot ((\degf{C}\fizb (\degf{A'}(\degf{C'}\tridb\degf{A''})))(\degf{C'}\fizb\degf{A''})\degf{C''})\trid A'''}

\newcommand{\objpentmapOneToTwo}{\alpha_{A,\degf{C}\trid A', (\degf{C}\fizb \degf{A'})\degf{C'}\trid A''}\ot\id_{((((\degf{C}\fizb \degf{A'})\degf{C'})\fizb \degf{A''})\degf{C''})\trid A'''}}
\newcommand{\objpentmapOneToFour}{\alpha_{A \ot \degf{C}\trid A', ((\degf{C}\fizb \degf{A'})\degf{C'})\trid A'', ((((\degf{C}\fizb \degf{A'})\degf{C'})\fizb\degf{A''})\degf{C''})\trid A'''}}
\newcommand{\objpentmapTwoToThree}{\alpha_{A, \degf{C}\trid A'\ot (\degf{C}\fizb \degf{A'})\degf{C'}\trid A'' ,(((\degf{C}\fizb\degf{A'})\degf{C'})\fizb \degf{A''})\degf{C''}\trid A'''}}
\newcommand{\objpentmapTwoToSix}{\id_A \ot \id_{\degf{C}\trid A'}\ot (\bL^2_{\degf{C}\fizb\degf{A'}, \degf{C'}})^{-1}_{A''}\ot \id_{((((\degf{C}\fizb \degf{A'})\degf{C'})\fizb \degf{A''})\degf{C''})\trid A'''}}
\newcommand{\objpentmapThreeToFive}{\id_A\ot\alpha_{\degf{C}\trid A', (\degf{C}\fizb \degf{A'})\degf{C'}\trid A'' ,((((\degf{C}\fizb \degf{A'})\degf{C'})\fizb \degf{A''})\degf{C''})\trid A'''}}
\newcommand{\objpentmapThreeToSeven}{\id_A \ot \id_{\degf{C}\trid A'}\ot (\bL^2_{\degf{C}\fizb\degf{A'}, \degf{C'}})^{-1}_{A''}\ot \id_{((((\degf{C}\fizb \degf{A'})\degf{C'})\fizb \degf{A''})\degf{C''})\trid A'''}}
\newcommand{\objpentmapFourToFive}{\alpha_{A, \degf{C}\trid A' ,\degf{C'}\trid A''\ot ((((\degf{C}\fizb \degf{A'})\degf{C'})\fizb\degf{A''})\degf{C''})\trid A'''}}
\newcommand{\objpentmapFourToTen}{\id_{A\ot \degf{C}\trid A'}\ot \id_{((\degf{C}\fizb\degf{A'})\degf{C'})\trid A''} \ot (\bL^2_{((\degf{C}\fizb\degf{A'})\degf{C'})\fizb\degf{A''},\degf{C''}})^{-1}_{A'''}}
\newcommand{\objpentmapFiveToEight}{\id_A \ot \id_{\degf{C}\trid A'}\ot (\bL_{\degf{C}\fizb\degf{A'},\degf{C'}}^2)^{-1}_{A''}\ot \id_{((((\degf{C}\fizb \degf{A'})\degf{C'})\fizb \degf{A''})\degf{C''})\trid A'''}}
\newcommand{\objpentmapFiveToNine}{\id_A \ot \id_{\degf{C}\trid A'} \ot \id_{((\degf{C}\fizb\degf{A'})\degf{C'})\trid A''} \ot (\bL^2_{((\degf{C}\fizb \degf{A'})\degf{C'})\fizb\degf{A''},\degf{C''}})^{-1}_{A'''}}
\newcommand{\objpentmapSixToSeven}{\alpha_{A,\degf{C}\trid A' \ot (\degf{C}\fizb\degf{A'})\trid(\degf{C'}\trid A'') ,((((\degf{C}\fizb \degf{A'})\degf{C'})\fizb \degf{A''})\degf{C''})\trid A'''}}
\newcommand{\objpentmapSixToTwentysix}{\id_A\ot (\gamma^{\degf{C}}_{A', \degf{C'}\trid A''})^{-1}\ot \id_{((((\degf{C}\fizb \degf{A'})\degf{C'})\fizb \degf{A''})\degf{C''})\trid A'''}}
\newcommand{\objpentmapSevenToEight}{\id_A \ot \alpha_{\degf{C}\trid A', (\degf{C}\fizb\degf{A'})\trid(\degf{C'}\trid A''), ((((\degf{C}\fizb \degf{A'})\degf{C'})\fizb \degf{A''})\degf{C''})\trid A'''}}
\newcommand{\objpentmapSevenToFourteen}{\id_A \ot (\gamma^{\degf{C}}_{A', \degf{C'}\trid A''})^{-1} \ot \id_{((((\degf{C}\fizb \degf{A'})\degf{C'})\fizb \degf{A''})\degf{C''})\trid A'''}}
\newcommand{\objpentmapSevenToSixteen}{\id_A \ot \id_{\degf{C}\trid A'} \ot \id_{(\degf{C} \fizb \degf{A'})\trid(\degf{C'}\trid A'')} \ot (\bL^2_{\degf{C}\fizb (\degf{A'}(\degf{C'}\tridb \degf{A''})),(\degf{C'}\fizb \degf{A''})\degf{C''}})^{-1}_{A'''}}
\newcommand{\objpentmapEightToTwentythree}{\id_A \ot \id_{\degf{C}\trid A'}\ot \id_{(\degf{C} \fizb \degf{A'})\trid(\degf{C'}\trid A'')} \ot (\bL^2_{\degf{C}\fizb (\degf{A'}(\degf{C'}\tridb \degf{A''})),(\degf{C'}\fizb\degf{A''})\degf{C''}})^{-1}_{A'''}}
\newcommand{\objpentmapEightToTwentyfive}{\id_A \ot \id_{\degf{C}\trid A'} \ot \id_{(\degf{C}\fizb \degf{A'})\trid(\degf{C'}\trid A'')} \ot (\bL^2_{(\degf{C}\fizb (\degf{A'}(\degf{C'}\tridb \degf{A''})))(\degf{C'}\fizb \degf{A''}), \degf{C''}})^{-1}_{A'''}}
\newcommand{\objpentmapNineToTwelve}{\id_A \ot \id_{\degf{C}\trid A'}\ot (\gamma^{(\degf{C}\fizb \degf{A'})\degf{C'}}_{A'',\degf{C''}\trid A'''})^{-1}}
\newcommand{\objpentmapNineToTwentyfive}{\id_A \ot \id_{\degf{C}\trid A'} \ot (\bL^2_{\degf{C}\fizb \degf{A'},\degf{C'}})^{-1}_{A''} \ot \id_{(((\degf{C}\fizb \degf{A'})\degf{C'})\fizb \degf{A''})\trid (\degf{C''}\trid A''')}}
\newcommand{\objpentmapTenToNine}{\alpha_{A, \degf{C}\trid A', ((\degf{C}\fizb \degf{A'})\degf{C'})\trid A''\ot (((\degf{C}\fizb \degf{A'})\degf{C'})\fizb \degf{A''})\trid (\degf{C''}\trid A''') }}
\newcommand{\objpentmapTenToEleven}{\id_{A\ot \degf{C}\trid A'}\ot (\gamma^{(\degf{C}\fizb \degf{A'})\degf{C'}}_{A'', \degf{C''}\trid A'''})^{-1}}
\newcommand{\objpentmapElevenToTwelve}{\alpha_{A, \degf{C}\trid A', ((\degf{C}\fizb \degf{A'})\degf{C'})\trid (A''\ot \degf{C''}\trid A''')}}
\newcommand{\objpentmapTwelveToThirdteen}{\id_A \ot \id_{\degf{C}\trid A'}\ot (\bL^2_{\degf{C}\fizb \degf{A'}, \degf{C'}})^{-1}_{A''\ot \degf{C''}\trid A'''}}
\newcommand{\objpentmapThirdteenToSeventeen}{\id_A \ot (\gamma^{\degf{C}}_{A', \degf{C'}\trid (A''\ot \degf{C''}\trid A''')})^{-1}}
\newcommand{\objpentmapFourteenToFifthteen}{\id_A \ot \id_{\degf{C}\trid (A' \ot \degf{C'}\trid A'')} \ot (\bL^2_{\degf{C}\fizb (\degf{A'}(\degf{C'}\tridb \degf{A''})),(\degf{C'}\fizb \degf{A''})\degf{C''}})^{-1}_{A'''}}
\newcommand{\objpentmapFifthteenToTwenty}{\id_A \ot (\gamma^{\degf{C}}_{A'\ot \degf{C'}\trid A'',((\degf{C'}\fizb \degf{A''})\degf{C''})\trid A'''})^{-1}}
\newcommand{\objpentmapSixteenToFifthteen}{\id_A \ot (\gamma^{\degf{C}}_{A',\degf{C'}\trid A''})^{-1}\ot \id_{(\degf{C}\fizb (\degf{A'} (\degf{C'}\tridb \degf{A''})))\trid (((\degf{C'}\fizb \degf{A''})\degf{C''})\trid A''')}}
\newcommand{\objpentmapSixteenToTwentythree}{\id_A \ot \alpha_{\degf{C}\trid A',(\degf{C}\fizb \degf{A'})\trid(\degf{C'}\trid A'') , (\degf{C}\fizb (\degf{A'} (\degf{C'}\tridb \degf{A''})))\trid (((\degf{C'}\fizb \degf{A''})\degf{C''})\trid A''') }}
\newcommand{\objpentmapEighteenToSeventeen}{\id_A \ot \degf{C}\trid (\id_{A'}\ot (\gamma^{\degf{C'}}_{A'',\degf{C''}\trid A'''})^{-1})}
\newcommand{\objpentmapNineteenToThirdteen}{\id_A \ot \id_{\degf{C}\trid A'} \ot (\degf{C}\fizb \degf{A'})\trid (\gamma^{\degf{C'}}_{A'', \degf{C''}\trid A'''})^{-1}}
\newcommand{\objpentmapNineteenToEighteen}{\id_A \ot (\gamma^{\degf{C}}_{A', \degf{C'}\trid A'' \ot (\degf{C'}\fizb \degf{A''})\trid (\degf{C''}\trid A''')})^{-1}}
\newcommand{\objpentmapTwentyToTwentyone}{\id_A \ot \degf{C}\trid \alpha_{A', \degf{C'}\trid A'' ,((\degf{C'}\fizb \degf{A''})\degf{C''})\trid A'''}}
\newcommand{\objpentmapTwentyoneToEighteen}{\id_A \ot \degf{C}\trid (\id_{A'}\ot \id_{\degf{C'}\trid A''} \ot (\bL^2_{\degf{C'}\fizb \degf{A''},\degf{C''}})^{-1}_{A'''})}
\newcommand{\objpentmapTwentytwoToNineteen}{\id_A \ot \id_{\degf{C}\trid A'} \ot (\degf{C}\fizb \degf{A'})\trid (\id_{\degf{C'}\trid A''}\ot (\bL^2_{\degf{C'}\fizb \degf{A''}, \degf{C''}})^{-1}_{A'''})}
\newcommand{\objpentmapTwentytwoToTwentyone}{\id_A \ot (\gamma^{\degf{C}}_{A', \degf{C'}\trid A''\ot ((\degf{C'}\fizb \degf{A''})\degf{C''})\trid A''' })^{-1}}
\newcommand{\objpentmapTwentythreeToTwentytwo}{\id_A \ot \id_{\degf{C}\trid A'} \ot (\gamma^{\degf{C}\fizb \degf{A'}}_{\degf{C'}\trid A'', ((\degf{C'}\fizb \degf{A''})\degf{C''})\trid A'''})^{-1}}
\newcommand{\objpentmapTwentythreeToTwentyfour}{\id_A \ot \id_{\degf{C}\trid A'} \ot \id_{(\degf{C}\fizb \degf{A'})\trid(\degf{C'}\trid A'')} \ot (\degf{C}\fizb(\degf{A'} (\degf{C'}\tridb\degf{A''})))\trid (\bL^2_{\degf{C'} \fizb \degf{A''}, \degf{C''}})^{-1}_{A'''}}
\newcommand{\objpentmapTwentyfourToNineteen}{\id_A \ot \id_{\degf{C}\trid A'} \ot (\gamma^{\degf{C}\fizb \degf{A'}}_{\degf{C'}\trid A'', (\degf{C'}\fizb \degf{A''})\trid (\degf{C''}\trid A''')})^{-1}}
\newcommand{\objpentmapTwentyfiveToTwentyfour}{\id_A \ot \id_{\degf{C}\trid A'} \id_{((\degf{C}\fizb \degf{A'})\degf{C'})\trid A''} \ot (\bL^2_{\degf{C}\fizb(\degf{A'}(\degf{C'}\tridb\degf{A''})), \degf{C'}\fizb \degf{A''}})^{-1}_{\degf{C''}\trid A'''}}
\newcommand{\objpentmapTwentysixToFourteen}{\alpha_{A, \degf{C}\trid (A' \ot \degf{C'}\trid A'') , ((\degf{C}\fizb (\degf{A'} (\degf{C'}\tridb \degf{A''})) )(\degf{C'} \fizb \degf{A''})\degf{C''})\trid A'''}}

\title[Bicrossed product and exact factorizations]{On bicrossed product of fusion categories and exact factorizations}

\usepackage{graphicx}

\newtheorem{theorem}{Theorem}[section]
\newtheorem{lemma}[theorem]{Lemma}
\newtheorem{coro}[theorem]{Corollary}

\newtheorem{prop}[theorem]{Proposition}

\theoremstyle{definition}
\newtheorem{definition}[theorem]{Definition}
\newtheorem{example}[theorem]{Example}
\newtheorem{question}{Question}

\newtheorem*{claim}{Claim}

\makeatletter
\newenvironment{proofw}{\par
  \pushQED{\qed}%
  \normalfont \topsep6\p@\@plus6\p@\relax
  \trivlist
  \item[]\ignorespaces
}{%
  \popQED\endtrivlist\@endpefalse
}
\makeatother

\theoremstyle{remark}
\newtheorem{remark}[theorem]{Remark}

\author[M. M\"uller]{Monique M\"uller}
\address{Departamento de Matem\'atica e Estat\'istica, Universidade Federal de S\~ao
Jo\~ao del-Rei, Brazil}
\email{monique@ufsj.edu.br}
\author[H. M. Pe\~na Pollastri]{H\'ector Mart\'in Pe\~na Pollastri}
\address{Department of Mathematics, Indiana University, USA}
\email{hpenapol@iu.edu}
\author[J. Plavnik]{Julia Plavnik}
\address{Department of Mathematics, Indiana University, USA \& Fachbereich Mathematik, Universit\"at Hamburg, Germany}
\email{jplavnik@iu.edu}

\begin{document}

\begin{abstract}
We introduce the notion of a matched pair of fusion categories (fusion rings), generalizing the one for groups. Using this concept, we define the bicrossed product of fusion categories (fusion rings) and construct exact factorizations for them. This concept generalizes the bicrossed product of groups, also known as the external Zappa-Sz\'ep product. We also show that every exact factorization of fusion rings can be presented as a bicrossed product. With this characterization, we describe the adjoint subcategory and universal grading group of an exact factorization of fusion categories. We give explicit fusion rules and associativity constraints for examples of fusion categories arising as a bicrossed product of combinations of Tambara-Yamagami and pointed fusion categories. 
\end{abstract}

\maketitle

\setcounter{tocdepth}{1}
\tableofcontents

\section{Introduction}
Fusion categories are rich mathematical objects with many connections to different areas of mathematics and physics. In the foundational work of Moore and Seiberg \cite{ms}, they proposed that conformal quantum field theory can be viewed as a generalization of group theory and that the correct way to axiomatize quantum field theory is with what we call today modular fusion categories. Other applications of fusion categories involve the construction of knot invariants as done in \cite{turaev-viro} by Turaev and Viro and \cite{rt} by Reshetikhin and Turaev.

To study any algebraic (or mathematical) object, it is of major importance to construct new examples from known ones and to understand how complex objects can be decomposed into simpler ones. As fusion categories can be thought of as generalizations of finite groups, a very fruitful strategy is to mimic notions from group theory in the fusion categorical settings. One such notion is that of exact factorization of groups that generalize the notions of direct and semidirect products. A group $G$ is said to be an exact factorization of two subgroups $H$ and $K$ if $G = HK$ and $H\cap K = \{e\}$. This notion was defined for fusion categories by Gelaki in \cite[Definition 3.4]{G-exact-factorization} and further studied for tensor categories in \cite{generalization-G-nonsemisimple}.

A fusion category can be thought of as its Grothendieck ring, which is a fusion ring, plus associativity coherence data. In this article, we provide answers to both of the following natural questions.

\begin{question}\label{question:how-are-exact-fact-of-fusion-rings}
	What is the fusion ring of an exact factorization of fusion categories?
\end{question}

\begin{question}\label{question:contruction-exact-fact-fusion}
	How can we construct exact factorizations of fusion categories?
\end{question}
 For Question \ref{question:how-are-exact-fact-of-fusion-rings}, our approach is to generalize the answer to a similar question for the group case. Exact factorizations of groups $G = H\bullet K$ are completely characterized by actions $\tridb\colon K\times H\to H$ and $\fizb\colon K\times H \to K$ satisfying certain conditions (see Subsection \ref{subsection:exact-fact-groups}). The tuple $(H,K,\tridb,\fizb)$ is called a matched pair of groups, and an exact factorization of groups can be recovered from these data using a construction called the bicrossed product $H\bowtie K$ of groups, or the  Zappa–Sz\'ep product. Every exact factorization of groups can be obtained like this. One of the contributions of this work is the generalization of the previously mentioned notions to the context of fusion categories. A key insight that enabled this generalization is the central role played by the universal grading group of a fusion category in controlling the structure of an exact factorization. More specifically, it was possible to define actions that depend only on the graded component in which an object lives, that is, have a (grading) group acting on each of the categories. This can already be seen at the level of the Grothendieck ring, that is, at the level of fusion rings.

 We define the notion of matched pair of fusion rings that generalizes the one from groups (Definition \ref{def:mathced-pair-fusion-rings}) using the notion of grading of a fusion ring. In addition, we construct an exact factorization of fusion rings from this data via what we call the bicrossed product of fusion rings (Theorem \ref{thm:bicrossed-product-fusion-rings}). Moreover, we prove that every exact factorization of fusion rings is obtained like this (Theorem \ref{teo:char-exact-fact-fusion-ring}). This completes the answer to Question \ref{question:how-are-exact-fact-of-fusion-rings}. 

As an application of this characterization, we obtain results about every exact factorization of fusion categories related to concepts that depend only on the fusion ring. For example, we prove that an exact factorization  $\cB = \cA \bullet \cC$ is nilpotent if and only if $\cA$ and $\cC$ are nilpotent (Corollary \ref{coro:nilpotent-exact-fact}). In addition, we show that the universal grading group of $\cB$ is given by an exact factorization of the universal grading groups of $\cA$ and $\cC$, and that there is an exact factorization of the adjoint subcategories $\cB_{\ad} = \cA_{\ad}\bullet \cC_{\ad}$ (Proposition \ref{prop:exact-fact-of-univ-grad-group}) for these last two.

To answer Question \ref{question:contruction-exact-fact-fusion}, we generalize the notion of a matched pair of fusion rings to get a definition of matched pair of fusion categories, see Definition \ref{def:mathced-pair-categories}. In analogy to the group setting, with the data of a matched pair of fusion categories, we define the notion of \emph{bicrossed product of fusion categories}, which is an exact factorization of fusion categories. Hence, the following question is natural.

\begin{question}
Is every exact factorization of fusion categories given by a bicrossed product? If not, how much can the associativities differ from the one given by the bicrossed product?
\end{question}

We know that every exact factorization of fusion rings is given by a bicrossed product of fusion rings, and for the fusion categories setting, the question is still open as we pointed out in Remark \ref{remark-clave-2}.

As an application of the bicrossed product of fusion categories, we construct exact factorizations of fusion categories between a Tambara-Yamagami fusion category and a pointed one, and between two Tambara-Yamagami fusion categories. We emphasize that explicit formulas for the associator data can be obtained from this construction.

\subsection*{Organization of the article} In Section \ref{section:Preliminaries}, we introduce the relevant definitions and notation. Section \ref{section:exact-fact-fusion-rings} is about the exact factorization of fusion rings and is divided into three main subsections. Subsection \ref{subsection:bicrossed-product-fusion-rings} defines the notion of a matched pair for fusion rings, and from these data we construct exact factorizations of fusion rings; we call this construction the bicrossed product, this is Theorem \ref{thm:bicrossed-product-fusion-rings}. Subsection \ref{subsection:characterization-fusion-rings} is entirely dedicated to the proof of Theorem \ref{teo:char-exact-fact-fusion-ring}, which shows that every exact factorization of fusion rings is given by a bicrossed product. Finally, Subsection \ref{subsection:applications-exact-fact} shows some direct applications of this characterization to the general structure of exact factorizations of fusion categories.
Section \ref{section:exact-factorization-fusion-categories} provides categorifications of the ideas presented in Section \ref{section:exact-fact-fusion-rings}, we define the notion of a matched pair of fusion categories (Definition \ref{def:mathced-pair-categories}) and the bicrossed product of fusion categories, see Theorem \ref{teo:bowtie-monoidal} and Corollary  \ref{teo:bowtie}. We also discuss how pivotal and spherical structures of fusion categories are induced in the bicrossed product, see Propositions \ref{prop:pivotal-bicrossed-product} and \ref{prop:spherical-bicrossed-product}.
Section \ref{section:examples-of-fusion-cat-from-bicrossed-product} is about applications of the construction from Section \ref{section:exact-factorization-fusion-categories} to generate examples of fusion categories, the fusion rings of exact factorizations between a Tambara-Yamagami fusion category and a pointed fusion category, and between two Tambara-Yamagami fusion categories are fully characterized (Theorems \ref{teo:matched-pair-rings-TY-vecK} and \ref{teo:matched-pair-rings-TY-TY}). Their categorifications as a bicrossed product of fusion categories are characterized in Theorems \ref{teo:mp-fusion-TambaraY-pointed} and \ref{teo:mp-fusion-TY-TY}. Explicit fusion rules and associators are computed for these examples.

\section*{Acknowledgements}
The authors thank P. Etingof, D. Nikshych, and V. Ostrik for fruitful discussion regarding fusion categories with the same fusion rules as the Deligne product and for kindly sharing their draft \cite{ENO-no-published} with us. The authors are also grateful to A. Schopieray for useful comments on the first draft of this article. 

The research of JP was partially supported by the NSF grant DMS-2146392. HMPP and JP were partially supported by Simons Foundation Award 889000 as part of the Simons Collaboration on Global Categorical Symmetries. HMPP and JP would like to thank the hospitality and excellent working conditions at the Department of Mathematics at the Universit\"at Hamburg, where JP has carried out part of this research as an Experienced Fellow of the Alexander von Humboldt Foundation and HMPP as a visitor. HMPP and JP carried out part of this work at different institutes including the Aspen Center for Physics (supported by the NSF grant PHY-1607611), the Centre de Recherches Math\'ematiques 
during the thematic month on Quantum symmetries (HMPP's visit was supported by NSF grant DMS-2228888 while JP was supported by funding from the Simons Foundation and CRM through the Simons-CRM scholar-in-residence program), the Banff International Research Station (BIRS) through a Focused research group on Non-Classical Constructions in Tensor Categories and Conformal Field Theory. Part of this work was done while MM was visiting the Department of Mathematics at Indiana University in Bloomington; she is grateful for their warm hospitality. The author MM was partially supported by the Association for Women in Mathematics (AWM), through the AWM Mathematical Endeavors Revitalization Program (AWM-MERP) grant.

\section{Preliminaries}\label{section:Preliminaries}
In this article, we denote by $\N$ the natural numbers, and $\N_0\coloneqq \N \cup \{0\}$. If $\ell \leq n \in\N_0$, we set $\I_{\ell, n}=\{\ell, \ell +1,\dots,n\}$, $\I_n = \I_{0, n}$. 
We will denote by $\ku$ an algebraically closed field of characteristic $0$. We denote by $C_n$ the cyclic group of order $n$.

\subsection{Monoidal and fusion categories}
In this subsection, we recall some definitions and fix notation. We refer the reader to \cite{EGNO-book} for a more detailed exposition. 

A \emph{monoidal category} $\mathcal{C}$ is a category equipped with a bifunctor $\ot\colon \cC \times \cC \to \cC$ together with an unit object $\uno$, and natural isomorphisms $\ell_X\colon \uno \ot X\to X$, $r_X\colon X\ot 1\to X$, $\alpha_{X,Y,Z}\colon (X\ot Y)\ot Z\to X\ot (Y\ot Z)$ satisfying pentagon and triangle axioms for all objects $X, Y, Z$ of $\cC$, see \cite[Definition 2.1.1]{EGNO-book} for details.

A \emph{monoidal functor} between two monoidal categories $\cA$ and $\cC$ is a pair $(\bF,\bF^2)$, where $\bF\colon \cA \to \cC$ is a functor and $\bF^2_{X,Y}\colon \bF(X)\ot \bF(Y) \to \bF(X\ot Y)$ is a natural isomorphism in both variables such that $\bF(\uno)\simeq \uno$ and $\bF^2$ satisfies a commutative diagram, see \cite[Definition 2.4.1]{EGNO-book}. A \emph{monoidal natural transformation} between two monoidal functors $(\bF,\bF^2)$ and $(\bF',(\bF')^2)$ is a natural transformation $\theta_X\colon \bF(X)\to \bF'(X)$ such that $(\bF')^2_{X, Y}\theta_{X\ot Y} = (\theta_X \ot \theta_Y){\bF^2_{X, Y}}$, for all $X$, $Y\in \cA$.

In this article, the categories with which we work are fusion, which means they are abelian, $\ku$ -linear, monoidal, rigid, finite, and semisimple, and the unit is simple, see \cite[Definition 4.1.1]{EGNO-book}. The functors are also required to preserve the relevant structure, so they are additive, $\ku$-linear and exact. The functors are monoidal only when indicated.

Let $\cC$ be a fusion category. We denote by $\Irr(\cC)$ the set of isomorphism classes of simple objects of $\cC$. 
Since the category $\cC$ is semisimple, the tensor product of two objects $X, Y\in\cC$ can be decomposed as
\begin{align}\label{eq:tensor_product} 
X\ot Y \simeq  \bigoplus_{Z\in\Irr(\mathcal{C})} N_{X,Y}^{Z} Z,
\end{align}
where $N_{X,Y}^{Z} = \dim_{\ku} \Hom_{\cC}(Z,X\ot Y)$ is the multiplicity of the object $Z$ in the tensor product of $X$ and $Y$. 

\subsection{Fusion rings}
One fundamental invariant in understanding a fusion category is their Grothendieck ring. This is a special kind of ring, whose properties have been axiomatized by the notion of \emph{fusion ring}. In this subsection, we recall some definitions about this topic from \cite{ostrik}.
\begin{definition}\label{def:fusion-ring}
	A \emph{fusion ring} is a pair $(\mathsf{R},\mathsf{B}(\mathsf{R}))$, where $\mathsf{R}$ is a ring with a fixed $\Z$-basis $\mathsf{B}(\mathsf{R})=\{b_0,\dots,b_n\}$, such that:
	\begin{enumerate}[leftmargin=*,label=\rm{(\roman*)}]
		\item\label{def:fusion-ring-1} the \emph{structure coefficients} of the multiplication are non-negative integers, that is, $b_i b_j = \sum_{k=0}^n N_{i,j}^k b_k$ for some $N_{i,j}^k\in\N_0$,
		\item\label{def:fusion-ring-2} $b_0 = 1$ is the unit of the ring,
	
 \item\label{def:fusion-ring-3} there is an involution $^*\colon \I_n\to \I_n$ such that the structure coefficients satisfy
    \begin{equation*}
        N_{i,j}^0 = \begin{dcases*}
1 & if  $j = i^*$ ,\\
0 & otherwise,
        \end{dcases*}
    \end{equation*}
 and the induced map given by $x = \sum_{i\in\I_n} a_i b_i \mapsto x^* = \sum_{i\in\I_n} a_i b_{i^*}$  is an anti-automorphisms of rings. 
	\end{enumerate} 
\end{definition}

    Let $(\mathsf{R},\mathsf{B}(\mathsf{R}))$ be a fusion ring with the basis $\mathsf{B}(\mathsf{R})=\{b_0,\dots,b_n\}$. For $i,j,k\in\I_n$, we denote by $N_{b_i,b_j}^{b_k} \coloneqq N_{i,j}^k$.     
    \begin{example}
The \emph{Grothendieck ring} of a fusion category $\mathcal{C}$ is the free $\Z$-module generated by $\Irr(\cC)$ and product induced from the tensor product
\begin{align*}
    C C' =  \sum_{C''\in\Irr(\cC)}N_{C,C'}^{C''} C'', && C,C'\in\Irr(\cC),
\end{align*} 
The involution is induced by the duality of the category. This is a fusion ring and we will denote it as $K_0(\mathcal{C)}$.   
    \end{example}
    
    A \emph{fusion subring} $(\mathsf{S},\mathsf{B}(\mathsf{S}))$ is a fusion ring such that $\mathsf{S}$ is a subring of $\mathsf{R}$ and $\mathsf{B}(\mathsf{S})\subseteq \mathsf{B}(\mathsf{R})$. A \emph{based (left) module} for a fusion ring $(\mathsf{R},\mathsf{B}(\mathsf{R}))$ is a $\mathsf{R}$-module $M$, free as a $\Z$-module, endowed with a fixed basis $\{v_j\}_{j\in\cJ}$ such that
	\begin{align*}
		b_i v_j = \sum_{k\in\cJ} d_{i,j}^k v_k,
	\end{align*}
	for some $d_{i,j}^k\in\N_0$, and $d_{i,j}^k = d_{i^*,k}^j$ for all $j,k\in\cJ$ and $i\in\I_n$.

The following important result proven originally in \cite{ostrik} about based modules will be used repeatedly in this article.
\begin{lemma}\label{lemma:}\cite[Lemma 1]{ostrik} An indecomposable based module over a fusion ring is irreducible. In other words, based modules have the complete reducibility property, that is, every based module is a direct sum of irreducible based modules.
\end{lemma}

\subsubsection{Gradings of fusion rings}
Let $(\mathsf{R},\mathsf{B}(\mathsf{R}))$ be a fusion ring and $G$ a group. We say that $\mathsf{R}$ is \emph{graded} by $G$ if we have decompositions
\begin{align*}
 \mathsf{R} = \sum_{g\in G} \mathsf{R}_g, && \text{ and } && \mathsf{B}(\mathsf{R})= \bigsqcup_{g\in G} \mathsf{B}(\mathsf{R})_g,   
\end{align*}
where $\mathsf{R}_g$ are $\mathbb{Z}$-submodules generated by $\mathsf{B}(\mathsf{R})_g$ satifying $ \mathsf{R}_g \mathsf{R}_{g'} \subseteq \mathsf{R}_{g g'}$ and $\mathsf{R}_g^* = \mathsf{R}_{g^{-1}}$. Such $G$-grading is called \emph{faithful} if $\mathsf{R}_g \neq 0$ for all $g\in G$. A $G$-grading on $\mathsf{R}$ induces a map $\operatorname{deg}:\mathsf{B}(\mathsf{R})\to G$, $b\mapsto \degf{b}$ such that if $xy$ has $z$ in their $\mathsf{B}(\mathsf{R})$-decomposition one has
$\degf{z} = \degf{x} \degf{y}$.

In \cite{GN}, the authors proved that every fusion ring $\mathsf{R}$ is faithfully graded by a group $U(\mathsf{R})$ called the universal grading group. We first need the following definition to describe this grading in detail.
\begin{definition}
	The fusion subring of $\mathsf{R}$ generated by the elements $x x^*$, for $x\in \mathsf{R}$, is called the \emph{adjoint fusion subring} and it is denoted $\mathsf{R}_{\ad}$ with basis $\mathsf{B}(\mathsf{R}_{\ad})\subseteq \mathsf{B}(\mathsf{R})$.
\end{definition}
\begin{remark}\label{remark:I(1)-genera}
	The object $I(1) = \sum_{i=1}^n b_i b_i^*$ generates $\mathsf{R}_{\ad}$ in the sense that every element of $\mathsf{B}(\mathsf{R}_{\ad})$ is a summand of $I(1)^n$ for some $n\in\N$.
\end{remark}

A fusion ring $(\mathsf{R},\mathsf{B}(\mathsf{R}))$  can be regarded as a (left) based module over $(\mathsf{R}_{\ad},\mathsf{B}(\mathsf{R}_{\ad}))$. This induces a decomposition $\mathsf{R} = \sum_{j\in\cJ} \mathsf{R}_j$ where the $\mathsf{R}_j$ are indecomposable based $\mathsf{R}_{\ad}$-modules by \cite[Lemma 1]{ostrik}. The index set $\cJ$ has a group structure described in the next proposition. 
\begin{prop}\cite[Theorem 3.5]{GN}
	There is a canonical group structure in the index set $\cJ$ with the multiplication defined by the following property for $a,b,c\in\cJ$:
	\begin{align*}
		a b &= c &\text{if and only if }&& x y \in \mathsf{R}_c, &&\forall x\in \mathsf{R}_a, y\in \mathsf{R}_b.
	\end{align*}
	The identity of $\cJ$ is an element $1\in\cJ$ such that $\mathsf{R}_1 = \mathsf{R}_{\ad}$. The inverse of an element $j\in\cJ$ is an element $j^{-1}$ such that $\mathsf{R}_{j^{-1}} = \mathsf{R}_j^*$.
\end{prop}

The index set $\cJ$ with this group structure is called the \emph{universal grading group} of $\mathsf{R}$ and is denoted by $U(\mathsf{R})$. The decomposition $\mathsf{R} = \sum_{j\in\cJ} \mathsf{R}_j$ endows $\mathsf{R}$ with a $U(\mathsf{R})$-grading. The following result shows why the adjective universal is appropriate for this group.

\begin{prop}\cite[Corollary 3.7]{GN} Every fusion ring $\mathsf{R}$ has a canonical faithful grading by the group $U(\mathsf{R})$. Any other faithful grading of $\mathsf{R}$ by a group $G$ is determined by a surjective group homomorphism $\pi\colon U(\mathsf{R})\to G$.
\end{prop}

\subsection{Gradings of fusion categories}
 Let $\cC$ be a fusion category and $G$ a group. We say that $\cC$ is graded by $G$ if $\cC = \bigoplus_{g\in G} \cC_g$, where $\cC_g$ are full abelian subcategories such that $\otimes: \cC_g\times \cC_h\to \cC_{gh}$ and $\cC_g^* = \cC_{g^{-1}}$. Such $G$-grading is called faithful if $\cC_g \neq 0$ for all $g\in G$. As in fusion rings, such $G$-grading on $\cC$ induces a map $|-|:\Irr(\cC)\to G$ such that if $X \otimes Y$ contains $Z$ one has $\degf{Z} = \degf{X} \degf{Y}$.


We define the \emph{adjoint subcategory} $\cC_{\ad}$ as the tensor subcategory generated by the elements $X\ot X^*$ for $X\in\cC$. It is clear that $K_0(\cC_{\ad}) = K_0(\cC)_{\ad}$.

There is a one-to-one correspondence between (faithful) gradings in $\cC$ and gradings in the fusion ring $K_0(\mathcal{C})$. 
Therefore every fusion category admits a faithful grading by $U(\mathcal{C})\coloneqq U(K_0(\mathcal{C}))$; this is the \emph{universal grading group} of $\mathcal{C}$.

\subsection{Categorical group actions}\label{subsec: cat action} Let $G$ be a finite group and $\cC$ be a fusion category. A  left (respectively right) categorical action by $\ku$-linear autoequivalences of $G$ on $\cC$ is a family of $\ku$-linear functors $\mathbf{L}_g: \cC\to \cC$ (respectively $\mathbf{R}_g: \cC\to \cC$), for each $g\in G$, where the action on an object $X$ and a morphism $f$ is denoted by
\begin{align*}
    \mathbf{L}_g(X) = g\trid X, && \mathbf{L}_g(f) = g\trid f, && 
 \mathbf{R}_g(X) = X\fiz g, &&  \mathbf{R}_g(f) = f\fiz g,
\end{align*} and natural isomorphisms $\mathbf{L}^2_{g,h}: g\trid(h\trid -)\to gh\trid -$ (respectively $\mathbf{R}^2_{g,h}: (-\fiz g)\fiz h\to -\fiz gh$), $\mathbf{L}^0: \id_{\cC}\to e\trid -$ ($\mathbf{R}^0: \id_{\cC}\to  -\fiz e$) such that the following diagrams commute:

\begin{minipage}{0.4\textwidth}
{\tiny \begin{equation}\label{action-eq1}
   \begin{tikzcd}
        g\trid (h\trid (k\trid X)) \ar[rr, "g\trid (\mathbf{L}^2_{h,k})_X"] \ar[dd, swap, "(\mathbf{L}^2_{g,h})_{k\trid X}"]&& g\trid (hk\trid X) \ar[dd, "(\mathbf{L}^2_{g,hk})_{X}"]\\
        && \\
        gh\trid (k\trid X) \ar[rr, swap, "(\mathbf{L}^2_{gh,k})_{X}"] && ghk\trid X
    \end{tikzcd}
\end{equation}}
\end{minipage}
\begin{minipage}{0.4\textwidth}
{\tiny\begin{equation*}
     \begin{tikzcd}
        ((X\fiz g)\fiz h)\fiz k \ar[rr, "((\mathbf{R}^2_{g,h})_X)\fiz k"] \ar[dd, swap, "(\mathbf{R}^2_{h,k})_{X\fiz g}"]&&  (X\fiz gh)\fiz k \ar[dd, "(\mathbf{R}^2_{gh,k})_{X}"]\\
        && \\
        (X\fiz g)\fiz hk \ar[rr, swap, "(\mathbf{R}^2_{g,hk})_{X}"] &&  X\fiz ghk
    \end{tikzcd}
\end{equation*}}
\end{minipage}

\begin{minipage}{0.4\textwidth}
{\tiny \begin{equation}\label{action-eq2}
     \begin{tikzcd}
        g\trid X \ar[rr, "\mathbf{L}^0_{g\trid X}"] \ar[dd, swap, "g\trid \mathbf{L}^0_{X}"] 
        \ar[rrdd, dashed, "\id_{g\trid X}"]&& e\trid (g\trid X) \ar[dd, "(\mathbf{L}^2_{e,g})_{X}"]\\
        && \\
        g\trid (e\trid X) \ar[rr, swap, "(\mathbf{L}^2_{g,e})_{X}"] && g\trid X
    \end{tikzcd}
\end{equation}}
\end{minipage}
\begin{minipage}{0.4\textwidth}
{\tiny \begin{equation*}
     \begin{tikzcd}
        X \fiz g \ar[rr, "\mathbf{R}^0_{X\fiz g}"] \ar[dd, swap, "\mathbf{R}^0_{X} \fiz g"] 
        \ar[rrdd, dashed, "\id_{X\fiz g}"]&&  (X\fiz g)\fiz e \ar[dd, "(\mathbf{R}^2_{g,e})_{X}"]\\
        && \\
        (X\fiz e)\fiz g \ar[rr, swap, "(\mathbf{R}^2_{e,g})_{X}"] && X\fiz g
    \end{tikzcd}
\end{equation*}}
\end{minipage}

A left (respectively right) categorical action by tensor autoequivalences of $G$ on $\cC$ occurs when additionally each of the functors $\mathbf{L}_g$ (respectively $\mathbf{R}_g$) is a monoidal functor.  In other words,  for every $g\in G$, we have natural isomorphisms $\zeta^g_{X, Y}: g\trid X\otimes g\trid Y\to g\trid (X\otimes Y)$ (respectively $\vartheta^g_{X, Y}: X\fiz g\otimes  Y\fiz g\to (X\otimes Y)\fiz g$) and  $\zeta^g_0:  \uno \to g\trid \uno$ (respectively $\vartheta^g_0:  \uno \to \uno\fiz g$) satisfying associativity and unit constraints, see \cite[Definition 2.4.1]{EGNO-book}.

\subsection{Pointed fusion rings and fusion categories}\label{subsection:pointed-definition}
A \emph{pointed fusion ring} is a fusion ring $(\mathsf{R}, \mathsf{B})$ such that the Frobenius-Perron dimension of the elements in the basis is one, that is $\fpdim b = 1$ for all $b\in\mathsf{B}$, see \cite[Section 3.2]{FK} for the definition of Frobenius-Perron dimension. A \emph{pointed fusion category} is a fusion category whose fusion ring is a pointed fusion ring. All pointed fusion rings are completely determined by a finite group $G$, and the corresponding fusion ring is $\mathsf{R}(G) = \mathbb{Z} G$ the integral group ring, with basis $\mathsf{B}(\mathsf{R}(G)) = G$. Pointed fusion categories are parametrized by a finite group $G$ and a normalized $3$-cocycle $\omega\in H^3(G,\ku^\times)$ that determines the associativity, and are denoted as $\vect_G^\omega$. This category consists of finite dimensional $G$-graded vector spaces, simple objects are parametrized by elements $g\in G$ and consist in one-dimensional vector spaces concentrated in degree $g$. The associativity natural isomorphism in the simples is $\alpha_{g,h,k} = \omega(g,h,k) \id_{ghk},$ for $g,h,k\in G$.

Given a fusion ring $\mathsf{R}$ (respectively fusion category $\cC$), we denote the maximal pointed fusion subring (respectively subcategory) by $\mathsf{R}_{\pt}$ (respectively $\cC_{\pt}$).

\subsection{Tambara-Yamagami fusion rings and fusion categories}\label{subsection:Tambara-Yamagami-definition}
In this subsection, we introduce a nice family of examples of fusion rings and fusion categories studied for the first time in \cite{TY}. It is also one of the few families where the pentagon equation for the associativity were solved explicitly. Let $\Gamma$ be a group, the Tambara-Yamagami fusion ring $\mathsf{TY}(\Gamma)$ is the free $\mathbb{Z}$-module with basis $\Gamma\cup \{X\}$ and multiplication given by
\begin{align*}
    g X = X g = X, \quad g\in\Gamma, && X^2 = \sum_{g\in \Gamma} g,  
\end{align*}
and the multiplication between group elements is the group multiplication. This ring admits a categorification only when $\Gamma$ is abelian, the extra data needed is a non-degenerate symmetric bicharacter  $\chi\colon \Gamma\times \Gamma\to \ku^\times$ and $\tau\in\ku$ such that $\tau^2 = 1/|\Gamma|$. The corresponding fusion category is denoted as $\mathcal{TY}(\Gamma,\chi,\tau)$. The natural isomorphisms $\ell$ and $r$ are identities. The associativity natural isomorphisms are given by
{\small\begin{align*}
    \alpha_{g,h,k} &= \id_{ghk}, & \alpha_{g,h,X} &= \alpha_{X,g,h} = \id_X, & \alpha_{g,X,h} &= \chi(g,h) \id_X, &
    \alpha_{g,X,X} &= \alpha_{X,X,g} = \id_{X\ot X},
\end{align*}
\begin{align*}
    \alpha_{X,g,X} = \bigoplus_{g'\in \Gamma} \chi(g,g') \id_{g'}, &&
\alpha_{X,X,X} = \left(\frac{\tau}{\chi(g,h)}\right)_{g,h} \colon \oplus_{g\in \Gamma} X\to \oplus_{g\in \Gamma} X, && g,h,k\in \Gamma.
\end{align*}}

\section{Exact factorization of fusion rings and fusion categories}\label{section:exact-fact-fusion-rings}

In this section, we recall the definition and properties of exact factorization of fusion categories given by \cite{G-exact-factorization}. We also propose a similar notion for fusion rings and a definition of matched pairs of fusion rings and their bicrossed product. Finally, we show that every exact factorization of fusion rings can be realized as a bicrossed product. We start by recalling some definitions related to exact factorizations of groups.

\subsection{Exact factorization of groups}\label{subsection:exact-fact-groups}
An \emph{exact factorization} of a group $G$ is a pair $(H, K)$ of subgroups of $G$ such that $G = H K$ and $H \cap K = \{e\}$ or, in other words, the restriction of the multiplication $H \times K \to G$
is a bijection \cite{kac,Mackey,Takeuchi-matched}. In this case, we write $G = H \bullet K$.
Exact factorizations can be described in terms of the notion of a matched pair.
A \emph{matched pair} of groups is a collection $(H, K, \tridb, \fizb)$ where $H$ and $K$ are groups,
$\tridb$ and $\fizb$ are left and right actions 
$\xymatrix{K & K \times H \ar  @{->}[r]^{\quad\tridb}\ar  @{->}[l]_{\fizb\quad } & H }$
and the following  conditions hold:
\begin{align}\label{eq:mp-group1}
	(kt) \fizb h &= \left(k \fizb (t \tridb h)\right)(t \fizb h),
	\\ \label{eq:mp-group2}
	k \tridb (h g) &=\left(k \tridb h\right)\left((k \fizb h) \tridb g \right),
\end{align}
for all $k, t \in K$, $h, g \in H$. Then $H \bowtie K := H \times K$ with the multiplication
\begin{align*}
	(h, k)(g , t) &= \left(h (k \tridb g), (k \fizb g) t\right),
	& k, t &\in K, \ h, g \in H,
\end{align*}
is a group which is an exact factorization. Moreover, any exact factorization of $H$ and $K$ is like this. The group $H\bowtie K$ is known as the \emph{Zappa–Sz\'ep product} or \emph{bicrossed product} of the groups $H$ and $K$.

\begin{remark}
	If the action $\fizb\colon K\times H\to K$ is trivial, meaning $k\fizb h = k$ for every $h\in H$, $k\in K$, then $\tridb\colon K\times H\to H$ is an action of $K$ in $H$ by group automorphisms. In that case, $H\bowtie K\simeq H\ltimes K$ is the semidirect product of $H$ and $K$ respect to $\tridb$. A similar result holds if $\tridb$ is trivial instead. In the general case, neither action $\tridb$ or $\fizb$ is by group automorphisms.
\end{remark}

\begin{remark}\label{remark:inversas-matched-pair-grupos}
In a bicrossed product of groups $H\bowtie K$, the inverses are given by $(k\tridb h)^{-1} = (k\fizb h)\tridb h^{-1}$, $(k\fizb h)^{-1} = k^{-1}\fizb (k\tridb h)$. 
\end{remark}

\subsection{Exact factorization of fusion categories} The following definition was first proposed in \cite{G-exact-factorization}. 
\begin{definition}\label{def-exact-fact}
Let $\cB$ be a fusion category and let $\cA$, $\cC$ be fusion subcategories of $\cB$. The category $\cB$ is an \textit{exact factorization} of $\cA$ and $\cC$, denoted by $\cB= \cA \bullet \cC$, if any of the next 
equivalent conditions are satisfied:
\begin{enumerate}[leftmargin=*,label=\rm{(\roman*)}]
    \item\label{def-exact-fact-1} $\cB$ is the full abelian subcategory spanned by direct summands of $X\otimes Y$, $X\in \cA$, $Y\in\cC$,  and $\cA\cap \cC=\vect$,
    \item\label{def-exact-fact-2} $\cA\cap\cC = \vect$ and $\fpdim(\cB) = \fpdim(\cA)\fpdim(\cC)$,
    \item\label{def-exact-fact-3} Every simple object of $\cB$ can be written uniquely as $A\ot C$ with $A$ and $C$ simple objects in $\cA$ and $\cC$, respectively.
\end{enumerate}

\end{definition}

\begin{remark}\label{remark-tensor}
The definition of exact factorization of fusion categories is symmetric in $\cA$ and $\cC$. It follows from Definition \ref{def-exact-fact} \ref{def-exact-fact-2} that if $\cB=\cA\bullet\cC$ then $\cB=\cC\bullet\cA$. Hence for any $B\in \Irr(\cB)$ there exist unique $A,\widetilde{A}\in\Irr(\cA)$ and $C,\widetilde{C}\in\Irr(\cC)$ such that $B = A\ot C = \widetilde{C}\ot\widetilde{A}$.
\end{remark}

 The following example shows why this definition is a generalization of the definition of exact factorization for groups.

\begin{example}\cite[Example 3.6]{G-exact-factorization}\label{example:pointed}
Exact factorizations $\mathcal{B}=\textrm{vec}_{G_1}^{\omega_1}\bullet \textrm{vec}_{G_2}^{\omega_2}$ of two pointed fusion categories are classified by groups $G$ with an exact factorization $G=G_1\bullet G_2$ and $\omega\in H^3(G, \Bbbk^\times)$ such that the restrictions to $G_1$ and $G_2$ are $\omega_1$ and $\omega_2$, respectively. And it follows that $\cB\simeq \vect_G^\omega$.
\end{example}

\begin{remark}\label{rmk:pointed}
    If a fusion category $\cB$ admits an exact factorization $\cB=\cA\bullet\cC$ then there is an induced exact factorization of its pointed subcategory $\cB_{\pt}=\cA_{\pt}\bullet\cC_{\pt}$. This follows from the fact that the Frobenius-Perron dimensions of simple objects in $\cB$ are the product of the Frobenius-Perron dimensions of simple objects in $\cA$ and $\cC$
\end{remark}

 Also, it is worth mentioning that any fusion category $\mathcal{B}$ admits the \emph{trivial} exact factorization $\mathcal{B}=\mathcal{B}\bullet \vect$.

 \subsection{Exact factorization of fusion rings}
 We propose the following definition for the exact factorization of fusion rings.

\begin{definition}
    Let $(\mathsf{R},\mathsf{B}(\mathsf{R}))$ be a fusion ring and $(\mathsf{A},\mathsf{B}(\mathsf{A}))$, $(\mathsf{C},\mathsf{B}(\mathsf{C}))$ be fusion subrings. We say that $(\mathsf{R},\mathsf{B}(\mathsf{R}))$ is an \emph{exact factorization} of $(\mathsf{A},\mathsf{B}(\mathsf{A}))$ and $(\mathsf{C},\mathsf{B}(\mathsf{C}))$ if every element $b$ of $\mathsf{B}(\mathsf{R})$ can be written in a unique way as $b = ac$ for $a\in \mathsf{B}(\mathsf{A})$ and $c\in \mathsf{B}(\mathsf{C})$. We then write $\mathsf{R} = \mathsf{A}\bullet \mathsf{C}$.
\end{definition}

\begin{remark}
    If $\mathcal{B} = \mathcal{A}\bullet \mathcal{C}$ is an exact factorization of fusion categories, then there is an associated exact factorization of its Grothendieck ring, that is, $K_0(\mathcal{B}) = K_0(\mathcal{A})\bullet K_0(\mathcal{C})$ is an exact factorization of fusion rings.
\end{remark}

The next proposition is analogous to Remark \ref{remark-tensor} for exact factorizations of fusion rings.

\begin{prop}\label{prop:exact-fact-is-symmetric}
    Let $(\mathsf{R},\mathsf{B}(\mathsf{R}))$ be a fusion ring and $(\mathsf{A},\mathsf{B}(\mathsf{A}))$, $(\mathsf{C},\mathsf{B}(\mathsf{C}))$ be fusion subrings. If $\mathsf{R} = \mathsf{A} \bullet \mathsf{C}$ then $\mathsf{R}=\mathsf{C}\bullet \mathsf{A}$.
\end{prop}
\begin{proof}
    We first show that the elements $c a$ are in $\mathsf{B}(\mathsf{R})$ for $c\in \mathsf{B}(\mathsf{C})$ and $a\in \mathsf{B}(\mathsf{A})$. Since $a^* c^*$ is in $\mathsf{B}(\mathsf{R})$, then $c a =(a^* c^*)^* \in \mathsf{B}(\mathsf{R})$. Also, all elements $c a$ are different for distinct $a$ and $c$ since their duals are different elements by hypothesis.
\end{proof}

\subsection{Matched pairs and bicrossed products of fusion rings}\label{subsection:bicrossed-product-fusion-rings}
We propose the following definition for a matched pair of fusion rings generalizing the one for groups.
\begin{definition}\label{def:mathced-pair-fusion-rings}
    Let $(\mathsf{A},\mathsf{B}(\mathsf{A}))$, $(\mathsf{C},\mathsf{B}(\mathsf{C}))$ be fusion rings with faithful gradings given by
    \begin{align*}
        \mathsf{A} = \sum_{h\in H} \mathsf{A}_h, && \mathsf{B}(\mathsf{A}) = \bigsqcup_{h\in H} \mathsf{B}(\mathsf{A})_h, &&   \mathsf{C} = \sum_{k\in K} \mathsf{C}_k, && \mathsf{B}(\mathsf{C}) = \bigsqcup_{k\in K} \mathsf{B}(\mathsf{C})_k, 
    \end{align*}
    for groups $H$ and $K$ respectively. A \emph{matched pair of fusion rings} between $\mathsf{A}$ and $\mathsf{C}$ consist of a $8-$tuple $(\mathsf A, \mathsf C, H, K, \tridb, \fizb, \trid, \fiz)$, where
    \begin{enumerate}[leftmargin=*,label=\rm{(\roman*)}]
    \item the tuple $(H, K, \tridb, \fizb)$ is a matched pair between the groups $H$ and $K$,
    \item\label{def:mathced-pair-fusion-rings-2} a $\mathbb{Z}$-linear left action $\trid\colon K\times \mathsf{A} \to \mathsf{A}$ and a $\mathbb{Z}$-linear right action $\fiz\colon \mathsf{C}\times H \to \mathsf{C}$ such that $k \trid \mathsf{B}(\mathsf{
A})_h = \mathsf{B}(\mathsf{A})_{k\tridb h}$ and $\mathsf{B}(\mathsf{C})_k \fiz h= \mathsf{B}(\mathsf{C})_{k \fizb h}$,
    \item\label{def:mathced-pair-fusion-rings-3} for each $k\in K$ and $h\in H$,
    \begin{align*}
        & k\trid (aa') = (k\trid a) ((k\fizb\degf{a})\trid a'), & a,a'\in \mathsf{B} (\mathsf{A}),\\
        &(cc')\fiz h = (c\fiz (\degf{c'}\tridb h)) (c'\fiz h), & c,c'\in\mathsf{B} (\mathsf{C}),
    \end{align*}
 \item\label{def:mathced-pair-fusion-rings-4} for each $k\in K$ and $h\in H$, $k\trid 1=1$ and $1\fiz h=1.$
    \end{enumerate}
    \end{definition}

The following lemma provides an analogous to Remark \ref{remark:inversas-matched-pair-grupos}.
\begin{lemma}\label{lemma:dual-of-triangles}
Let $(\mathsf A, \mathsf C, H, K, \tridb, \fizb, \trid, \fiz)$ be a matched pair of fusion rings. Then the actions $\trid$ and $\fiz$ satisfy
\begin{align*}
    (k\trid a)^* = (k\fizb \degf{a})\trid a^*, && (c\fiz h)^* = c^*\fiz (\degf{c}\tridb h), && a\in \mathsf{B}(\mathsf{A}), c\in\mathsf{B}(\mathsf{C}), h\in H, k\in K.
\end{align*}
\end{lemma}
\begin{proof}
    Let $a\in \mathsf{B}(A)$ and $k\in K$. By Definition \ref{def-exact-fact} \ref{def-exact-fact-3}, we know that $a a^* = 1 + \sum_{a'\in\mathsf{B}(\mathsf A), a'\neq 1} N_{a,a^*}^{a'} a'$. Acting with $k$ on both sides, we get
    \begin{align*}
        k\trid (a a^*) = k\trid 1 + \sum_{a'\in\mathsf{B}(\mathsf A), a'\neq 1} N_{a,a^*}^{a'} (k\trid a') =  1 + \sum_{a'\in\mathsf{B}(\mathsf A), a'\neq 1} N_{a,a^*}^{a'} (k\trid a').
    \end{align*}
    On the other hand,
    \begin{align*}
        k\trid (a a^*)  =  (k\trid a) ((k\fizb\degf{a})\trid a^*) = \sum_{a'\in\mathsf{B(\mathsf A)}} N_{(k\trid a), (k\fizb\degf{a})\trid a^* }^{a'} a'.
    \end{align*}
    From both of this it follows that $N_{(k\trid a), (k\fizb\degf{a})\trid a^* }^{1} = 1 \neq 0$, hence by Definition \ref{def-exact-fact} \ref{def-exact-fact-3} $(k\fizb\degf{a})\trid a^* = (k\trid a)^*$. The proof that $(c\fiz h)^* = c^*\fiz (\degf{c}\tridb h)$ is analogous.
\end{proof}

\begin{definition}\label{def:bicrossed-fusion-rings}
    Given a matched pair of fusion rings $(\mathsf A, \mathsf C, H, K, \tridb, \fizb, \trid, \fiz)$, we define the 
\emph{bicrossed product} of $\mathsf{A}$ and $\mathsf{C}$, denoted as $\mathsf{A}\bowtie\mathsf{C}$, as the following $\mathbb{Z}$-ring:
\begin{enumerate}[leftmargin=*,label=\rm{(\roman*)}]
    \item $\mathsf{A}\bowtie \mathsf{C} = \mathsf{A} \otimes_\mathbb{Z} \mathsf{C}$ as a $\mathbb{Z}$-module, and the elements are denoted as $a\bowtie c \coloneqq a \ot c$ for $a\in\mathsf{A}$ and $c\in\mathsf{C}$,
    \item the fixed $\mathbb{Z}$-basis is $\mathsf{B}(\mathsf{A}\bowtie \mathsf{C}) = \{a \bowtie c \colon a\in \mathsf{B}(\mathsf{A}), c\in \mathsf{B}(\mathsf{C})\}$,
    \item the multiplication is given by
    \begin{align*}
        (a\bowtie c)(a'\bowtie c') = a (\degf{c}\trid a') \bowtie (c\fiz \degf{a'}) c',
    \end{align*}
    for all $a, a'\in \mathsf{B}(\mathsf{A}), c, c'\in \mathsf{B}(\mathsf{C})$ and extended $\mathbb{Z}$-linearly,
    \item the involution $^*\colon \mathsf{A}\bowtie \mathsf{C}\to \mathsf{A}\bowtie \mathsf{C}$ is given by
    \begin{align*}
        (a\bowtie c)^* = \degf{c}^{-1}\trid a^* \bowtie c^* \fiz \degf{a}^{-1},
    \end{align*}
    for all $a\in \mathsf{B}(\mathsf{A}), c\in \mathsf{B}(\mathsf{C})$ and extended $\mathbb{Z}$-linearly.
\end{enumerate}
\end{definition}
\begin{theorem}\label{thm:bicrossed-product-fusion-rings}     The pair $(\mathsf{A}\bowtie \mathsf{C}, \mathsf{B}(\mathsf{A}\bowtie \mathsf{C}))$ with $^*$ is a fusion ring. Furthermore, $(\mathsf{A},\mathsf{B}(\mathsf{A}))$ and $(\mathsf{C},\mathsf{B}(\mathsf{C}))$  are fusion subrings of $\mathsf{A}\bowtie \mathsf{C}$ and the bicrossproduct $\mathsf{A}\bowtie \mathsf{C}$ is an exact factorization of $\mathsf{A}$ and $\mathsf{C}$.
\end{theorem}
\begin{proof} It is not difficult to see that the multiplication defined in $\mathsf{A}\bowtie \mathsf{C}$ is associative by Definition \ref{def:mathced-pair-fusion-rings} \ref{def:mathced-pair-fusion-rings-3} and the structure coefficients are natural numbers given by the formula $N^{a\bowtie c}_{a_1\bowtie c_1, a_2\bowtie c_2}=N^a_{a_1, \degf{c_1}\trid a_2}N^c_{c_1\fiz \degf{a_2}, c_2}$. The unitality condition follows from Definition \ref{def:mathced-pair-fusion-rings} \ref{def:mathced-pair-fusion-rings-4}.

It follows from Lemma \ref{lemma:dual-of-triangles} and Remark \ref{remark:inversas-matched-pair-grupos} that the map $^*$ is indeed an involution. 

We need to check that $N^{1\bowtie 1}_{a_1\bowtie c_1, a_2\bowtie c_2} = 1$ if $a_2\bowtie c_2 = (a_1\bowtie c_1)^* =\degf{c_1}^{-1}\trid a_1^* \bowtie c_1^* \fiz \degf{a_1}^{-1}$ and it is equal $0$ otherwise. It follows that
\begin{align*}
N^{1\bowtie 1}_{a_1\bowtie c_1, a_2\bowtie c_2}=N^1_{a_1, \degf{c_1}\trid a_2}N^1_{c_1\fiz \degf{a_2}, c_2 } = \begin{dcases*}
    1 & if $\degf{c_1}\trid a_2 = a_1^*$ and $c_2= (c_1\fiz \degf{a_2})^*$,\\
    0 & otherwise.
\end{dcases*}
\end{align*}
Then $N^{1\bowtie 1}_{a_1\bowtie c_1, a_2\bowtie c_2} = 1$ when $a_2 =  \degf{c_1}^{-1}\trid a_1^* $ and $c_2 = (c_1\fiz \degf{a_2})^*$ and $0$ otherwise. By Lemma \ref{lemma:dual-of-triangles}, 
we have $(c_1\fiz \degf{a_2})^* = c_1^* \fiz (\degf{c_1}\tridb \degf{a_2})=c_1^* \fiz (\degf{c_1}\tridb (\degf{c_1}^{-1}\tridb \degf{a_1}^{-1})=c_1^* \fiz\degf{a_1}^{-1}$.
  To finish the proof notice that the maps $\iota_{\mathsf A}\colon \mathsf A \to \mathsf A\bowtie \mathsf C$ and $\iota_{\mathsf C}\colon \mathsf C \to \mathsf A\bowtie \mathsf C$ given by $a\mapsto a\bowtie 1$ and $c\mapsto 1\bowtie c$ are inclusions of fusion rings, and $\mathsf A\bowtie \mathsf C =\iota_{\mathsf A}(\mathsf A) \bullet \iota_{\mathsf C}(\mathsf C) \simeq \mathsf A \bullet \mathsf C$.
\end{proof}

\subsection{Characterization of exact factorizations of fusion rings}\label{subsection:characterization-fusion-rings}
In the previous subsection, we showed that we can construct an exact factorization given a matched pair of fusion rings. Now, we prove that all exact factorizations can be obtained in this way.
\begin{theorem}\label{teo:char-exact-fact-fusion-ring}
Let $\mathsf{R} = \mathsf{A}\bullet \mathsf{C}$ be an exact factorization of fusion rings. Then there exists a matched pair of fusion rings between $\mathsf{A}$ and $\mathsf{C}$ such that $\mathsf{R}\simeq \mathsf{A}\bowtie \mathsf{C} $.
\end{theorem}

The proof of this theorem will be accomplished through a series of lemmas and constructions. So for the remaining of this subsection, let us fix a fusion ring $(\mathsf{R},\mathsf{B}(\mathsf{R}))$, and fusion subrings  $(\mathsf{A},\mathsf{B}(\mathsf{A}))$ and $(\mathsf{C},\mathsf{B}(\mathsf{C}))$ such that $\mathsf{R} = \mathsf{A}\bullet \mathsf{C}$. By Proposition \ref{prop:exact-fact-is-symmetric}, we also have $\mathsf{R} = \mathsf{C}\bullet \mathsf{A}$. Therefore, given $b\in \mathsf{B}(\mathsf{R})$ there exist unique $a, a' \in \mathsf{B}(\mathsf{A})$ and $c, c' \in \mathsf{B}(\mathsf{C})$ such that $b = a'c' = c a$. Hence, for each $a \in \mathsf{B}(\mathsf{A})$ and $c \in \mathsf{B}(\mathsf{C})$, there are well defined functions $\ell_c \colon \mathsf{B}(\mathsf{A}) \to \mathsf{B}(\mathsf{A})$ and $r_a\colon \mathsf{B}(\mathsf{C}) \to \mathsf{B}(\mathsf{C})$ such that $r_a(c) = c'$ and $\ell_c(a) = a'$. In other words, for every $a\in \mathsf{B}(\mathsf{A})$ and $c\in \mathsf{B}(\mathsf{C})$, we have $ca = \ell_c(a) r_a(c)$. 


Our next goal is to define certain $\mathbb Z$-linear actions by the universal grading group $U(\mathsf{A})$ on $\mathsf{C}$ ($U(\mathsf{C})$ on $\mathsf{A}$, respectively). For this, we need a series of auxiliaries lemmas and omit some proofs that are straightforward calculations from the definitions. 

\begin{lemma}\label{lemma:accion-del-1}
The following statements hold:
\begin{enumerate}[leftmargin=*,label=\rm{(\roman*)}]
    \item\label{item1-funciones-identidad} $\ell_{1} = \id_{\mathsf{B}(\mathsf{A})}$ and $r_{1} = \id_{\mathsf{B}(\mathsf{C})}$,
\item\label{item2-funciones-identidad} $\ell_c(1) = 1$ and $r_a(1) = 1$, for all $a\in \mathsf{B}(\mathsf{A})$ and $c\in \mathsf{B}(\mathsf{C})$.
\end{enumerate}
\end{lemma}
\begin{lemma}\label{lemma:composicion-de-funciones-relacion-con-producto}
The following statements hold:
\begin{enumerate}[leftmargin=*,label=\rm{(\roman*)}]
    \item\label{item:composicion-de-funciones-relacion-con-producto-1}  For $c_1, c_2, c_3\in \mathsf{B}(\mathsf{C}) $ such that $N_{c_1,c_2}^{c_3}\neq 0$, it follows that $\ell_{c_3} = \ell_{c_1}\circ \ell_{c_2}$. 
    \item\label{item:composicion-de-funciones-relacion-con-producto-2} For $a_1, a_2, a_3\in \mathsf{B}(\mathsf{A})$ such that $N_{a_1,a_2}^{a_3}\neq 0$, it follows that $r_{a_3} = r_{a_2}\circ r_{a_1}$. 
\end{enumerate}
\end{lemma}
\begin{proof}
    We just prove \ref{item:composicion-de-funciones-relacion-con-producto-1} since \ref{item:composicion-de-funciones-relacion-con-producto-2} has an analogous proof. Let $a\in\mathsf{B}(\mathsf{A})$. Then
      {\small \begin{align*}
        c_1 c_2 a =  c_1 \ell_{c_2}(a) r_a(c_2) = (\ell_{c_1}\circ \ell_{c_2})(a) r_{\ell_{c_2}(a)}(c_1) r_a(c_2) =\sum_{c\in \mathsf{B}(\mathsf{C})} N_{r_{\ell_{c_2}(a)}(c_1),r_a(c_2)}^c (\ell_{c_1}\circ \ell_{c_2})(a) c.
    \end{align*}}
    On the other hand, $c_1 c_2 a = \sum_{c\in \mathsf{B}(\mathsf{C})} N_{c_1,c_2}^c c a = \sum_{c\in \mathsf{B}(\mathsf{C})} N_{c_1,c_2}^c \ell_c(a) r_a(c)$.
    Hence, we have the following equality in $\mathsf{R}$
    \begin{align}\label{eq:proof-composition-l-e}
        \sum_{c\in \mathsf{B}(\mathsf{C})} N_{c_1,c_2}^c \ell_c(a) r_a(c) = \sum_{c\in \mathsf{B}(\mathsf{C})} N_{r_{\ell_{c_2}(a)}(c_1),r_a(c_2)}^c (\ell_{c_1}\circ \ell_{c_2})(a) c. 
    \end{align}
     Since $N_{c_1,c_2}^{c_3}\neq 0$, then $\ell_{c_3}(a) r_a(c_3)$ on the left-hand side Equation \eqref{eq:proof-composition-l-e} corresponds to a basic element on the right-hand side. Thus
$\ell_{c_3}(a) r_a(c_3) = (\ell_{c_1}\circ \ell_{c_2})(a) \widetilde{c}$, for some $\widetilde{c}\in\mathsf{B}(\mathsf{C})$. By uniqueness of the factorization, we have $\ell_{c_3}(a) = (\ell_{c_1}\circ \ell_{c_2})(a)$.
\end{proof}
\begin{coro}\label{coro:T-A-es-invertible}
The functions $\ell_c$ and $r_a$ are bijective, for all $a\in\mathsf{B}(\mathsf{A})$ and $c\in\mathsf{B}(\mathsf{C})$. Moreover, $\ell_c^{-1} = \ell_{c^*}$ and $r_a^{-1} = r_{a^*}$.
\end{coro}
\begin{proof}
This follows from Lemma \ref{lemma:composicion-de-funciones-relacion-con-producto} and Lemma \ref{lemma:accion-del-1} since in a fusion ring $1$ is a summand of $x x^*$.
\end{proof}
\begin{prop}\label{prop:adjunta-actua-trivial}
The following statements hold:
\begin{enumerate}[leftmargin=*,label=\rm{(\roman*)}]
    \item\label{item:adjunta-actua-trivial-1} If $c\in\mathsf{B}(\mathsf{C}_{\ad})$ then $\ell_c = \id_{\mathsf{B}(\mathsf{A})}$.
    \item\label{item:adjunta-actua-trivial-2} If $a\in\mathsf{B}(\mathsf{A}_{\ad})$ then $r_a = \id_{\mathsf{B}(\mathsf{C})}$.
\end{enumerate}
\end{prop}
\begin{proof}
    We prove \ref{item:adjunta-actua-trivial-1} only, since the proof of \ref{item:adjunta-actua-trivial-2} is similar. We define recursively
    \begin{align*}
        &K_1 := \{c\in\mathsf{B}(\mathsf{C})\colon c\text{ is a summand of } xx^*\text{ for some }x\in\mathsf{B}(\mathsf{C})\},\\
        &K_n :=  \{c\in\mathsf{B}(\mathsf{C})\colon c\text{ is a summand of } xy\text{ for some }x\in K_1, y\in K_{n-1}\},
    \end{align*}
    for $n>1$. Then $K := \bigcup_{n\in\N} K_n = \mathsf{B}(\mathsf{C}_{\ad})$, as the elements of $K_n$ are exactly the elements of $\mathsf{B}(\mathsf{C})$ that are summands of $I(1)^n$, see Remark \ref{remark:I(1)-genera}. Hence, it is enough to prove that if $c\in K_n$ then $\ell_c = \id$, for every $n\in\N$ . We prove this by induction on $n$. If $n=1$ and $c\in K_1$, then there exist $x\in\mathsf{B}(\mathsf{C})$ such that $c$ is a summand of $x x^*$, that is, $N_{x,x^*}^c\neq 0$. Hence, by Lemma \ref{lemma:composicion-de-funciones-relacion-con-producto} and Corollary \ref{coro:T-A-es-invertible}, $\ell_c = \ell_x\circ \ell_{x^*} = \id$. Let $n>1$ and $c\in K_n$. There exists $x\in K_1, y\in K_{n-1}$ such that $c$ is a summand of $xy$, that is, $N_{x,y}^c\neq 0$. By the inductive hypothesis, $\ell_x = \ell_y= \id$. Therefore, by Lemma \ref{lemma:composicion-de-funciones-relacion-con-producto}, $\ell_c = \ell_x \circ \ell_y = \id$.
\end{proof}

The next proposition shows that the maps $\ell_c$ and $r_a$ only depend on the degrees of $c$ and $a$ in $U(\mathsf{C})$ and $U(\mathsf{A})$, respectively. 

\begin{prop}\label{prop:mismo-grado-universal-actua-igual}
 Consider the universal grading of the fusion rings $\mathsf{A} = \sum_{h\in U(\mathsf{A})} \mathsf{A}_h$, $\mathsf{B}(\mathsf{A}) = \bigsqcup_{h\in U(\mathsf{A})} \mathsf{B}(\mathsf{A})_h$, and $\mathsf{C} = \sum_{k\in U(\mathsf{C})} \mathsf{C}_k$, $\mathsf{B}(\mathsf{C}) = \bigsqcup_{k\in U(\mathsf{C})} \mathsf{B}(\mathsf{C})_k$. 
 \begin{enumerate}[leftmargin=*,label=\rm{(\roman*)}]
     \item\label{item:mismo-grado-universal-actua-igual-1}  If $c,c'\in \mathsf{B}(\mathsf{C})_k$ then $\ell_c = \ell_{c'}$, for every $k\in U(\mathsf{C})$.
     \item\label{item:mismo-grado-universal-actua-igual-2}  If $a,a'\in \mathsf{B}(\mathsf{A})_h$ then $r_a = r_{a'}$, for every $h\in U(\mathsf{A})$.
 \end{enumerate}
 \end{prop}
 \begin{proof}
      We only prove \ref{item:mismo-grado-universal-actua-igual-1} as the proof of \ref{item:mismo-grado-universal-actua-igual-2} is similar. Let fix $k\in U(\mathsf{C})$ and $c\in\mathsf{B}(\mathsf{C})_k$. Then $\mathsf{C}_k$ is an indecomposable $\mathsf{C}_{\ad}$-module by definition of the universal grading. Now we define recursively:
    \begin{align*}
        &L_1 := \{x\in\mathsf{B}(\mathsf{C})_k\colon x\text{ is a summand of } y c\text{ for some }y\in\mathsf{B}(\mathsf{C}_{\ad})\},\\
        &L_n :=  \{x\in\mathsf{B}(\mathsf{C})_k\colon x\text{ is a summand of } yz\text{ for some }y\in\mathsf{C}_{\ad},  z\in L_{n-1}\}.
    \end{align*}
    Consider $L:=\cup_{n\in\N} L_n$, and the based $\mathsf{C}_{\ad}$-submodule $E:=\{\sum_i n_i x_i\colon n_i\in\N_0, x_i\in L\}$ of $\mathsf{C}_k$. Then $E= \mathsf{C}_k$ and $\mathsf{B}(\mathsf{C})_k = L$ since based modules have the complete reducibility property \cite[Lemma 1]{ostrik} and $\mathsf{C}_k$ is indecomposable. 
    
    We will prove by induction on $n$ that $\ell_x = \ell_c$, for all $x\in L=\cup_{n\in\N} L_n = \mathsf{B}(\mathsf{C})_h$. If $n=1$ and $x\in L_1$, then there exists $y\in\mathsf{B}(\mathsf{C}_{\ad})$ such that $x$ is a summand of $yc$, that is, $N_{y,c}^x\neq 0$. Then, by Lemma \ref{lemma:composicion-de-funciones-relacion-con-producto} and Proposition \ref{prop:adjunta-actua-trivial}, $\ell_x = \ell_y\circ \ell_c = \ell_c$. If $n>1$ and $x\in L_n$, then there exist $y\in\mathsf{B}(\mathsf{C}_{\ad})$ and $z\in L_{n-1}$ such that $x$ is a summand of $yz$, that is, $N_{y,z}^x\neq 0$. The inductive hypothesis implies $\ell_z = \ell_c$, and by the same argument as in the base case, $\ell_x = \ell_y \circ \ell_z = \ell_c$. This completes the proof.
 \end{proof}

\begin{coro}\label{coro:group-actions-fusion-ring}
 Let $\mathsf{R} = \mathsf{A}\bullet \mathsf{C}$ be an exact factorization of fusion rings. Then there exist $\mathbb{Z}$-linear group left action $\mathbf{\trid}\colon U(\mathsf{C})\times \mathsf{A}\to \mathsf{A}$ and right action $\mathbf{\fiz}\colon  \mathsf{C}\times U(\mathsf{A})\to \mathsf{C}$ that map basic elements to basic elements and 
    \begin{align}\label{eq:dar-vuelta-fact-exacta}
        c a = (\degf{c}\trid a) (c\fiz \degf{a}), && c\in \mathsf{B}(\mathsf{C}), a\in\mathsf{B}(\mathsf{A}).
    \end{align}
\end{coro}
\begin{proof}
We define $\degf{c}\trid a \coloneqq \ell_c(a)$ and $c\fiz\degf{a} \coloneqq r_a(c)$, for $c\in \mathsf{B}(\mathsf{C})$ and $a\in\mathsf{B}(\mathsf{A})$, and extend them $\mathbb{Z}$-linearly. These functions are well defined by Proposition \ref{prop:mismo-grado-universal-actua-igual} and are group actions by Lemmas \ref{lemma:accion-del-1} and \ref{lemma:composicion-de-funciones-relacion-con-producto}, and Corollary \ref{coro:T-A-es-invertible}.
\end{proof}

We study now the structure of the adjoint subring $\mathsf{R}_{\ad}$ and the universal grading group $U(\mathsf{R})$ in terms of the ones from $\mathsf{A}$ and $\mathsf{C}$.

\begin{prop}\label{prop:Adjoint-subcat-of-an-exact-fact}
    Let $\mathsf{R} = \mathsf{A}\bullet \mathsf{C}$ be an exact factorization of fusion rings. Then there exist an exact factorization for the adjoint subring $\mathsf{R}_{\ad} = \mathsf{A}_{\ad}\bullet \mathsf{C}_{\ad}$.  
\end{prop}
\begin{proof}
    Let $\widetilde{\mathsf{R}}_{\mathsf{A}}$ (respectively $\widetilde{\mathsf{R}}_{\mathsf{C}}$) be the $\mathbb{Z}$-submodule generated by the elements $a c$ with $a\in\mathsf{B}(\mathsf{A}_{\ad})$ and $c\in\mathsf{B}(\mathsf{C})$ (respectively $a\in\mathsf{B}(\mathsf{A})$ and $c\in\mathsf{B}(\mathsf{C}_{\ad})$). Let $\widetilde{\mathsf{R}} = \widetilde{\mathsf{R}}_{\mathsf{A}} \cap \widetilde{\mathsf{R}}_{\mathsf{C}}$.
    \begin{claim}
    The element $I_{\mathsf{R}}(1) = \sum_{b\in\mathsf{B}(\mathsf{R})} b b^*$ is in $\widetilde{\mathsf{R}}$.
    \end{claim}
    Given that every $b\in \mathsf{B}(\mathsf{R})$ can be written as $b = a c$, with $a\in\mathsf{B}(\mathsf{A})$ and $c\in\mathsf{B}(\mathsf{C})$, we have
    \begin{align*}
        I_{\mathsf{R}}(1) = \sum_{b\in\mathsf{B}(\mathsf{R})} bb^* =  \sum_{a\in\mathsf{B}(\mathsf{A})}  \sum_{c\in\mathsf{B}(\mathsf{C})} acc^*a^* =  \sum_{a\in\mathsf{B}(\mathsf{A})} a I_\mathsf{C}(1) a^*.
    \end{align*}
    Since $I_{\mathsf{C}}(1) = \sum_j c_j$, for some $c_j\in\mathsf{B}(\mathsf{C}_{\ad})$, we get that
    {\small\begin{align*}
        I_{\mathsf{R}}(1) = \sum_{a\in\mathsf{B}(\mathsf{A})}\sum_j  a c_j a^* = \sum_{a\in\mathsf{B}(\mathsf{A})} \sum_j a (\degf{c_j}\trid a^*) (c_j\fiz \degf{a^*}) = \sum_{a\in\mathsf{B}(\mathsf{A})} \sum_j a  a^* (c_j\fiz \degf{a^*}) \in \widetilde{\mathsf{R}}_{\mathsf{A}},
    \end{align*}}  where the second equality follows from Equation \eqref{eq:dar-vuelta-fact-exacta} and the third one from $\degf{c_j} = e$ and $a a^*\in \mathsf{A}_{\ad}$.
    A similar argument 
    shows that $I_\mathsf{R}(1)\in \widetilde{\mathsf{R}}_{\mathsf{C}}$.

    \begin{claim}
   The $\mathbb{Z}$-module $\widetilde{\mathsf{R}}$ with basis $\mathsf{B}(\widetilde{\mathsf{R}}) = \left\{ac \colon a\in \mathsf{B}(\mathsf{A}_{ad}), c\in \mathsf{B}(\mathsf{C}_{ad}) \right\}$ is a fusion subring of $\mathsf{R}_{ad}$. 
    \end{claim}
   
   If $b = ac$ and $b' = a'c'$ are elements of $B(\widetilde{\mathsf{R}})$ with $a, a'\in\mathsf{B}(\mathsf{A}_{\ad})$ and $c,c'\in\mathsf{B}(\mathsf{C}_{\ad})$ then 
    \begin{align*}
        bb' =  aca' c' = a (\degf{c}\trid a')(c\fiz\degf{a'})c'= aa' cc'\in\widetilde{\mathsf{R}},
    \end{align*}
    where the second equality follows from Equation \eqref{eq:dar-vuelta-fact-exacta} and the last one from $\degf{c} = e = \degf{a'}$. 
    
    \begin{claim} 
    $  \mathsf{R}_{\ad} = \widetilde{\mathsf{R}} = \mathsf{A}_{\ad}\bullet \mathsf{C}_{\ad}$.
    \end{claim}
    It follows directly from the definition of $\mathsf{B}(\widetilde{\mathsf{R}})$ and from the fact that $\mathsf{R} = \mathsf{A}\bullet \mathsf{C}$ that $\widetilde{\mathsf{R}} = \mathsf{A}_{\ad}\bullet \mathsf{C}_{\ad}$.  
    Since $\mathsf{R}_{\ad}$ is generated as a fusion ring by the element  $I_{\mathsf{R}}(1)$ and $I_\mathsf{R}(1)\in\widetilde{\mathsf{R}}$, we have that $\mathsf{R}_{\ad}$ is a fusion subring of $\widetilde{\mathsf{R}}$. Hence $\mathsf{R}_{\ad} \subseteq \widetilde{\mathsf{R}}$ and, since the other inclusion is clear, this concludes the proof.
\end{proof}

Let $(\mathsf{A}, \mathsf{B}(\mathsf{A})) \subseteq (\mathsf{R}, \mathsf{B}(\mathsf{R}))$ be a fusion subring and $H_{\mathsf{A}} := \{h\in U(\mathsf{R}) \colon \mathsf{A}\cap \mathsf{R}_h \neq 0 \}$  subgroup of $U(\mathsf{R})$ then $\mathsf{A} = \sum_{h\in H_{\mathsf{A}}} \mathsf{A}_h$, where $\mathsf{A}_h := \mathsf{A}\cap \mathsf{R}_h$ is a faithful grading. See \cite[Section 2]{Na_faithful} for details.
Because of the universality of $U(\mathsf{A})$ (respectively $U(\mathsf{C})$), there is a surjective group homomorphism $\phi_{\mathsf{A}}: U(\mathsf{A})\to H_{\mathsf A}$ (respectively $\phi_{\mathsf{C}}: U(\mathsf{C})\to H_{\mathsf C}$) with kernel denoted by $K_{\mathsf{A}}$ (respectively $K_{\mathsf{C}}$). 

If $\mathsf{R} = \mathsf A\bullet \mathsf C$ is an exact factorization of fusion rings then $U(\mathsf R)$ is a factorization (not necessarily exact) of the groups $H_{\mathsf A}$ and $H_{\mathsf C}$. Now, we will characterize the universal grading group $U(\mathsf R)$ of the exact factorization $\mathsf R = \mathsf A\bullet \mathsf C$. 
\begin{prop}\label{prop:canonical-matched-pair}
	If $\mathsf{R} = \mathsf A\bullet \mathsf C$ is an exact factorization of fusion rings then:
	\begin{enumerate}
		\item  The morphisms $\phi_{\mathsf A}$ and $\phi_{\mathsf C}$ are isomorphisms of groups.
		\item There is an exact factorization of groups $U(\mathsf R) \simeq U(\mathsf A)\bullet U(\mathsf C)$.
	\end{enumerate}
\end{prop}

\begin{proof}
    The groups $H_{\mathsf A}$ and $H_{\mathsf C}$ form a factorization of $U(\mathsf R)$, since for every $g\in U(\mathsf R)$ there is a basic element $b\in \mathsf{B}(\mathsf R)_g$. Since $\mathsf R =\mathsf A\bullet \mathsf C$ is exact, there exist unique $a\in \mathsf{B}(\mathsf{A})$ and $c\in \mathsf{B}(\mathsf{C})$ such that $b = ac$. If $a$ has degree $h\in H_\mathsf{A}$ and $c$ has degree $k\in H_\mathsf{C}$, then $g = hk$. Therefore, $U(\mathsf R) = H_{\mathsf A}  H_{\mathsf C}$ is a factorization and, by applying the Orbit-Stabilizer Theorem to the $H_{\mathsf A}$-action on the cosets $h H_{\mathsf C}$, we get  $|U(\mathsf R)| = \frac{|H_\mathsf A||H_\mathsf C|}{|H_\mathsf A \cap H_\mathsf C|}$. 
    Also, since $\phi_{\mathsf{A}}: U(\mathsf{A})\to H_{\mathsf A}$ and $\phi_{\mathsf{C}}: U(\mathsf{C})\to H_{\mathsf C}$ are epimorphisms with kernels $K_{\mathsf{A}}$ and $K_{\mathsf{C}}$ respectively, then $|H_\mathsf A| = \frac{|U(\mathsf A)|}{|K_\mathsf A|}$ and $|H_\mathsf C| = \frac{|U(\mathsf C)|}{|K_\mathsf C|}$.   
 By considering Frobenius-Perron dimensions, we get
\begin{align*}
    \fpdim \mathsf R &= \frac{|H_{\mathsf A}||H_{\mathsf{C}}|}{|H_{\mathsf A} \cap H_{\mathsf C}|} \fpdim \mathsf{R}_{\ad} = \frac{|U(\mathsf A)||U(\mathsf C)|}{|K_{\mathsf A}||K_{\mathsf C}||H_{\mathsf A} \cap H_{\mathsf C}|} \fpdim \mathsf{A}_{\ad} \fpdim \mathsf{C}_{\ad}\\
    &= \frac{\fpdim \mathsf{A}\fpdim \mathsf{C}}{|K_{\mathsf A}||K_{\mathsf C}||H_{\mathsf A} \cap H_{\mathsf C}|},
\end{align*}
where the first and third equalities follow from \cite[Proposition 8.20]{ENO} and the second equality follows from Proposition \ref{prop:Adjoint-subcat-of-an-exact-fact}.
    Since $\mathsf{R} = \mathsf{A} \bullet \mathsf{C}$ then
    \begin{align*}
       \fpdim \mathsf{R} = \sum_{b\in\mathsf{B}(\mathsf R)} (\fpdim b)^2 = \sum_{a\in\mathsf{B}(\mathsf A)} \sum_{c\in\mathsf{B}(\mathsf c)} (\fpdim a)^2 (\fpdim c)^ =\fpdim \mathsf{A} \fpdim \mathsf{C}.
    \end{align*}
    Therefore $|K_{\mathsf A}||K_{\mathsf C}||H_{\mathsf A} \cap H_{\mathsf C}| = 1$,
    which implies $|K_{\mathsf A}|=|K_{\mathsf C}|=|H_{\mathsf A} \cap H_{\mathsf C}|=1$. Hence, $\phi_{\mathsf A}$ and $\phi_{\mathsf C}$ are isomorphisms and the factorization $U({\mathsf R}) = H_{\mathsf A}  H_{\mathsf C}\simeq U(\mathsf A) U(\mathsf C) $ is exact because $H_{\mathsf A} \cap H_{\mathsf C} \simeq U(\mathsf A)\cap U(\mathsf C)$ is trivial.
\end{proof}

It follows from the previous proposition that every exact factorization of fusion rings $\mathsf{R} = \mathsf{A} \bullet \mathsf{C}$ has associated an exact factorization of their universal grading groups $U(\mathsf R) = U(\mathsf A) \bullet U(\mathsf C)$ and therefore, a matched pair $(U(\mathsf A), U(\mathsf C), \tridb, \fizb)$.

\begin{prop}[\emph{Canonical matched pair of fusion rings}]\label{prop:canonical-matchedpair-rings}  Let $\mathsf{R} = \mathsf{A}\bullet \mathsf{C}$ be an exact factorization of fusion rings. 
The data below gives rise to a matched pair of fusion rings in the following way: 
\begin{enumerate}[leftmargin=*,label=\rm{(\roman*)}]
    \item the grading of $\mathsf A$ and $\mathsf C$ by their universal grading groups $U(\mathsf A)$ and $U(\mathsf C)$,
    \item the associated matched pair of groups $(U(\mathsf A), U(\mathsf C), \tridb, \fizb)$ defined above,
    \item the group actions $\mathbf{\trid}\colon U(\mathsf{C})\times \mathsf{A}\to \mathsf{A}$ and $\mathbf{\fiz}\colon  \mathsf{C}\times U(\mathsf{A})\to \mathsf{C}$ from Corollary \ref{coro:group-actions-fusion-ring}.
\end{enumerate}
We call this matched pair the \emph{canonical matched pair} of fusion rings associated to the exact factorization $\mathsf{R} = \mathsf{A}\bullet \mathsf{C}$.
\end{prop}
\begin{proof}
  To prove that Definition \ref{def:mathced-pair-fusion-rings}\ref{def:fusion-ring-2} holds, it is enough to show that $k\trid a \in \mathsf{B}(\mathsf A)_{k\tridb h}$ and $c\fiz h\in \mathsf{B}(\mathsf C)_{k \fizb h}$, for $h\in U(\mathsf A)$, $k\in U(\mathsf C)$, $a\in\mathsf{B}(\mathsf A)_h$, and $c\in\mathsf{B}(\mathsf{C})_k$. Identifying $U(\mathsf A)$ and $U(\mathsf C)$ as subgroups of $U(\mathsf R)$ via $\phi_{\mathsf A}$ and $\phi_{\mathsf C}$, respectively, the element $c a$ has degree $\degf{ca} = \degf{c}\degf{a} = k h = (k\tridb h)(k \fizb h)\in U(\mathsf R)$.
    On the other hand, since $c a = (k\trid a) (c\fiz h)$ then $ \degf{ca} = \degf{(k\trid a) (c\fiz h)} = \degf{k\trid a} \degf{c\fiz h}.$
    By the uniqueness of the decomposition of elements of  $U(\mathsf R)$ in product of elements of $U(\mathsf A)$ and $U(\mathsf C)$,  we have $\degf{k\trid a} = k\tridb h$ and $\degf{c\fiz h } = k \fizb h$. 
   Hence  $k\trid \mathsf{B}(\mathsf A)_h = \mathsf{B}(\mathsf A)_{k\tridb h}$ and $\mathsf{B}(\mathsf C)_k\fiz h = \mathsf{B}(\mathsf C)_{k\fizb h}$ as desired.

   To check that Definition \ref{def:mathced-pair-fusion-rings} \ref{def:mathced-pair-fusion-rings-3} holds, we consider
   \begin{align*}
       c a_1 a_2 &= (k\trid a_1) (c\fiz \degf{a_1})a_2 = (k\trid a_1) (\degf{c\fiz \degf{a_1}}\trid a_2) ((c\fiz \degf{a_1}) \fiz \degf{a_2})\\
       &=(k\trid a_1) ((k\fizb \degf{a_1})\trid a_2) (c\fiz \degf{a_1}\degf{a_2}),
   \end{align*}
    for $a_1, a_2\in \mathsf{B}(\mathsf A)$, $k\in U(\mathsf C)$ and $c\in \mathsf{B}(\mathsf C)_k$.

Notice that if $N_{a_1,a_2}^a \neq 0$ then $\degf{a} =\degf{a_1}\degf{a_2}$ and we get the alternative expression
   \begin{align*}
       ca_1a_2&=\sum_{a\in \mathsf{B}(\mathsf A)} N_{a_1,a_2}^a (k\trid a) (c\fiz \degf{a}) = \sum_{a\in \mathsf{B}(\mathsf A)} N_{a_1,a_2}^a (k\trid a) (c\fiz \degf{a_1}\degf{a_2})\\
       &= \left(k\trid \left(\sum_{a\in \mathsf{B}(\mathsf A)} N_{a_1,a_2}^a a\right)\right) (c\fiz \degf{a_1}\degf{a_2}) =(k\trid (a_1 a_2))  (c\fiz \degf{a_1}\degf{a_2}),
   \end{align*}
   for $a_1, a_2\in \mathsf{B}(\mathsf A)$, $k\in U(\mathsf C)$ and $c\in \mathsf{B}(\mathsf C)_k$.
   
   Then $(k\trid (a_1 a_2))  (c\fiz \degf{a_1}\degf{a_2}) = (k\trid a_1) ((k\fizb \degf{a_1})\trid a_2) (c\fiz \degf{a_1}\degf{a_2})$.
    By uniqueness of the decomposition in basic elements, we have $ k\trid (a_1 a_2)  = (k\trid a_1) ((k\fizb \degf{a_1})\trid a_2).$
   Similarly, $(c_1 c_2)\fiz h = (c_1\fiz (\degf{c_2}\tridb h)) (c_2\fiz h)$, for $c_1,c_2\in\mathsf{B}(\mathsf C)$ and $h\in U(\mathsf{A})$. 
   
   It follows from Lemma \ref{lemma:accion-del-1} that the condition in Definition \ref{def:mathced-pair-fusion-rings}\ref{def:mathced-pair-fusion-rings-4} holds.
\end{proof}

Now, we will prove the main theorem from this subsection.

\subsubsection*{Proof of Theorem \ref{teo:char-exact-fact-fusion-ring}}
\begin{proofw}
    For the exact factorization $\mathsf{R} = \mathsf{A}\bullet \mathsf{C}$, we consider the canonical matched pair of fusion rings associated to it, as in Proposition \ref{prop:canonical-matchedpair-rings}, and the bicrossed product $\mathsf A \bowtie \mathsf C$, as in Definition \ref{def:bicrossed-fusion-rings}.  We define the $\mathbb{Z}$-module map $\psi\colon \mathsf A \bowtie \mathsf C \to \mathsf R$ on the basic elements $a\in\mathsf{B}(\mathsf A)$ and $c\in\mathsf{B}(\mathsf C)$ by
    $\psi(a\bowtie c) = ac$, which is an isomorphism since it sends basic elements into basic elements.
    To finish the proof, we check below that $\psi$ is a ring map: 
    \begin{align*}
        &\psi\left((a_1\bowtie c_1) (a_2\bowtie c_2)\right) =
        \psi(a_1 (\degf{c_1} \trid a_2)\bowtie (c_1\fiz \degf{a_2}) c_2)\\
        &= \sum_{a\in \mathsf{B}(A), c\in \mathsf{B}(\mathsf{C})} N_{a_1, \degf{c_1} \trid a_2}^a N_{c_1\fiz \degf{a_2}, c_2}^c \psi(a\bowtie c) = \sum_{a\in \mathsf{B}(A), c\in \mathsf{B}(\mathsf{C})} N_{a_1, \degf{c_1} \trid a_2}^a N_{c_1\fiz \degf{a_2}, c_2}^c ac\\
        &= a_1 (\degf{c_1} \trid a_2)(c_1\fiz \degf{a_2}) c_2 = a_1 c_1 a_2 c_2 = \psi\left(a_1\bowtie c_1\right) \psi\left(a_2\bowtie c_2\right),
    \end{align*}
    for $a_1, a_2\in \mathsf{B}(\mathsf A)$ and $c_1, c_2\in\mathsf{B}(\mathsf C)$.
\end{proofw}
\begin{coro}\label{coro:fusion-ring-adjoint-subcat} Let $\mathsf{R} = \mathsf{A}\bullet\mathsf{C}$ be an exact factorization of fusion rings. Then $\mathsf{R}_{\ad} \simeq \mathsf{A}_{\ad}\ot_{\mathbb{Z}}\mathsf{C}_{\ad}$ as fusion rings.
\end{coro}
\begin{proof}
    This follows from Proposition \ref{prop:Adjoint-subcat-of-an-exact-fact} and Theorem \ref{teo:char-exact-fact-fusion-ring}. Indeed, $\mathsf{R} \simeq \mathsf{A}\bowtie\mathsf{C}$ and the result follows from looking at the product formula in Definition \ref{def:bicrossed-fusion-rings} in the trivial component, which is exactly the adjoint fusion ring. 
\end{proof}
\subsection{Grothendieck invariants of exact factorizations of fusion categories}\label{subsection:applications-exact-fact}

In this subsection, we present some results at the level of fusion categories that depend only on the underlying Grothendieck ring and, therefore, are direct consequence of the results in previous sections, since an exact factorization of fusion categories $\cB = \cA \bullet \cC$ induces an exact factorization of the corresponding fusion rings $K_0(\cB) = K_0(\cA)\bullet K_0(\cC)$.

 The next proposition follows from Proposition \ref{prop:Adjoint-subcat-of-an-exact-fact} and Proposition \ref{prop:canonical-matched-pair}.

\begin{prop}\label{prop:exact-fact-of-univ-grad-group}
    Let $\cB = \cA \bullet \cC$ be an exact factorization of fusion categories then
    \begin{enumerate}
        \item there is an exact factorization of the adjoint subcategory $\cB_{\ad} = \cA_{\ad}\bullet \cC_{\ad}$, and
        \item there is an exact factorization $U(\cB) \simeq U(\cA)\bullet U(\cC)$ of the universal grading group.
    \end{enumerate}
\end{prop}

The \emph{upper central series} $\mathcal{B}^{(0)} = \mathcal{B}$ of a fusion category is defined recursively as follows:  $\mathcal{B}^{(1)} := \mathcal{B}_\text{ad}$, and $\mathcal{B}^{(n)} := (\mathcal{B}^{(n-1)})_\text{ad}$, for $n \geq 1$. This gives rise to a non-increasing sequence of fusion subcategories $\mathcal{B} = \mathcal{B}^{(0)} \supseteq \mathcal{B}^{(1)} \supseteq \cdots \supseteq \mathcal{B}^{(n)} \supseteq \cdots$ of $\mathcal{B}$. 
A fusion category $\mathcal{B}$ is \emph{nilpotent} if there exists $n$ such that $\mathcal{B}^{(n)} = \vect$, see \cite[Definition 4.4]{GN}.

The following corollary is a direct consequence of Proposition \ref{prop:Adjoint-subcat-of-an-exact-fact}.

\begin{coro}\label{coro:nilpotent-exact-fact}
Let $\mathcal{B}=\mathcal{A}\bullet\mathcal{C}$ be an exact factorization of fusion categories. Then $\mathcal{B}$ is nilpotent if and only if $\mathcal{A}$ and $\mathcal{C}$ are nilpotent fusion categories. 
\end{coro}

\begin{remark}\label{remark-clave}
The previous results allow us to recover every exact factorization of fusion categories $\cB=\cA \bullet \cC$ by extension theory (see \cite[Theorem 7.7]{eno2}). Indeed, by Proposition \ref{prop:exact-fact-of-univ-grad-group} such a factorization induces an exact factorization $\cB_{\ad} = \cA_{\ad}\bullet \cC_{\ad}$ for the adjoint subcategory with same fusion rules as the Deligne product $\cA_{\ad}\boxtimes \cC_{\ad}$ by Corollary \ref{coro:fusion-ring-adjoint-subcat}. Categories satisfying this last condition are discussed in a paper in preparation by  P. Etingof, D. Nikshych, and V. Ostrik \cite{ENO-no-published}. Then the category $\cB = \cA \bullet \cC$ is recovered as a $U(\cB) = U(\cA)\bullet U(\cC)$-extension of $\cB_{\ad} = \cA_{\ad}\bullet \cC_{\ad}$. Then, by extension theory, we get exact factorizations $\cB^\omega$ of $\cA$ and $\cC$ by multiplying the associativity of $\cB$ by a cocycle $\omega \in H^3(U(\cA)\bullet U(\cC), \ku^\times)$. All exact factorizations with trivial component given $\cB_{\ad}$ and the same fusion rules as $\cB$ are of this form. Notice that the fusion rules of $\cB=\cA \bullet \cC$ correspond to some bicrossed product of the fusion rings of $\cA$ and $\cC$ by Theorem \ref{teo:char-exact-fact-fusion-ring}.
\end{remark}

\section{Matched pair of fusion categories and their bicrossed product}\label{section:exact-factorization-fusion-categories}
In the previous sections, we showed how to construct exact factorizations of fusion rings from matched pairs between them. We now explore the categorification of these ideas. We will define matched pair between two fusion categories $\mathcal A$ and $\mathcal C$ and from them we will construct a new fusion category $\mathcal{A}\bowtie \mathcal C$ called the bicrossproduct fusion category between them. This will provide an exact factorization of $\mathcal{A}$ and $\mathcal{C}$, and will allow us to construct new fusion categories from known ones. 
\begin{definition}\label{def:mathced-pair-categories} 
Let $\mathcal{A}$ and $\mathcal{C}$ be fusion categories. 
A \emph{matched pair of fusion categories} between $\mathcal{A}$ and $\mathcal{C}$ is a $8$-tuple $(\mathcal A, \mathcal C, H, K, \tridb, \fizb, \trid, \fiz)$, where
    \begin{enumerate}[leftmargin=*,label=\rm{(\roman*)}]
    \item  the categories $\mathcal{A}$ and $\mathcal{C}$ have faithful gradings by groups the $H$ and $K$, respectively,
    \begin{align*}
        \cA = \bigoplus_{h\in H} \cA_h, && \cC = \bigoplus_{k\in K} \cC_k,
    \end{align*}

    \item a matched pair of groups $(H, K, \tridb, \fizb)$ between $H$ and $K$,
    \item categorical $\ku$-linear actions:
    \begin{align*}
        \trid &\colon \underline{K} \to \Aut(\cA),  &
        \fiz &\colon \underline{H}^{\operatorname{op}} \to \Aut(\cC),
    \end{align*}
    such that $k \trid \cA_h = \cA_{k\tridb h}$ and $\cC_k \fiz h= \cC_{k \fizb h}$,
    \item for each $k\in K$ and $h\in H$, there are families of natural isomorphisms 
    \begin{align*}
        &\gamma_{A,A'}^k\colon k\trid (A\ot A') \xrightarrow{\quad\simeq\quad} (k\trid A)\ot (k\fizb\degf{A})\trid A', & A,A'\in\Irr(\cA),\\
        &\eta_{C,C'}^h\colon (C\ot C')\fiz h \xrightarrow{\quad\simeq\quad} C\fiz (\degf{C'}\tridb h)\ot C'\fiz h, & C,C'\in\Irr(\cC),
    \end{align*}
 \item for each $k\in K$ and $h\in H$, there are isomorphisms
    \begin{align*}
        &\gamma_{0}^k\colon k\trid \mathbf{1} \xrightarrow{\quad\simeq\quad} \mathbf{1}, & \\
        &\eta_{0}^h\colon \mathbf{1}\fiz h \xrightarrow{\quad\simeq\quad} \mathbf{1}, & 
    \end{align*}
\end{enumerate}
such that the following diagrams commute
{\footnotesize
\begin{equation}\label{eq:gamma-T2-diagram}
   \begin{adjustbox}{scale=.8}
        \begin{tikzcd}
            k\trid(k'\trid(A\ot A')) \ar[d, "(\mathbf{L}^2_{k,k'})_{A\ot A'}"] \ar[rr, "k\trid \gamma^{k'}_{A,A'}"]&& k\trid (k'\trid A \ot (k'\fizb \degf{A})\trid A') \ar[d, "\gamma^k_{k'\trid A, (k'\fizb \degf{A})\trid A'}"]\\
            k k'\trid(A\ot A') \ar[d, "\gamma^{k k'}_{A,A'}"] && k\trid(k'\trid A)\ot (k\fizb(k'\tridb\degf{A}))\trid ((k'\fizb \degf{A})\trid A') 
            \ar[lld,bend left=10, "(\mathbf{L}^2_{k,k'})_A\ot (\mathbf{L}^2_{k\fiz(k'\tridb\degf{A}), k'\fizb \degf{A}})_{A'}"]     \\
          k k' \trid A \ot (k k'\fizb \degf{A})\trid A'  &  & 
        \end{tikzcd}
        \end{adjustbox}
\end{equation}

    \begin{equation}\label{eq:eta-S2-diagram}
   \begin{adjustbox}{scale=.8}
        \begin{tikzcd}
          ((C\ot C')\fiz h)\fiz h' \ar[d, "(\mathbf{R}^2_{h,h'})_{C\ot C'}"] \ar[rr, "\eta^{h}_{C,C'}\fiz h'"]& & ((C \fiz (\degf{C'}\tridb h)) \ot C'\fiz h)\fiz h' \ar[d, "\eta^{h'}_{C\fiz(\degf{C'}\tridb h), C'\fiz h}"]\\
            (C\ot C')\fiz(hh') \ar[d, "\eta^{hh'}_{C,C'}"] && (C \fiz (\degf{C'}\fizb h))\fiz ((\degf{C'}\fizb h)\tridb h')\ot (C'\fiz h)\fiz h' 
            \ar[lld,bend left=10,"(\mathbf{R}^2_{\degf{C'}\tridb h, (\degf{C'}\fizb h)\tridb h'})_C\ot (\mathbf{R}^2_{h, h'})_{C'}"]     \\
            (C\fiz(\degf{C'}\tridb (hh')))\ot C'\fiz(hh') &  & 
        \end{tikzcd}
        \end{adjustbox}
    \end{equation}

\begin{minipage}{0.45\textwidth}
	\begin{equation}
		  \begin{tikzcd}\label{diagram:MP-13}
				k\trid (\uno \ot A) \ar[r,"\gamma^k_{\uno,A}"] \ar[d,"k\trid \ell_A"]& k\trid \uno \ot k\trid A \ar[d,"\gamma_0^k\ot \id_{k\trid A}"]\\
				k\trid A  & \uno\ot k\trid A \ar[l,"\ell_{k\trid A}"] 
		\end{tikzcd} 
	\end{equation}
\end{minipage}
\begin{minipage}{0.45\textwidth}
	\begin{equation}\label{diagram:MP-14}
		\begin{tikzcd}
			k\trid (A\ot \uno) \ar[r,"\gamma^k_{A,\uno}"] \ar[d,"k\trid r_A"]& k\trid A \ot (k\fizb \degf{A})\trid \uno \ar[d,"\id_{k\trid A}\ot \gamma_0^{k\fizb\degf{A}} "]\\
			k\trid A  &  k\trid A \ot \uno \ar[l,"r_{k\trid A}"]
		\end{tikzcd}
	\end{equation}
\end{minipage}

\begin{minipage}{0.45\textwidth}
	\begin{equation}\label{diagram:MP-15}
		    \begin{tikzcd}
				(\uno \ot C)\fiz h \ar[r,"\eta^h_{\uno,C}"] \ar[d,"\ell_C\fiz h"]&\uno \fiz (\degf{C}\tridb h)\ot C\fiz h \ar[d,"\eta_0^{\degf{C}\tridb h}\ot \id_{C\fiz h} "]\\
				C\fiz h  & \uno\ot C\fiz h \ar[l,"\ell_{C\fiz h}"]
		\end{tikzcd}
	\end{equation}
\end{minipage}
\begin{minipage}{0.45\textwidth}
	\begin{equation}\label{diagram:MP-16}
		\begin{tikzcd}
			(C\ot \uno)\fiz h \ar[r,"\eta^h_{C,\uno}"] \ar[d,"r_C\fiz h"]&C\fiz h \ot \uno\fiz h \ar[d,"\id_{C\fiz h}\ot \eta_0^{h} "]\\
			C\fiz h  & C\fiz h \ot \uno \ar[l,"r_{C\fiz h}"]
		\end{tikzcd}
	\end{equation}
\end{minipage}

\begin{minipage}{0.45\textwidth}
	\begin{equation}\label{eq:T2-gamma-uno}
		\begin{tikzcd}
			k\trid (k'\trid \uno) \ar[r,"(L^2_{k,k'})_\uno"] \ar[d, "k\trid \gamma_0^{k'}"] & k k'\trid \uno \ar[d,"\gamma_0^{k k'}"]\\
			k\trid \uno \ar[r,"\gamma_0^k"] & \uno
		\end{tikzcd}
	\end{equation}
\end{minipage}
\begin{minipage}{0.45\textwidth}
	\begin{equation}\label{diagram:MP-18}
		\begin{tikzcd}
			(\uno \fiz h)\fiz h' \ar[r,"(R^2_{h,h'})_\uno"] \ar[d,"\eta_0^h \fiz h'"] & \uno \fiz h h' \ar[d,"\eta_0^{h h'}"]\\
			\uno \fiz h' \ar[r,"\eta_0^{h'}"] & \uno
		\end{tikzcd}
	\end{equation}
\end{minipage}

    \begin{equation}\label{eq:gamma-alpha-diagram}
\begin{adjustbox}{scale=.8}
    \begin{tikzcd}
        k\trid ((A\ot A')\ot A'') \ar[dd,"\gamma^k_{A\ot A',A''}"]\ar[rr, "k\trid \alpha_{A,A',A''}"] & & k\trid (A\ot (A'\ot A'')) \ar[dd, "\gamma^k_{A,A'\ot A''}"] \\
        &&\\
        k \trid (A\ot A') \ot (k\fizb |A||A'|)\trid A'' \ar[dd, "\gamma^k_{A,A'}\ot \id"] && k\trid A\ot (k\fizb |A|)\trid (A'\ot A'') \ar[ddddll, bend left, "\id\ot \gamma^{k\fizb |A|}_{A',A''}"]\\
        &&\\
        (k\trid A\ot (k\fizb |A|)\trid A')\ot (k\fizb |A||A'|)\trid A''\ar[dd,"\alpha_{k\trid A, (k\fizb |A|)\trid A', (k\fizb |A||A'|)\trid A''}"]\\
        &&\\
        k\trid A\ot ((k\fizb |A|)\trid A'\ot (k\fizb |A||A'|)\trid A'') 
    \end{tikzcd}
    \end{adjustbox}
    \end{equation}
    
    \begin{equation}\label{eq:eta-alpha-diagram}
  \begin{adjustbox}{scale=.8}
    \begin{tikzcd}
         ((C\ot C')\ot C'')\fiz h \ar[dd,"\eta^h_{C\ot C',C''}"]\ar[rr, " \alpha_{C,C',C''}\fiz h"] & & (C\ot (C'\ot C''))\fiz h \ar[dd, "\eta^h_{C,C'\ot C''}"] \\
        &&\\
        (C\ot C') \fiz (|C''|\tridb h) \ot C'' \fiz h \ar[dd, "\eta^{|C''|\tridb h}_{C,C'}\ot \id"] && C\fiz (|C'||C''|\tridb h) \ot (C'\ot C'')\fiz h \ar[ddddll,bend left, "\id\ot \eta^{h}_{C',C''}"]\\
        &&\\
       (C \fiz (|C'||C''|\tridb h) \ot C' \fiz (|C''|\tridb h)) \ot C''\fiz h\ar[dd,"\alpha_{C \fiz (|C'||C''|\tridb h), C' \fiz (|C''|\tridb h), C''\fiz h}"]\\
        &&\\
        C \fiz (|C'||C''|\tridb h) \ot (C' \fiz (|C''|\tridb h) \ot C''\fiz h ).
    \end{tikzcd}
    \end{adjustbox}
    \end{equation}
}
\end{definition}
\begin{remark}
    Notice that the maps $\gamma^k_0$ and $\eta^h_0$ are completely determined by $\gamma^k_{e,e}$ and $\eta^h_{e,e}$ from the diagrams \eqref{diagram:MP-13} and \eqref{diagram:MP-16}, respectively.
\end{remark}

From the definition, it follows that more commutative diagrams are satisfied directly as the next lemma states. The proof can be found in the Appendix appearing in the ArXiv version of this article.

\begin{lemma}\label{diagrmas-e-Sonia}
   For every $A$, $A'$ homogeneous objects in $\cA$, and $C$, $C'$ homogeneous objects in $\cC$, the following diagrams commute: 
   
   \begin{minipage}{0.45\textwidth}
 \begin{equation}\label{diagrama-e-Sonia-1}
\begin{adjustbox}{scale=.8,center}
\begin{tikzcd}
   e\trid (A\ot A') \ar[rr, "\gamma^e_{A, A'}"] & & e\trid A \ot e\trid A' \\
    & & \\
   A\ot A'.\ar[uu,"(\mathbf{L}^0)_{A\ot A'}"]\ar[rruu,swap, "(\mathbf{L}^0)_A\ot (\mathbf{L}^0)_{A'}"]  &&
\end{tikzcd}
\end{adjustbox}
    \end{equation}
    \end{minipage}
    \begin{minipage}{0.5\textwidth}
\begin{equation}\label{diagrama-e-Sonia-2}
\begin{adjustbox}{scale=.8,center}
\begin{tikzcd}
   \uno\ar[ddrr,"\id"] \ar[rr, "({\mathbf L}^0)_\uno"] & & e\trid \uno\ar[dd,"\gamma_0^e"] \\
    & & \\
   && \uno.
\end{tikzcd}
\end{adjustbox}
    \end{equation}
    \end{minipage}
    
\begin{minipage}{0.45\textwidth}
\begin{equation}\label{diagrama-e-Sonia-3}
\begin{adjustbox}{scale=.8,center}
\begin{tikzcd}
   (C\ot C')\fiz e \ar[rr, "\eta^e_{C, C'}"] & &  C\fiz e \ot C'\fiz e \\
    & & \\
   C\ot C'.\ar[uu,"(\mathbf{R}^0)_{C\ot C'}"]\ar[rruu,swap, "(\mathbf{R}^0)_C\ot (\mathbf{R}^0)_{C'}"]  &&
\end{tikzcd}
\end{adjustbox}
    \end{equation}
    \end{minipage}
    \begin{minipage}{0.5\textwidth}
\begin{equation}\label{diagrama-e-Sonia-4}
\begin{adjustbox}{scale=.8,center}
\begin{tikzcd}
   \uno\ar[ddrr,"\id"] \ar[rr, "({\mathbf R}^0)_\uno"] & &  \uno\fiz e\ar[dd,"\eta_0^e"] \\
    & & \\
   && \uno.
\end{tikzcd}
\end{adjustbox}
    \end{equation}
    \end{minipage}

\end{lemma}

It is clear that a matched pair of fusion categories $(\mathcal A, \mathcal C, H, K, \tridb, \fizb, \trid, \fiz)$ induces a matched pair of the corresponding fusion rings $(K_0(\mathcal A), K_0(\mathcal C), H, K, \tridb, \fizb, \trid, \fiz)$ by decategorification. This motivates the following definition.

\begin{definition}\label{def:lifting-matched-pair}
    Let $(\mathsf{A}, \mathsf{C}, H, K, \tridb, \fizb, \trid, \fiz)$ be a matched pair of fusion rings, and let $\cA$ and $\cC$ be fusion categories such that $K_0(\cA) = \mathsf{A}$ and $K_0(\cC) = \mathsf{C}$. A \emph{lifting} of this matched pair is a matched pair of fusion categories $(\mathcal A, \mathcal C, H, K, \tridb, \fizb, \trid, \fiz)$ such that its decategorification is the original matched pair of fusion rings.
\end{definition}

Given a matched pair of fusion categories $(\mathcal A, \mathcal C, H, K, \tridb, \fizb, \trid, \fiz)$, we define the 
\emph{bicrossed product} $\mathcal{A}\bowtie\mathcal{C}$ of $\mathcal{A}$ and $\mathcal{C}$ as the following semisimple monoidal category:
\begin{enumerate}[leftmargin=*,label=\rm{(\roman*)}]
    \item As an abelian category $\mathcal{A}\bowtie \mathcal{C} = \mathcal{A} \boxtimes \mathcal{C}$. This is a semisimple abelian category with simple objects of the form $A\bowtie C\coloneqq A\boxtimes C$, with $A\in\Irr(\mathcal{A})$ and $C\in\Irr(\mathcal{C})$. For $A$, $A'\in \mathcal{A}$ and $C, C'\in\mathcal{C}$, we have 
\begin{align*}
    \Hom_{\mathcal{A}\bowtie \mathcal{C}}(A\bowtie C, A'\bowtie C') \simeq \Hom_{\mathcal{A}}(A,A')\ot \Hom_{\mathcal{C}}(C,C').
\end{align*}
and we denote by $f\bowtie g:=f \boxtimes g$, for $f\in \Hom_{\mathcal{A}}(A,A')$ and $g\in \Hom_{\mathcal{C}}(C,C')$.
    \item The monoidal structure is given on simple objects as
    \begin{align*}
    A\bowtie C \ot A'\bowtie C' \coloneqq A\ot \degf{C} \trid A'\bowtie C\fiz \degf{A'} \ot C', && A,A'\in\Irr(\cA), C,C'\in\Irr(\cC),
\end{align*}
\end{enumerate}

 It is enough to define tensor products of morphisms of the form $f\bowtie g$ and $f'\bowtie g'$ for $f\colon A\to B$, $f'\colon A'\to B'$ in $\mathcal{A}$, $g\colon C\to D, g^\prime\colon C^\prime\to  D^\prime$ in $\cC$ with homogeneous degree objects as source and target.  Then the morphism
    \begin{align*}
        (f\bowtie g)\ot (f'\bowtie g)\colon (A\bowtie C)\ot (A'\bowtie C') \to (B\bowtie D)\ot (B'\bowtie D')
    \end{align*}
is given by the formula $f\bowtie g \otimes f^\prime\bowtie g^\prime :=f\otimes \degf C\trid f^\prime\bowtie g\fiz \degf{A^\prime}\otimes g^\prime$.

The associativity $\alpha_{A\bowtie C, A'\bowtie C', A''\bowtie C''}:= \alpha^L_{A\bowtie C, A'\bowtie C', A''\bowtie C''}\bowtie \alpha^R_{A\bowtie C, A'\bowtie C', A''\bowtie C''}$ is defined as follows
{\tiny\begin{align*}\label{formula:asociatividad}
   \alpha^L_{A\bowtie C, A'\bowtie C', A''\bowtie C''} & = \left(\id_A \otimes {\gamma^{\degf{C}}_{A', \degf{C'}\trid A''}}^{-1}\right)\left(\id_A\otimes\id_{\degf C\trid A^\prime}\otimes ({\mathbf{L}^2_{\degf{C}\fiz \degf{A'}, \degf{C'}}}^{-1})_{A''}\right)\alpha_{A, |C|\trid A',(|C|\fiz |A'|)|C'|\trid A''},\\ \nonumber
   \alpha^R_{A\bowtie C, A'\bowtie C', A''\bowtie C''} & = \alpha_{C\fiz |A'|(|C'|\tridb |A''|), C'\fiz |A''|, C''}\left((\mathbf{R}^2_{\degf{A'}, \degf{C'}\tridb\degf{A''}})_C\otimes \id_{C^\prime\fiz \degf{A''}}\otimes \id_{C^{\prime\prime}}\right)\left(\eta^{\degf{A''}}_{C\fiz \degf{A'}, C'}\otimes \id_{C^{\prime\prime}}\right).
\end{align*}}

The unit object is $\uno\bowtie \uno$, with left and right natural transformations 
given by
{\footnotesize\begin{align*}
 \ell_{A\bowtie C}\coloneqq(\ell_A\bowtie \ell_C) (\id_\uno\otimes {(\mathbf{L}^0_A)}^{-1}\bowtie \eta_0^{\degf A}\otimes \id_C),
 &&r_{A\bowtie C}:=(r_A\bowtie r_C) (\id_A\otimes \gamma_0^{\degf C}\bowtie {(\mathbf{R}^0_C)}^{-1}\otimes \id_{\uno}).
\end{align*}}

Before showing that $\cA \bowtie \cC$ is a monoidal category, we need the following lemma which will be useful for proving certain diagrams are commutative.
\begin{lemma}\label{lemma:commutativity-both-sides-bowtie}
Let $\cA$ and $\cC$ be semisimple abelian categories.The diagram 
\begin{equation}\label{eq:mixed-lemma-deligne-product}
\begin{adjustbox}{scale=.8}
    	\begin{tikzcd}
		A_{1,1}\boxtimes C_{1,1} \ar[r, "f_{1,2}\boxtimes g_{1,2}"] \ar[d,swap, "f_{2,1}\boxtimes g_{2,1}"]& A_{1,2}\boxtimes  C_{1,2} \ar[rd, dotted] & \\
		A_{2,1}\boxtimes C_{2,1} \ar[rd, dotted] & & A_{1,n-1}\boxtimes C_{1,n-1} \ar[d, "f_{1,n}\boxtimes g_{1,n}"]\\
		& A_{m-1,1}\boxtimes C_{m-1,1} \ar[r,swap, "f_{m,1}\boxtimes g_{m,1}"] & A_{m,1}\boxtimes C_{m,1} = A_{1,n}\boxtimes C_{1,n} 
	\end{tikzcd}
    \end{adjustbox}
\end{equation}
is commutative in $\cA\boxtimes \cC$ if the two diagrams
\begin{equation}\label{eq:LI-lemma-delingue-product}
\begin{adjustbox}{scale=.8}
	\begin{tikzcd}
		A_{1,1} \ar[r, "f_{1,2}"] \ar[d,swap, "f_{2,1}"]& A_{1,2} \ar[rd, dotted] & \\
		A_{2,1} \ar[rd, dotted] & & A_{1,n-1} \ar[d, "f_{1,n}"]\\
		& A_{m-1,1} \ar[r,swap, "f_{m,1}"] & A_{m,1} = A_{1,n}, 
	\end{tikzcd}
    \end{adjustbox}
\end{equation}
	\begin{equation}\label{eq:LD-lemma-delingue-product}
    \begin{adjustbox}{scale=.8}
	\begin{tikzcd}
		C_{1,1} \ar[r, "g_{1,2}"] \ar[d,swap, "g_{2,1}"]& C_{1,2} \ar[rd, dotted] & \\
		C_{2,1} \ar[rd, dotted] & & C_{1,n-1} \ar[d, "g_{1,n}"]\\
		& C_{m-1,1} \ar[r,swap, "g_{m,1}"] & C_{m,1} = C_{1,n}. 
	\end{tikzcd}
     \end{adjustbox}
\end{equation}
are commutative in $\cA$ and $\cC$, respectively.
\end{lemma}
\begin{proof}
The commutativity of Diagram \ref{eq:mixed-lemma-deligne-product} is equivalent to checking the equality
	\begin{align*}
		(f_{m,1}\boxtimes g_{m,1}) \circ \dots \circ (f_{2,1}\boxtimes g_{2,1})   = (f_{1,n}\boxtimes g_{1,n}) \circ \dots \circ (f_{1,2}\boxtimes g_{1,2}).  
	\end{align*}
It follows from the commutativity of Diagrams \eqref{eq:LI-lemma-delingue-product} and \eqref{eq:LD-lemma-delingue-product} that
\begin{align*}
f_{m,1} \circ \dots \circ f_{2,1}   = f_{1,n} \circ \dots \circ f_{1,2}, && g_{m,1} \circ \dots \circ g_{2,1}   = g_{1,n} \circ \dots \circ g_{1,2},   
\end{align*}
 respectively. Then, we get that
\begin{align*}
	&(f_{m,1}\boxtimes g_{m,1}) \circ \dots \circ (f_{2,1}\boxtimes g_{2,1}) = (f_{m,1} \circ \dots \circ f_{2,1}) \boxtimes (g_{m,1} \circ \dots \circ g_{2,1})\\
	& = (f_{1,n} \circ \dots \circ f_{1,2}) \boxtimes (g_{1,n} \circ \dots \circ g_{1,2}) = (f_{1,n}\boxtimes g_{1,n}) \circ \dots \circ (f_{1,2}\boxtimes g_{1,2}),
\end{align*} 
by repeatedly applying the formula $(f\boxtimes g) \circ (\widetilde{f}\boxtimes \widetilde{g}) = f\circ \widetilde{f} \boxtimes g \circ \widetilde{g}$,
which proves the lemma.
\end{proof}

\begin{theorem}\label{teo:bowtie-monoidal}
    The bicrossed product $\cA\bowtie \cC$ with $\ot$ as above is a monoidal category.
\end{theorem}
\begin{proof} It is a straightforward computation to check that $\ot$ is a bifunctor. It is enough to check the axioms for homogeneous objects $A\bowtie C$ and $A'\bowtie C'$. Since all the morphisms involved arrows of the form $f\bowtie g$, for some morphisms $f$ and $g$, we can apply Lemma \ref{lemma:commutativity-both-sides-bowtie}.
Indeed, the relevant morphisms are
{\footnotesize\begin{align*}
    &r_{A\bowtie  C}\ot \id_{A'\bowtie  C'}=(r_A\ot \id_{\degf{C}\trid A'})(\id_A\ot\gamma_0^{\degf{C}}\ot \id_{\degf{C}\trid A'})
    \bowtie (r_C\fiz \degf{A'}\ot\id_{C'})(((\bR^0_C)^{-1}\ot \id_C)\fiz \degf{A'}),\\
    &\id_{A\bowtie  C}\ot \ell_{A'\bowtie  C'}= (\id_A\ot |C|\trid \ell_{A'})(\id_A\ot |C|\trid(\id_\uno\otimes {(\mathbf{L}^0_{A'})}^{-1}))\bowtie \id_{C\fiz \degf{A'}} \ot \ell_{C'}(\eta_0^{\degf{A'}}\otimes \id_{C'}),\\
    &\alpha_{A\bowtie  C, \uno\bowtie \uno, A'\bowtie  C'} = \alpha_{A\bowtie  C, \uno\bowtie \uno, A'\bowtie  C'}^L\bowtie \alpha_{A\bowtie  C, \uno\bowtie \uno, A'\bowtie  C'}^R,\,\text{where}\\ 
    &\alpha_{A\bowtie  C, \uno\bowtie \uno, A'\bowtie  C'}^L =(\id_A \otimes {\gamma^{\degf{C}}_{\uno, e\trid A'}}^{-1})(\id_A\otimes\id_{\degf C\trid \uno}\otimes ({\mathbf{L}^2_{\degf{C}, e}})^{-1}_{A'})\alpha_{A, |C|\trid\uno, |C|\trid A'},\\
    &\alpha_{A\bowtie  C, \uno\bowtie \uno, A'\bowtie  C'}^R =\alpha_{(C\fiz |A'|)\fiz e, \uno\fiz |A'|, C'}((\mathbf{R}^2_{\degf{A'}, e})_C\otimes \id_{\uno\fiz \degf{A'}}\otimes \id_{C^{\prime}})(\eta^{\degf{A'}}_{C\fiz e, \uno}\otimes \id_{C^{\prime}}).
\end{align*}}
 For the triangle axiom, Diagram \eqref{eq:LI-lemma-delingue-product} translates to the one in Figure \ref{fig:triangle-lhs}. We have that $(i)$ commutes by the naturality of $\alpha_{A, -, \degf{C}\trid A'}$, $(ii)$ commutes trivially, $(iii)$ and $(v)$ commute by the naturality of $\ell_{-}$, $(iv)$ commutes by the triangle axiom in $\cA$, and the others diagrams are Equations \eqref{action-eq2} and \eqref{diagram:MP-13}. The analogous  of the Diagram \ref{eq:LD-lemma-delingue-product} commutes by a similar computation. The proof of the pentagon axiom is included only in the Appendix  of the Arxiv version.
\begin{figure}[htb]
	\centering
\begin{adjustbox}{width=\textwidth}
\begin{tikzcd}
     (A\ot \degf{C}\trid\uno)\ot \degf{C}\trid A' \ar[ddd,swap,"\id_A\ot\gamma_0^k\ot\id_{\degf{C}\trid A'}"]\ar[rr, "\alpha_{A,  \degf{C}\trid\uno, \degf{C}\trid A'}"] &  & \ar[dd,"\id_A\ot\id_{\degf{C}\trid\uno}\ot ({\mathbf L}^2_{\degf{C}, e})^{-1}_{A'}"]  A\ot (\degf{C}\trid\uno\ot \degf{C}\trid A')\ar[ld,swap, "\id_A\ot\gamma_0^{\degf{C}}\ot\id_{\degf{C}\trid A'}"] \\
     \hspace{.7in} (i) & A\ot (\uno\ot \degf{C}\trid A')\ar[ldddddd,swap,  "\id_A\ot \ell_{\degf{C}\trid A'}"] \ar[ddd,"\id_A\ot\id_{\uno}\ot({\mathbf L}^2_{\degf{C}, e})^{-1}"]  &\\
         &   &    A\ot (\degf{C}\trid \uno\ot \degf{C}\trid (e\trid A'))\ar[ldd,bend left,swap, "\id_A\ot\gamma^{\degf{C}}_0\ot \id_{\degf{C}\trid (e\trid A')}"]\ar[dd, "\id_A\ot (\gamma_{\uno, e\trid A'}^{\degf{C}})^{-1}"]\\
     (A\ot \uno)\ot \degf{C}\trid A'\ar[ruu, "\alpha_{A, \uno,\degf{C}\trid A'}"] \ar[dddd,swap, "r_A\ot\id_{\degf{C}\trid A'}"] & \hspace{-2in} (iv) & \hspace{-2in}(ii)\\
      & A\ot (\uno\ot \degf{C}\trid (e\trid A'))\ar[dd,"\id_A\ot \ell_{\degf{C}\trid (e\trid A')}"] & A\ot \degf{C}\trid (\uno\ot e\trid A')\ar[ddd, "\id_{A}\ot \degf{C}\trid (\id_{\uno}\ot ({\mathbf L}^0)_{A'}^{-1})"]\ar[ddl, bend left, swap, "\id_A\ot\degf{C}\trid \ell_{e\trid A'}"]\\
      & \hspace{-.7in}(v) & \hspace{-1.5in} \eqref{diagram:MP-13}\\
    \hspace{1.9in}\eqref{action-eq2} &   A\ot \degf{C}\trid (e\trid A')\ar[ld,"\id_A\ot \degf{C}\trid ({\mathbf L}^0_{A'})^{-1}"]  & \hspace{-.7in} (iii)\\
     A\ot \degf{C}\trid A'\ar[ru, bend left,pos=.8, "\id_A\ot ({\mathbf L}^2_{\degf{C}, e})_{A'}^{-1}"]  & & A\ot \degf{C}\trid (\uno\ot A').\ar[ll,"\id_A\ot \degf{C}\trid \ell_{A'}"]
\end{tikzcd}
\end{adjustbox}
\caption{}
	\label{fig:triangle-lhs}
\end{figure}

\end{proof}

For the rest of this subsection, we fix a bicrossed product $\cA \bowtie \cC$ between fusion categories $\cA$ and $\cC$. Consider the monoidal subcategories  $\widetilde{\cA}$ and $\widetilde{\cC}$ generated $A\bowtie \uno$ and $\uno\bowtie C$, for $A\in\Irr(\cA)$ and $C\in\Irr(\cC)$, respectively. The proof of the following equivalence is available only in the Appendix of the Arxiv version.

\begin{prop}\label{prop: subcategorias}
	The category $\widetilde{\cA}$ is monoidally equivalent to $\cA$ via the functor $(\bF, \bF^2)$, where $\bf F$ is defined by $\bF(A) = A\bowtie \uno$, for $A\in\cA$, $\bF(f) = f\bowtie \id_{\uno}$, for $f\colon A\to A'$, and ${\bf F}^2$ is defined by $\bF^2_{A,A'} = (\id_A \ot (\bL^0)^{-1}_{A'})\bowtie \ell_\uno (\eta_0^{\degf{A'}}\ot \id_\uno)$, for $A, A'\in \Irr(\cA)$. In a similar way, the category $\,\widetilde{\cC}$ is monoidally equivalent to $\cC$ via the functor $({\bf G}, {\bf G}^2)$, where ${\bf G}(C) = \uno\bowtie C$, for $C\in\cC$, ${\bf G}(g) = \id_\uno \bowtie g$, for $g\colon C\to C'$, and ${\bf G}^2$ is defined by ${\bf G}^2_{C,C'} = r_\uno(\id_\uno \ot \gamma_0^{\degf{C}}) \bowtie ((\bR^0)_C^{-1}\ot\id_{C'})$, for $C, C'\in\Irr(\cC)$.
\end{prop}

\begin{coro}\label{cor: rigid}
The category $\mathcal A \bowtie \mathcal C$  is rigid.
\end{coro}
\begin{proof}
    It is enough to prove only that the simple objects in the category have duals, and all simple objects are isomorphic to $(A\bowtie \uno) \ot (\uno \bowtie C)$, for $A\in \Irr(\cA)$, $C\in\Irr(\cC)$. Since the objects $A\bowtie \uno$ and $\uno\bowtie C$ belong to copies of the categories $\mathcal A$ and $\mathcal C$, respectively, by Proposition \ref{prop: subcategorias}, and these are rigid monoidal, thus the result follows from the fact that the product of objects with duals also have duals. 
\end{proof}

\begin{coro}\label{teo:bowtie}
The bicrossed product $\mathcal A \bowtie \mathcal C$  is a fusion category. Moreover,  $\mathcal A \bowtie \mathcal C$ is an exact factorization between $\mathcal{A}$ and $\mathcal{C}$.
\end{coro}
\begin{proof}
   The bicrossed product $\mathcal A \bowtie \mathcal C$  is a fusion category by the definition of $\mathcal{A}\bowtie \mathcal{C}$ as an abelian category, Theorem \ref{teo:bowtie-monoidal}, Proposition \ref{prop: subcategorias} and Corollary \ref{cor: rigid}. In addition, $\mathcal A \bowtie \mathcal C$ is an exact factorization, since as an abelian category is $\mathcal A \boxtimes \mathcal C$.
\end{proof}

\begin{coro} The Grothendieck ring of $\cA \bowtie \cC$ is  $K_0(\cA \bowtie \cC) = K_0(\mathcal{A})\bowtie K_0(\mathcal{C})$.
\end{coro}

The proof of the following proposition is given in the Appendix  of the Arxiv version. See \cite[Definition 4.7.7]{EGNO-book} for the definition of a pivotal structure on a rigid monoidal category.

\begin{prop}\label{prop:pivotal-bicrossed-product}
    If $p^\cA_A: A \to A^{**}$ and $p_C^{\cC}:C\to C^{**}$ are pivotal structures on $\cA$ and $\cC$ respectively, then the bicrossed product $\cA\bowtie\cC$ has a canonical pivotal structure given by the natural transformation
    \begin{equation*}
\begin{split}
    p^{\cA\bowtie\cC}_{A\bowtie C}  = & (r_{A^{**}}\bowtie \ell_{C^{**}})(\id_{A^{**}}\ot (\gamma_0^e)^{-1}\bowtie (\eta_0^e)^{-1}\ot\id_{C^{**}})\\
    & (p^\cA_A\ot e\trid \id_\uno \bowtie \id_\uno\fiz e\ot p^\cC_C)     (\id_A\ot \gamma_0^e\bowtie \eta^e_0\ot\id_C)(r_A^{-1}\bowtie \ell_C^{-1}).
    \end{split}
\end{equation*}
\end{prop}

Pivotal structures allow us to define two types of traces in the category, the left trace $\textrm{Tr}^L$ and the right trace $\textrm{Tr}^R$, see  \cite[\S 4.7]{EGNO-book} for details. When both of these traces coincide, the category is said to be \emph{spherical}.

\begin{prop}\label{prop:spherical-bicrossed-product}
        If $\cA$ and $\cC$ are spherical 
then $\cA\bowtie\cC$ with the induced canonical pivotal structure is spherical.
\end{prop}
\begin{proof}
    It is enough to check the equality of the traces in the simple objects. Let $f\ot g: A\ot C \to A\ot C$ be a morphism in $\cA\bowtie\cC$, where $A$ is a simple in $\cA$ and $C$ is a simple in $\cC$. Then $\textrm{Tr}^L(f\ot g)=\textrm{Tr}^L(f)\textrm{Tr}^L(g)=\textrm{Tr}^R(f)\textrm{Tr}^R(g)=\textrm{Tr}^R(f\ot g)$.
\end{proof}

 In the next proposition, we give a condition for the bicrossed product being equivalent to the Deligne tensor product. The proof is given in the Appendix of the Arxiv version.

\begin{prop}\label{prop:equiv-to-deligne}
Let $(\cA, \cC, H, K, \tridb, \fizb, \trid, \fiz)$ be a matched pair of fusion categories. If $K$ acts on $\cA$ by tensor autoequivalences, $H$ acts on $\cC$ by tensor autoequivalences, and there exist  $u_k: k\trid - \to\id_\cA$ and $v_h: -\fiz h\to \id_\cC$  monoidal natural isomorphisms, for all $k\in K, h\in H$, then $\bf F: \cA\bowtie\cC\to \cA \boxtimes\cC$ is a monoidal equivalence, where $F(A\boxtimes C)=A\bowtie C$ and $\bF^2_{A\boxtimes C, A'\boxtimes C'}=\id_A\otimes (u_{\degf{C}})_{A'}\bowtie (v_{\degf{A'}})_C\otimes \id_{C'}$.
\end{prop}\label{lemma:eq-deligne}

\begin{coro}
    Under the previous hypothesis, if $\cA$ and $\cC$ are braided, then $\cA\bowtie\cC$ is braided. 
\end{coro}

\begin{coro}
     The category $(\cA\bowtie\cC)_{\ad}$ is monoidally equivalent to $\cA_{\ad}\boxtimes\cC_{\ad}$.
\end{coro}
\begin{proof}
    The result follows from Lemma \ref{lemma:eq-deligne}, for $u_e=({\mathbf L}^ 0)^{-1}$ and $v_e=({\mathbf R}^0)^{-1}$, Equation \eqref{action-eq2} and Lemma \ref{diagrmas-e-Sonia}.
\end{proof} 

\begin{remark}\label{remark-clave-2}
    In the same spirit as Remark \ref{remark-clave}, the category $\cA\bowtie \cC$ is an extension of $(\cA\bowtie\cC)_{\ad} \simeq \cA_{\ad} \boxtimes \cC_{\ad}$ by the previous corollary, so if another exact factorization has the adjoint subcategory as the Deligne product and same fusion rules as the original one, then they would differ by twisting by a $3$-cocycle. A natural follow-up question is whether there are exact factorizations whose adjoint subcategory has an associativity constraint different from that of the Deligne tensor product.
\end{remark}

\section{Examples of fusion categories from bicrossed products}\label{section:examples-of-fusion-cat-from-bicrossed-product}
In this section, we study exact factorization between Tambara-Yamagami fusion categories and pointed fusion categories, see Subsections \ref{subsection:pointed-definition} and \ref{subsection:Tambara-Yamagami-definition} for the relevant definitions. We do this by classifying their possible fusion rings as bicrossed products, and then classifying their possible liftings of matched pair as in Definition \ref{def:lifting-matched-pair}.

\subsection{Bicrossed products between a Tambara-Yamagami fusion category and a pointed fusion category}
 Let $\mathcal{TY}(\Gamma,\chi,\tau)$ be a Tambara-Yamagami fusion category and $\vect_K^\omega$ be a pointed fusion category, with $K_0(\mathcal{TY}(\Gamma,\chi,\tau))=\mathsf{TY}(\Gamma)$ and $K_0(\vect_K^\omega) = \mathsf{R}(K)$, respectively. In this subsection, we study the category $\mathcal{B} = \mathcal{TY}(\Gamma,\chi,\tau) \bullet \vect_K^\omega$. 
 
  \begin{theorem}\label{teo:matched-pair-rings-TY-vecK}
  The fusion ring of $\mathcal{B}$ is completely characterized by
  \begin{enumerate}[leftmargin=*,label=\rm{(\roman*)}]
  	\item an automorphism $\phi$ of $K$ of order at most two, and 
  	\item an action $\varphi\colon K\to \operatorname{Aut}(\Gamma)$ by group automorphisms.
  \end{enumerate}
  This fusion ring is a bicrossed product of fusion rings $\mathsf{TY}(\Gamma) \bowtie \mathsf{R}(K)$ given by a matched pair $(\mathsf{TY}(\Gamma) , \mathsf{R}(K), C_2, K, \tridb, \fizb, \trid, \fiz)$ such that 
  \begin{itemize}
  	\item the action $\tridb$ is trivial and $\fizb$ is determined by $\phi$,
        \item the action $\trid$ fixes the non-invertible element $X$ and acts on $\Gamma$ by $\varphi$, and the action $\fiz$ is the same as $\fizb$ on basic elements.
  \end{itemize} The basic elements of $K_0(\mathcal{B})$ are $g\bowtie k$ and $X\bowtie k$, for $g\in\Gamma$, $k\in K$. The fusion rules are
  \begin{align}\label{eq:fusion-rules-TY-vecK}
  \begin{aligned}
  	(g\bowtie k)(g'\bowtie k')=g\varphi_k(g')\bowtie kk', && (g\bowtie k)(X\bowtie k')=X\bowtie \phi(k)k'\\
  	(X\bowtie k')(g\bowtie k)=X\bowtie k' k, && (X\bowtie k)(X\bowtie k')=\sum_{g\in\Gamma} g\bowtie \phi(k)k',
    \end{aligned}
  \end{align} for $g, g'\in \Gamma$, $k, k'\in K$.
  The universal grading group is $U(\cB) \simeq C_2 \rtimes_\phi K$ and the pointed subring is $(K_0(\cB))_{\pt} = \mathsf{R}(\Gamma \ltimes_\varphi K)$.
  \end{theorem}
  \begin{proof}
  The fusion ring of $\mathcal{B}$ is a bicrossed product of fusion rings by Theorem \ref{teo:char-exact-fact-fusion-ring}, that is,  $K_0(\mathcal{B}) = \mathsf{TY}(\Gamma)\bullet \mathsf{R}(K) \simeq \mathsf{TY}(\Gamma) \bowtie \mathsf{R}(K)$. 
  The universal grading groups of the relevant fusion rings (and fusion categories) are $U(\mathsf{TY}(\Gamma)) = U(\mathcal{TY}(\Gamma,\chi,\tau)) = C_2$, and $U(\mathsf{R}(K)) = U(\vect_K^\omega) = K.$
 We fix $\sigma\in C_2$ as the generator of the group. Specifically, the gradings are  $\degf{X} = \sigma$,  $\degf{g} = e$, and $\degf{k} = k$, for $g\in \Gamma, k\in K.$

 A matched pair of groups $(C_2, K, \tridb, \fizb)$ has necessarily $\tridb$ trivial, and the action $\fizb\colon K \times C_2 \to K$ is completely determined by an automorphism of groups $\phi\colon K\to K$ of order two, indeed $k \fizb \sigma = \phi(k)$, for $k\in K$. Thus $U(\cB) \simeq C_2 \rtimes_\phi K$.
 
Let us describe the $\mathbb{Z}$-linear actions $\trid\colon K \times \mathsf{TY}(\Gamma) \to \mathsf{TY}(\Gamma)$ and $\fiz\colon \mathsf{R}(K) \times C_2 \to \mathsf{R}(K)$. For $b\in\mathsf B(\mathsf{TY}(\Gamma))$, $k\in K$, from Definition \ref{def:mathced-pair-fusion-rings}\ref{def:mathced-pair-fusion-rings-2},  $|k\trid b|=k\tridb |b|=|b|$, then $k\trid X=X$ and $k\trid g=\varphi_k(g)$, for some $\varphi_k: \Gamma \to \Gamma$. From Definition \ref{def:mathced-pair-fusion-rings}\ref{def:mathced-pair-fusion-rings-3}, $\varphi_k$ is a group automorphism. From Definition \ref{def:mathced-pair-fusion-rings}\ref{def:mathced-pair-fusion-rings-2}, $|k\fiz \sigma|=|k|\fizb \sigma=k\fizb\sigma$, then $k\fiz\sigma=k\fizb \sigma$. From this, the fusion rules are clear and  $(K_0(\mathcal{B}))_{\pt} = \mathsf{R}( \Gamma\rtimes_\varphi K)$.
 \end{proof}
 
We characterize the categorical matched pairs between $\mathcal{TY}(\Gamma,\chi,\tau)$ and $\vect_K^\omega$ by lifting the matched pairs in Theorem \ref{teo:matched-pair-rings-TY-vecK}.

\begin{theorem}\label{teo:mp-fusion-TambaraY-pointed}
    Let $\mathcal{TY}(\Gamma,\chi,\tau)$ be a Tambara-Yamagami fusion category and $\vect_K^\omega$ be a pointed fusion category. Let $\phi$ and $\varphi$ be as in Theorem \ref{teo:matched-pair-rings-TY-vecK} characterizing a matched pair $(\mathsf{TY}(\Gamma) , \mathsf{R}(K), C_2, K, \tridb, \fizb, \trid, \fiz)$ between their corresponding fusion rings. Then a necessary condition for a lifting of the matched pair of the fusion ring to exist is that the bicharacter $\chi$ is invariant by $\varphi_k$ in each variable:
    \begin{align}\label{eq:invariant-chi-for-lifting}
        \chi(\varphi_k(g), g') = \chi(g, g') = \chi(g, \varphi_k(g')), && k\in K, g,g'\in\Gamma.    \end{align}
    When this happens, every lifting of the matched pair of fusion rings above is characterized by
    \begin{enumerate}[leftmargin=*,label=\rm{(\roman*)}]
    \item\label{item:cat-TY-pointed-1} a normalized 2-cocycle $\lambda_X\in H^2(K,\ku^\times)$,
    \item\label{item:cat-TY-pointed-2-prime} a character $\mu\colon K\to \ku^\times$,
    \item\label{item:cat-TY-pointed-2} a function $f\colon K \to \ku^\times$ such that 
    \begin{align}\label{eq:condition-for-map-f-in-TY-pointed}
        f(\phi(k)) = f(k) \frac{\mu(k)}{\mu(\phi(k))}, && f(k k') = f(k) f(k') \lambda_X(k,k') \lambda_X(\phi(k), \phi(k')), k\in K,
    \end{align}
    \item\label{item:cat-TY-pointed-3} a function $\beta\colon K\times K \to \ku^\times$ such that
    \begin{align*}
        \omega(k,k',k'') \beta(k,k' k'') \beta(k', k'') = \beta(k k', k'') \beta(k, k') \omega(\phi(k), \phi(k'), \phi(k'')),
    \end{align*}
    \item\label{item:cat-TY-pointed-4} a family of functions $\lambda_k\colon C_2\times C_2\to \ku^\times$ for $k\in K$ such that
    \begin{align}\label{eq:condition-for-map-R2-in-TY-pointed}
        \lambda_k(h h', h'') \lambda_k(h, h') = \lambda_k(h, h' h'') \lambda_{k\fizb h}(h',h''), \,\,h, h', h''\in C_2.
    \end{align}
    \end{enumerate}
     The maps $\gamma$, $\eta$, $\mathbf{L}^2$ and $\mathbf{R}^2$ that lift the matched pair in terms of these data are:
    \begin{align*}
        \gamma_{g,g'}^k &= \mu(k) \id_{\varphi_k(g g')}, & \gamma_{g,X}^k &= \mu(k)\id_{X}, & \gamma^k_{X,g} &= \mu(\phi(k)) \id_X \\
       \gamma_{X,X}^k &= f(k) \id_{X\ot X}, & (\mathbf{L}^2_{k,k'})_g &= \id_{\varphi_{k k'}(g)},  &(\mathbf{L}^2_{k,k'})_X &= \lambda_X(k,k') \id_{X}, \\
         (\mathbf{R}^2_{h,h'})_k &= \lambda_k(h,h')\id_{k\fiz h h'}, & \eta^\sigma_{k,k'} &= \beta(k,k') \id_{\phi(k k')}, &\eta^e_{k,k'} &= \id_{kk'},
    \end{align*} 
for all $g,g'\in\Gamma$, $k,k'\in K$, $h,h'\in C_2$. 
\end{theorem}
 \begin{proof}

First, we check the conditions on $\chi$.  By Axiom \eqref{eq:gamma-alpha-diagram}, taking $A = X$, $A' = g$, and $A'' = X$, we get $\mu(\phi(k))\chi(g', g) = \chi(g', \varphi_{\phi(k)}(g)) \mu(\phi(k))$, for all $k\in K, g, g'\in \Gamma$.
Hence, $\chi$ is invariant under $\varphi_k$ in the second variable. Now, by taking $A = g$, $A' = X$, and $A'' = g'$, we get 
$   \chi(\varphi_k(g), \varphi_{\phi(k)}(g')) = \chi(g, g')$, for all $k\in K, g, g'\in \Gamma$, which means $\chi$ is also invariant under $\varphi_k$ in the first variable since $\chi$ is invariant under $\varphi_k$ in the second variable.

In order to lift the group actions $\trid$ and $\fiz$, notice that the corresponding maps are characterized by scalars $(\mathbf{L}^2_{k,k'})_A = \lambda_A(k,k') \id_{(k k')\trid A}$, for $k,k'\in K$, $A\in\Irr(\mathcal{TY}(\Gamma,\chi,\tau))$ (respectively,  $(\mathbf{R}^2_{h,h'})_k = \lambda_k(h,h') \id_{k\fiz (h h')}$, for $h,h'\in C_2$, $k\in\Irr(\vect_K^\omega) = K$ ). The family $\lambda_k$, for $k\in K$, satisfies Equation \eqref{eq:condition-for-map-R2-in-TY-pointed} due to Diagram \eqref{action-eq2}. The scalars $\lambda_A$, for $A\in\Irr(\mathcal{TY}(\Gamma,\chi,\tau))$, are 2-cocyles of $K$ by Diagram \eqref{action-eq1} and because $\fizb$ is trivial.  We will show later that we can choose $\lambda_g = 1$, for all $g\in\Gamma$, so only $\lambda_X$ remains non-trivial, corresponding to item \ref{item:cat-TY-pointed-1}. Since the action $\tridb$ is trivial, a lifting of $\fiz$ gives a $C_2$-action by tensor autoequivalences of $\vect_K^\omega$, with tensor structure $\vartheta^h_{k,k'} = (\eta^h_{k,k'})^{-1}$ as in Subsection \ref{subsec: cat action}. By \cite[\S 2.6]{EGNO-book}, the action $\fiz$ is characterized by a function $\beta\colon K\times K \to \ku^\times$ as in item \ref{item:cat-TY-pointed-3}. The possible maps $\gamma$ are given by  
 {\small \begin{align*}
     \gamma_{g,g'}^k &= \widetilde{\mu}(k,g,g') \id_{\varphi_k(gg')}, & \gamma_{X, g}^k &= \theta(k,g) \id_X, & \gamma_{g, X}^k &= \widetilde{\theta}(k,g) \id_X,   & \gamma_{X,X}^k = \bigoplus_{g\in\Gamma} \widetilde{f}(k,g) \id_g,
 \end{align*}} 
 
 \noindent for $g,g'\in\Gamma$, and $k\in K$, and functions $\widetilde{\mu}\colon K\times  \Gamma\times \Gamma \to \ku^\times$, $\theta\colon K \times \Gamma \to \ku^\times$,  $\widetilde{\theta}\colon K \times \Gamma \to \ku^\times$, and $\widetilde{f}\colon K\times \Gamma \to \ku^\times$.  
 By Axiom \eqref{eq:gamma-alpha-diagram}, taking $A = X$, $A' = X$, and $A'' = g$, we get $\theta(\phi(k), g) = \widetilde{\mu}(k,g',g)$, for all $g, g'\in \Gamma$, $k\in K$. So $\widetilde{\mu}(k,g', g)$ does not depends on $g'$. That is,  $\widetilde{\mu}(k,g',g) = \widetilde{\mu}(k,e,g) = \theta(\phi(k),g)$,  for all $k\in K, g\in \Gamma$. By taking $A = g$, $A' = X$, and $A'' = X$, we get $\widetilde{\theta}(k, g') = \widetilde{\mu}(k,g',g)$, for all $g, g'\in \Gamma$, $k\in K$. So $\widetilde{\mu}(k,g', g)$ does not depends on $g$. In particular,  $\widetilde{\mu}(k,g',g) = \widetilde{\mu}(k,g',e)$ for all $k\in K$ and $g'\in \Gamma$. Let $\mu\colon K\to \ku^\times$ defined as $\mu(k) = \widetilde{\mu}(k,e,e)$. Therefore $ \widetilde{\mu}(k,g,g') = \mu(k)$,  $\theta(k,g) = \mu(\phi(k))$, $\widetilde{\theta}(k,g) = \mu(k)$,  for all $k\in K, g,g'\in\Gamma$.  By Axiom \eqref{eq:gamma-T2-diagram}, taking $A = g$, $A' = X$, we get that 
\begin{align}\label{eq:thetatilde-mult-k}
\mu(kk') = \mu(k) \mu(k')\lambda_g(k,k'),
\end{align}
which means that $\lambda_g = d(1/\mu)$. Then, $\lambda_g$ is cohomologically trivial, for all $g\in\Gamma$, and we can choose $\lambda_g(k,k') = 1$, for all $g\in \Gamma$, and $\mu$ a character of $K$, by Equation \eqref{eq:thetatilde-mult-k}.

By Axiom \eqref{eq:gamma-alpha-diagram}, taking $A=A'=A''=X$, we get $\widetilde{f}(k,g) \mu(k) = \widetilde{f}(\phi(k),g')\mu(\phi(k))$, for all $k\in K$, $g,g'\in \Gamma$.
In particular,
\begin{align*}
    \widetilde{f}(k,g) = \frac{\widetilde{f}(\phi(k),e) \mu(\phi(k))}{\mu(k)} = \frac{\widetilde{f}(k,e) \mu(k) \mu(\phi(k))}{\mu(k) \mu(\phi(k))} = \widetilde{f}(k,e),
\end{align*}
for all $k\in K$, $g\in \Gamma$. Hence, $\widetilde{f}$ does not depend of $g\in\Gamma$. Let $f\colon K\to \ku^\times$ defined as $f(k) = \widetilde{f}(k,e)$. Then $\widetilde{f}(k,g) = f(k)$, for all $g\in\Gamma$. 
By Axiom \eqref{eq:gamma-alpha-diagram}, taking $A=A'=A''=X$, and by Axiom \eqref{eq:gamma-T2-diagram}, taking $A = X$, $A' = X$, we get that $f$ satisfies Equations \eqref{eq:condition-for-map-f-in-TY-pointed}. 

The commutativity of all other diagrams follows from the conditions that we already obtained.
 \end{proof} 
 We can explicitly describe the associators in terms of the data characterized before for the fusion ring $\mathsf{TY}(\Gamma) \bowtie \mathsf{R}(K)$ with fusion rules \eqref{eq:fusion-rules-TY-vecK}.
 Indeed,
\begin{align*}
    \alpha_{g\bowtie k, g'\bowtie k', g''\bowtie k''} & =  \mu^{-1}(k)\omega(k,k',k'')\lambda_k(e,e) \id_{g\varphi_k( g^\prime)\varphi_{kk'}(g'')}\bowtie\id_{kk^\prime k^{\prime\prime}},\\
\alpha_{g\bowtie k, g'\bowtie k', X\bowtie k''} & = \mu^{-1}(k)\lambda_X^{-1}(k,k')\omega(\phi(k), \phi(k'), k'')\beta(k,k')\lambda_k(e,\sigma)\id_X\bowtie\id_{\phi(k)\phi(k')k''},\\
\alpha_{g\bowtie k, X\bowtie k', g'\bowtie k''} & = \mu^{-1}(\phi(k))\omega(\phi(k), k', k'')\chi(g,\varphi_{\phi(k)k'}(g')) \lambda_k(\sigma,e) \id_X\bowtie \id_{\phi(k)k'k''},\\
\alpha_{X\bowtie k, g\bowtie k', g'\bowtie k''} & =\omega(k, k', k'')\mu^{-1}(k)\lambda_k(e,e)\id_X\bowtie \id_{kk'k''},\\
\alpha_{g\bowtie k, X\bowtie k', X\bowtie k''} &= \omega(k, \phi(k'), k''){f(k)}^{-1}\beta(\phi(k), k')\lambda^{-1}_X(\phi(k),k')\lambda_k(\sigma,\sigma)\id_{X\otimes X}\bowtie  \id_{k\phi(k')k''},\\
\alpha_{X\bowtie k, X\bowtie k', g\bowtie k''} &= \omega(\phi(k), k', k'')\mu^{-1}(\phi(k))\lambda_k(\sigma,e) \id_{X \otimes  X} \bowtie \id_{\phi(k)k'k''},\\
\alpha_{X\bowtie k, g\bowtie k', X\bowtie k''} &= \bigoplus_{g'\in\Gamma}\omega(\phi(k), \phi(k'), k'')\mu^{-1}(k)\lambda_X^{-1}(k,k')\chi(\varphi_k(g),g')\beta(k,k')\lambda_k(e,\sigma)\\
&\id_{X \otimes X}\ot \id_{g'}\bowtie  \id_{\phi(k)\phi(k')k''},
\end{align*}
    \begin{align*}
\alpha_{X\bowtie k, X\bowtie k', X\bowtie k''} & \colon \oplus_{g\in \Gamma}  X\bowtie k\phi(k')k''\to \oplus_{g\in \Gamma} X\bowtie k\phi(k')k'' \\
& = \left(\frac{\tau \omega(k, \phi(k'), k'')\lambda_X^{-1}(\phi(k), k')f(k)^{-1}\beta(\phi(k), k')\lambda_k(\sigma,\sigma)}{\chi(g,h)}\right)_{g,h}.
\end{align*}

\subsection{Bicrossed products between two Tambara-Yamagami fusion categories}
Let $\mathcal {TY}(H, \chi, \tau)$ and $\mathcal{TY}(K, \zeta, \upsilon)$ be Tambara-Yamagami fusion categories, with Grothendieck rings $K_0(\mathcal{TY}(H, \chi, \tau))$ $=\mathsf{TY}(H)$ and $K_0(\mathcal{TY}(K, \zeta, \upsilon))=\mathsf{TY}(K)$, respectively. In this subsection, we study $\mathcal{B}=\mathcal{TY}(H, \chi, \tau)\bullet \mathcal{TY}(K, \zeta, \upsilon)$.

\begin{theorem}\label{teo:matched-pair-rings-TY-TY}
  The fusion ring of $\mathcal{B}$ is completely characterized by the group automorphisms $\varphi: H \to H$ and $\psi: K \to K$ of order at most two.
  This fusion ring is a bicrossed product of fusion rings $\mathsf{TY}(H) \bowtie \mathsf{TY}(K)$ given by a matched pair $(\mathsf{TY}(H) , \mathsf{TY}(K), C_2, C_2, \tridb, \fizb, \trid, \fiz)$ such that 
  \begin{itemize}
  	\item the action $\tridb$  and $\fizb$  are trivial, and 
        \item the action $\trid$ fixes the non-invertible element $X$ and acts on $H$ by $\varphi$, and the action $\fiz$ fixes the non-invertible element $Y$ and acts on $K$ by $\psi$.
  \end{itemize} The basic elements of $K_0(\mathcal{B})$ are $h\bowtie k$, $X\bowtie k$, $h\bowtie Y$, and $X\bowtie Y$, for $k\in K$, $h\in H$. The fusion rules are
\begin{align}\label{eq:fusion-rules-TY-TY}
\begin{aligned}
        (h\bowtie k)(h'\bowtie k')&= h h' \bowtie kk', & (h\bowtie k)(X\bowtie k')&= X \bowtie \psi(k) k',\\
    (X\bowtie k)(h\bowtie k')&= X  \bowtie kk',& (h\bowtie k)(h'\bowtie Y)&= hh' \bowtie Y,
\end{aligned}
\end{align}
\begin{align*}
    (h\bowtie Y)(h'\bowtie k)&= h\varphi(h')\bowtie Y, & (X\bowtie k)(X\bowtie k')&= \sum_{h\in H} h \bowtie \psi(k)k', \\
     (h\bowtie Y)(h'\bowtie Y)&=\sum_{k\in K} h\varphi(h') \bowtie  k, & (X\bowtie Y)(X\bowtie Y)&= \sum_{k\in K, h\in H} h \bowtie k,
\end{align*}
\begin{align*}
 (X\bowtie Y)(X\bowtie k) &= \sum_{h\in H}h \bowtie Y  =(X\bowtie k)(X\bowtie Y), \\
 (X\bowtie Y)(h\bowtie Y) &= \sum_{k\in K} X \bowtie k  = (h\bowtie Y)(X\bowtie Y), \\
 (X\bowtie Y)(h\bowtie k)&= X \bowtie Y = (h\bowtie k)(X\bowtie Y),\\
 (X\bowtie k)(h\bowtie Y)&= X \bowtie Y = (h\bowtie Y)(X\bowtie k),
\end{align*} for all $k, k'\in K$, $h, h'\in H$. The universal grading group is $U(\cB) \simeq C_2\times C_2$ and the pointed subring is $(K_0(\cB))_{\pt} = \mathsf{R}(H\times K)$.
  \end{theorem}
\begin{proof}
The universal grading groups of the relevant fusion rings are $U(\mathsf{TY}(H)) =  U(\mathsf{TY}(K))$ $ = C_2.$
 We fix $\sigma\in C_2$ as the generator of the group. Specifically,  the gradings are
 $\degf{X} = \degf{Y} = \sigma$, and $\degf{h} = \degf{k} = e$,  for all $h\in H, k\in K$. It is clear that any matched pair $(C_2, C_2, \tridb, \fizb)$ gives rise to the exact factorization $C_2\bullet C_2\simeq C_2\times C_2$. Hence, the actions $\tridb \colon C_2\times C_2 \to C_2$ and $\fizb \colon C_2\times C_2 \to C_2$ must be trivial. 
 
 To describe the $\mathbb{Z}$-linear actions $\trid\colon C_2 \times \mathsf{TY}(H) \to \mathsf{TY}(H)$ and $\fiz\colon  \mathsf{TY}(K)\times C_2 \to \mathsf{TY}(K)$, notice that by Definition \ref{def:mathced-pair-fusion-rings}\ref{def:mathced-pair-fusion-rings-2}, $|\sigma\trid b|=\sigma\tridb |b|=|b|$, for $b\in\mathsf B(\mathsf{TY}(H))$ and for $h\in H$, and then $\sigma\trid X=X$ and $\sigma\trid h=\varphi(h)$, with $\varphi: H\to H$. It follows from Definition \ref{def:mathced-pair-fusion-rings}\ref{def:mathced-pair-fusion-rings-3} that $\varphi$ is a group automorphism of order at most two. Similarly, $Y\fiz \sigma =Y$ and $k\fiz \sigma =\psi(k)$, for $k\in H$, where $\psi: K\to K$ is a group automorphism of order at most two. The fusion rules are clear and $(K_0(\mathcal{B}))_{\pt} = \mathsf{R}(H\times K)$.
\end{proof}

We construct a categorical matched pair between $\mathcal{TY}(H,\chi,\tau)$ and $\mathcal{TY}(K,\zeta,\upsilon)$  by lifting the matched pair at the level of fusion rings in Theorem \ref{teo:matched-pair-rings-TY-TY}.

\begin{theorem}\label{teo:mp-fusion-TY-TY}
    Let $\mathcal{TY}(H,\chi,\tau)$ and $\mathcal{TY}(K,\zeta,\upsilon)$ be  Tambara-Yamagami fusion categories. Let $\psi$ and $\varphi$ as in Theorem \ref{teo:matched-pair-rings-TY-TY} characterizing a matched pair $(\mathsf{TY}(H), \mathsf{TY}(K),$ $ C_2, C_2, \tridb, \fizb, \trid, \fiz)$ between their corresponding fusion rings. Then a lifting exists if and only if $\chi$ and $\zeta$ are invariant in each variable by $\varphi$ and $\psi$, respectively. 
    In that case, every lifting of the matched pair of fusion rings above is characterized by a pair $(\theta_\ell, \theta_r)$ of characters of $C_2$.  In particular, $\theta_\ell$ and $\theta_r$ are either trivial or $\sgn\colon C_2 \to \ku^\times$.
    
     The maps $\gamma$, $\eta$, $\mathbf{L}^2$, and $\mathbf{R}^2$ that lift the matched pair in terms of this data are 
    \begin{align*}
          &  \gamma_{X,X}^g = \theta_\ell(g) \id_{X\ot X}, &  \eta_{Y,Y}^g = \theta_r(g) \id_{Y\ot Y},
    \end{align*} 
for all  $g\in C_2$, and identity in all the other cases. 
\end{theorem}
\begin{proof}
    First, we show that the conditions on $\chi$ and $\zeta$ are necessary. Indeed, assume a lifting exists. Then, by Axiom \eqref{eq:gamma-alpha-diagram}, taking $A = h$, $A' = X$, and $A'' = h'$, we get $\chi(h, h') = \chi(s\trid h, s\trid h')$,  for all $h, h'\in H, s\in C_2$.
Hence $\chi$ is invariant under $\varphi$ in both variables simultaneously. By taking $A = X$, $A' = h$, and $A'' = X$ in the same diagram, we get $\chi(h',h)=\chi( h', g\trid h)$,  for all $g\in C_2$.
Henceforth, $\chi$ is invariant under $\varphi$ in both variables separately. A similar argument can be made for $\zeta$. 

To show that this condition on $\chi$ and $\zeta$ is sufficient for the existence of a lifting, we will construct $\gamma$, $\eta$, $\bL^2$, and $\bR^2$. First, we lift the group actions $\trid$ and $\fiz$. In fact, the corresponding maps are characterized by scalars $(\mathbf{L}^2_{g,g'})_A = \lambda_A(g,g') \id_{(gg')\trid A}$, for $g,g'\in C_2$, $A\in\Irr(\mathcal{TY}(H,\chi,\tau))$ (respectively  
 $(\mathbf{R}^2_{g,g'})_C = \lambda_C(g,g') \id_{C\fiz (g g')}$, for $g,g'\in C_2$, $C\in\Irr(\mathcal{TY}(K,\zeta,\upsilon))$ ). These scalars are 2-cocycles on $C_2$ because the actions $\tridb$ and $\fizb$ are trivial and by Equations \eqref{action-eq1} and \eqref{action-eq2}. Since $C_2$ has trivial second cohomology group, we can assume that $\lambda_A(g,g') = 1=\lambda_C(g,g')$, for all  $A\in\Irr(\mathcal{TY}(H,\chi,\tau))$, $C\in\Irr(\mathcal{TY}(K,\zeta,\upsilon))$, $g,g'\in C_2$. Since the actions $\tridb$ and $\fizb$ are trivial, then a lifting of $\fiz$ and $\trid$ gives a $C_2$-action by tensor autoequivalences of $\mathcal{TY}(K,\zeta,\upsilon)$ and $\mathcal{TY}(K,\zeta,\upsilon)$, respectively. 

Now, we consider the maps $\gamma$. For each $g\in C_2$, let $\mu_g\colon H\times H \to \ku^\times$, $\rho\colon C_2 \times H \to \ku^\times$,  $\widetilde{\rho}\colon K \times C_2 \to \ku^\times$ and $f\colon C_2 \times H\to \ku^\times$ such that
 \begin{align*}
     \gamma_{h,h'}^g &= \mu_g(h,h') \id_{g\trid (hh')}, & \gamma_{X, h}^g &=  \rho(g,h)\id_X,\\
     \gamma_{h, X}^g &=  \widetilde{\rho}(g,h)\id_X, & \gamma_{X, X}^g&= \sum_{h\in H} f(g,h) \id_h,     
 \end{align*}
for $h,h'\in H, g\in C_2$.
By Axiom \eqref{eq:gamma-alpha-diagram}, taking $A = h$, $A' = h'$, and $A'' = h''$, we get that 
\begin{align*}
		\mu_g(h',h'')\mu_g(h, h' h'')=\mu_g(h,h')\mu_g(h h', h''), \,\,\,g\in C_2,
\end{align*} which means that $\mu_g$ is a 2-cocycle in $H$, for each $g\in C_2$. By Axiom \eqref{eq:gamma-alpha-diagram}, 
 taking $A = X$, $A' = X$, and $A'' = h$, we get that $\mu_g (h,h')=\rho(g,h)$, for all $g\in C_2, h'\in H$. In particular, $ \mu_g (h,h')=\rho(g,h) = \mu_g (h,e),$
 for all $g\in C_2$ and $h,h'\in H$. Since we can choose $\mu_g$ as a normalized 2-cocycle  in $H$, then $\mu_g=1$ for all $g\in C_2$. Then,  $\rho(g,h)=1$, for all $g\in C_2$, $h, h'\in H$.  By a similar argument for $A = h$, $A' = X$, and $A'' = X$, we get that $\widetilde{\rho}(g,h)=\mu_g(h,h')=1$, for all $g\in C_2, h'\in H$.

Let us analyze $f$. By Axiom \eqref{eq:gamma-alpha-diagram}, with $A=X$, $A'=X$ and $A''=X$, we get that $f(g,h)=f(g,h')$, for all $g\in C_2, h, h'\in H$. Then, $f$ does not depends on $H$, and we can write $f(g,h)=\theta_\ell(g)$, for some $\theta_\ell: C_2\to \ku^\times$. By Axiom \eqref{eq:gamma-T2-diagram}, with $A=X$ and $A'=X$, we get that $\theta_\ell(gg')=\theta_\ell(g)\theta_\ell(g')$, that is,  $\theta_\ell$ is a character of $C_2$. 

The conditions we already have imply that all the other diagrams for $\gamma$ commute. Similar arguments give the formulas for $\eta$, and the only nontrivial map is $\eta_{Y,Y}^g$, which is completely determined by a character $\theta_r\colon C_2 \to \ku^\times$.  
\end{proof}

 We can explicitly describe the associators in terms of the data for the fusion ring $\mathsf{TY}(H) \bowtie \mathsf{TY}(K)$ with the fusion rules given by \eqref{eq:fusion-rules-TY-TY}. The non-identity associativity natural isomorphisms are:
 {\footnotesize
 \begin{align*}
     \alpha_{h\bowtie Y, h'\bowtie Y, X\bowtie k} &= \theta_r(\sigma) \id_{X\bowtie (Y\ot Y)}, & \alpha_{h\bowtie Y, X\bowtie Y, X\bowtie k} &= \theta_\ell(\sigma) \id_{(X\ot X)\bowtie Y},\\
    \alpha_{h\bowtie Y, X\bowtie Y, X\bowtie k} &= \theta_r(\sigma)\theta_\ell(\sigma) \id_{(X\ot X)\bowtie (Y\ot Y)}, & \alpha_{h\bowtie k, X\bowtie k' , h'\bowtie k''} &= \chi(h,h') \id_{X\bowtie \psi(k) k' k''},\\
    \alpha_{h\bowtie k, X\bowtie k', h'\bowtie Y} &= \chi(h,h')\id_{X\bowtie Y}, & \alpha_{h\bowtie Y, X\bowtie k, h'\bowtie k'} &= \chi(h,h') \id_{X\bowtie Y},\\
    \alpha_{h\bowtie Y, X\bowtie Y, h'\bowtie k} &= \chi(h,h') \id_{X\bowtie (Y\ot Y)}, & \alpha_{h\bowtie k, X\bowtie Y, h'\bowtie Y} &= \chi(h,h') \id_{X\bowtie (Y\ot Y)},
    \end{align*}
    \begin{align*}
    \alpha_{h\bowtie k, h'\bowtie Y, h''\bowtie k'} &= \zeta(k,k') \id_{h h' \varphi(h'') \bowtie Y}, & \alpha_{h\bowtie k, h'\bowtie Y, X\bowtie k'} &= \zeta(k,k') \id_{X\bowtie Y},\\
    \alpha_{X\bowtie k, h\bowtie Y, h'\bowtie k'} &= \zeta(k,k') \id_{X\bowtie Y}, & \alpha_{X\bowtie k, X\bowtie Y, h\bowtie k'} &= \zeta(k,k') \id_{(X\ot X)\bowtie Y},\\
    \alpha_{h\bowtie k, X\bowtie Y, X\bowtie k'} &= \zeta(k,k') \id_{(X\ot X)\bowtie Y}, & \alpha_{h\bowtie k, X\bowtie Y, h'\bowtie k'} &= \chi(h,h')\zeta(k,k')\id_{X\bowtie Y},\\
    \alpha_{X\bowtie k, h\bowtie k', X\bowtie k''} &= \bigoplus_{h'\in H} \chi(h,h') \id_{h'\bowtie \psi(k k') k''}, & \alpha_{X\bowtie k, h\bowtie k', X\bowtie Y} &= \bigoplus_{h'\in H} \chi(h,h') \id_{h'\bowtie Y},\\
    \alpha_{X\bowtie Y, h\bowtie k, X\bowtie k'} &= \bigoplus_{h'\in H} \chi(h,h') \id_{h'\bowtie Y}, & \alpha_{X\bowtie Y, h\bowtie Y, X\bowtie k} &= \bigoplus_{h'\in H} \theta_r(\sigma)\chi(h,h') \id_{h'\bowtie(Y\ot Y)},\\
    \alpha_{X\bowtie k, h\bowtie Y, X\bowtie Y} &= \bigoplus_{h'\in H} \chi(h,h') \id_{h'\bowtie (Y\ot Y)}, & \alpha_{h\bowtie Y, h'\bowtie k, h''\bowtie Y} &= \bigoplus_{k'\in K} \zeta(k,k') \id_{h\varphi(h' h'')\bowtie k'},\\
    \alpha_{h\bowtie Y, h'\bowtie k, X\bowtie Y} &= \bigoplus_{k'\in K} \zeta(k,k') \id_{X\bowtie k'}, & \alpha_{X\bowtie Y, h\bowtie k, h'\bowtie Y} &= \bigoplus_{k'\in K} \zeta(k,k') \id_{X\bowtie k'},\\
    \alpha_{X\bowtie Y, X\bowtie k, h\bowtie Y} &= \bigoplus_{k'\in K} \zeta(k,k') \id_{(X\ot X)\bowtie k'}, & \alpha_{h\bowtie Y, X\bowtie k, X\bowtie Y} &= \bigoplus_{k'\in K} \theta_\ell(\sigma)\zeta(k,k') \id_{(X\ot X)\bowtie k'},\\
    \alpha_{X\bowtie k,h\bowtie Y,X\bowtie k'} &= \bigoplus_{h'\in H} \chi(h,h')\zeta(k,k') \id_{h'\bowtie Y}, &  \alpha_{h\bowtie Y, X\bowtie k, h'\bowtie Y} &= \bigoplus_{k'\in K} \chi(h,h') \zeta(k,k') \id_{X\bowtie k'},
    \end{align*}
    \begin{align*}
    &\alpha_{X\bowtie Y, h\bowtie k, X\bowtie Y} = \bigoplus_{h'\in H, k'\in K} \chi(h,h') \zeta(k,k') \id_{h'\bowtie k'},    \\
        \alpha_{X\bowtie k, X\bowtie k', X\bowtie k''} &= \left(\frac{\tau}{\chi(h,h')}\right)_{h,h'\in H} \bowtie \id_{k\psi(k')k''} \colon \left(\bigoplus_{h\in H} X\right)\bowtie k \psi(k') k'' \to \left(\bigoplus_{h\in H} X\right)\bowtie k \psi(k') k'',\\
        \alpha_{X\bowtie k, X\bowtie k', X\bowtie Y} &= \left(\frac{\tau}{\chi(h,h')}\right)_{h,h'\in H} \bowtie \id_Y \colon \left(\bigoplus_{h\in H} X\right)\bowtie Y \to \left(\bigoplus_{h\in H} X\right)\bowtie Y,\\
        \alpha_{X\bowtie Y, X\bowtie k, X\bowtie k'} &= \left(\frac{\tau \theta_\ell(\sigma)}{\chi(h,h')}\right)_{h,h'\in H} \bowtie \id_Y \colon \left(\bigoplus_{h\in H} X\right)\bowtie Y \to \left(\bigoplus_{h\in H} X\right)\bowtie Y,\\
        \alpha_{X\bowtie Y, X\bowtie Y, X\bowtie k} &= \left(\frac{\tau \theta_\ell(\sigma) \theta_r(\sigma)}{\chi(h,h')}\right)_{h,h'\in H} \bowtie \id_{Y\ot Y}\colon \left(\bigoplus_{h\in H} X\right)\bowtie (Y\ot Y) \to \left(\bigoplus_{h\in H} X\right)\bowtie (Y\ot Y),\\
        \alpha_{X\bowtie k, X\bowtie Y, X\bowtie Y} &= \left(\frac{\tau}{\chi(h,h')}\right)_{h,h'\in H} \bowtie \id_{Y\ot Y}\colon \left(\bigoplus_{h\in H} X\right)\bowtie (Y\ot Y) \to \left(\bigoplus_{h\in H} X\right)\bowtie (Y\ot Y),\\
        \alpha_{X\bowtie k, X\bowtie Y, X\bowtie k'} &= \left(\frac{\tau \zeta(k,k') \theta_\ell(\sigma)}{\chi(h,h')}\right)_{h,h'\in H} \bowtie \id_{Y}\colon \left(\bigoplus_{h\in H} X\right)\bowtie Y \to \left(\bigoplus_{h\in H} X\right)\bowtie Y,\\
        \alpha_{X\bowtie Y, X\bowtie k, X\bowtie Y} &= \bigoplus_{k'\in K}\left(\frac{\tau \zeta(k,k') \theta_\ell(\sigma)}{\chi(h,h')}\right)_{h,h'\in H} \bowtie \id_{k'}\colon \bigoplus_{k'\in K}\left(\bigoplus_{h\in H} X\right)\bowtie k' \to \bigoplus_{k'\in K}\left(\bigoplus_{h\in H} X\right)\bowtie k',\\
        \alpha_{h\bowtie Y, h'\bowtie Y, h''\bowtie Y}&= \id_{h\varphi(h')h''} \bowtie \left(\frac{v}{\zeta(k,k')}\right)_{k,k'\in K}\colon h\varphi(h')h''\bowtie \left(\bigoplus_{k\in K} Y\right) \to h\varphi(h')h''\bowtie \left(\bigoplus_{k\in K} Y\right),\\
        \alpha_{h\bowtie Y, h'\bowtie Y, X\bowtie Y}&= \id_{X} \bowtie \left(\frac{v \theta_r(\sigma)}{\zeta(k,k')}\right)_{k,k'\in K}\colon X\bowtie \left(\bigoplus_{k\in K} Y\right) \to X\bowtie \left(\bigoplus_{k\in K} Y\right),\\
        \alpha_{X\bowtie Y, h\bowtie Y, h'\bowtie Y}&= \id_{X} \bowtie \left(\frac{v}{\zeta(k,k')}\right)_{k,k'\in K}\colon X\bowtie \left(\bigoplus_{k\in K} Y\right) \to X\bowtie \left(\bigoplus_{k\in K} Y\right),\\
        \alpha_{X\bowtie Y, X\bowtie Y, h\bowtie Y}&= \id_{X\ot X} \bowtie \left(\frac{v\theta_\ell(\sigma)}{\zeta(k,k')}\right)_{k,k'\in K}\colon (X\ot X)\bowtie \left(\bigoplus_{k\in K} Y\right) \to (X\ot X)\bowtie \left(\bigoplus_{k\in K} Y\right),
        \end{align*}
    \begin{align*}
        \alpha_{h\bowtie Y, X\bowtie Y, X\bowtie Y}&= \id_{X\ot X} \bowtie \left(\frac{v \theta_r(\sigma)}{\zeta(k,k')}\right)_{k,k'\in K}\colon (X\ot X)\bowtie \left(\bigoplus_{k\in K} Y\right) \to (X\ot X)\bowtie \left(\bigoplus_{k\in K} Y\right),\\
        \alpha_{h\bowtie Y, X\bowtie Y, h'\bowtie Y}&= \id_{X} \bowtie \left(\frac{v \chi(h,h')}{\zeta(k,k')}\right)_{k,k'\in K}\colon X\bowtie \left(\bigoplus_{k\in K} Y\right) \to X\bowtie \left(\bigoplus_{k\in K} Y\right),\\
        \alpha_{X\bowtie Y, h\bowtie Y, X\bowtie Y}&= \bigoplus_{h'\in H} \id_{h'} \bowtie \left(\frac{v \theta_r(\sigma) \chi(h,h')}{\zeta(k,k')}\right)_{k,k'\in K}\colon \bigoplus_{h'\in H} h'\bowtie \left(\bigoplus_{k\in K} Y\right) \to \bigoplus_{h'\in H} h'\bowtie \left(\bigoplus_{k\in K} Y\right),\\
        \alpha_{X\bowtie Y, X\bowtie Y, X\bowtie Y} &= \left(\frac{v \tau \theta_\ell(\sigma) \theta_r(\sigma)}{\chi(h,h')\zeta(k,k')}\right)_{(h,k), (h',k')\in H\times K} \colon \bigoplus_{(h,k)\in H\times K} X\bowtie Y \to \bigoplus_{(h,k)\in H\times K} X\bowtie Y. 
    \end{align*}}
    
\section*{Appendix: Proofs of some results in Section \ref{section:exact-factorization-fusion-categories}(preprint version only)}\label{ap:diagramas-largos} 

In this section, we include the proofs of Lemma \ref{diagrmas-e-Sonia}, Proposition \ref{prop: subcategorias}, and Proposition \ref{prop:equiv-to-deligne}. In addition, we also show that the associativity of the bicrossed product given in Theorem \ref{teo:bowtie-monoidal} satisfies the pentagon axiom.

\subsubsection*{Proof of Lemma \ref{diagrmas-e-Sonia}}
\begin{proofw}

We will first show that the diagram \eqref{diagrama-e-Sonia-1} commutes.
For this, notice that the diagram
\begin{adjustbox}{scale=.5,center}
        \begin{tikzcd}
            e\trid(e\trid(A\ot A')) \ar[dd, swap,"(\mathbf{L}^2_{e,e})_{A\ot A'}"] \ar[rrrr, "e\trid \gamma^{e}_{A,A'}"]& && & e\trid (e\trid A \ot (e\fizb \degf{A})\trid A') \ar[dd, "\gamma^e_{e\trid A, e\trid A'}"]\\
          \hspace{.5cm}(i) & (ii) &  e\trid (e\trid A\ot A')\ar[dd,"\gamma_{e\trid A, A'}^e"]\ar[rru,"e\trid (\id_{e\trid A}\ot \mathbf{L}^0_{A'})"] & (iv) & \\
            e \trid(A\ot A')\ar[rru,swap,"e\trid (\mathbf{L}^0_A\ot \id_{A'})"] \ar[dddrr, swap,"\gamma^{e }_{A,A'}"]\ar[uu, bend right, swap, "e\trid \mathbf{L}^0_{A\ot A'}"] & (iii) && & e\trid(e\trid A)\ot (e\fiz(e\tridb\degf{A}))\trid ((e\fizb \degf{A})\trid A') 
            \ar[llddd,"(\mathbf{L}^2_{e,e})_A\ot (\mathbf{L}^2_{e, e})_{A'}"]     \\
            &  & e\trid (e\trid A)\ot e\trid A'\ar[rru,"\id\ot e\trid \mathbf{L}^0_{A'}"]& (v)&\\
            &&& & \\
            && e \trid A \ot (e\fizb \degf{A})\trid A'\ar[uu, swap, "e\trid \mathbf{L}^0_A\ot \id_{e\trid A'}"] & &
        \end{tikzcd}
        \end{adjustbox} commutes by \eqref{eq:gamma-T2-diagram}. In addition, \noindent $(i)$ and $(v)$ commute by \eqref{action-eq2}, $(iii)$ commutes by the naturality of $\gamma^e_{-,A'}$, and $(iv)$ commutes by the naturality of $\gamma^e_{e\trid A, -}$. This shows that $(ii)$ also commutes. Since $e\trid-$ is an equivalence, then  \eqref{diagrama-e-Sonia-1} commutes, as desired. 

To show that the diagram \eqref{diagrama-e-Sonia-2} commutes, notice first that the following diagram

\begin{adjustbox}{scale=.8,center}
 \begin{tikzcd}
   e\trid (e\trid \uno) \ar[rrr, "({\mathbf L}^2_{e,e})_\uno"]\ar[ddd,swap, "e \trid \gamma^e_0"]  & & & e\trid \uno\ar[ddd,"\gamma^e_0"]\ar[lll,swap, bend left=30, "e\trid ({\mathbf L}^0)_\uno"] \\
     &   &   & \\
  \hspace{2cm} (i) & & (ii)& \\
   A\ot A'\ar[rrr, "\gamma^e_0"]\ar[rrruuu,"\id"] & & & \uno.
\end{tikzcd}
\end{adjustbox} commutes by \eqref{action-eq2}. Since $(ii)$ commutes trivially then $(i)$ commutes. This shows that \eqref{diagrama-e-Sonia-2} commutes because $e\trid -$ is an equivalence. The commutativity of diagrams \eqref{diagrama-e-Sonia-3} and \eqref{diagrama-e-Sonia-4} is proven similarly.
    \end{proofw}

\subsubsection*{Proof of Theorem \ref{teo:bowtie-monoidal} - Pentagon}
\begin{proofw}
    We apply Lemma \ref{lemma:commutativity-both-sides-bowtie} to prove the commutativity of the pentagon axiom for $\alpha$. We only do the left side (the analog to diagram \eqref{eq:LI-lemma-delingue-product}) as the other side is similar. Let $A\bowtie C$, $A'\bowtie C'$,  $A''\bowtie C''$ and $A'''\bowtie C'''$ be homogeneous objects in $\cA \bowtie\cC$. The corresponding diagram is in Figure \ref{fig:pentagono-bicrossed-lhs-objetos}. The diagram commutes for the following reasons:
\begin{enumerate}[leftmargin=*,label=(\roman*)]
 \item[$(i)$] pentagon axiom for $\alpha^{\cA}$,
	\item[$(ii)$] naturality of $\alpha_{A, \degf{C}\trid A', -}$ applied to $\id_{((\degf{C}\fizb\degf{A'})\degf{C'})\trid A''} \ot (\bL^2_{((\degf{C}\fizb\degf{A'})\degf{C'})\fizb\degf{A''},\degf{C''}})^{-1}_{A'''}$,
	\item[$(iii)$]  naturality of $\alpha_{A, \degf{C}\trid A', -}$ applied to $(\gamma^{(\degf{C}\fizb \degf{A'})\degf{C'}}_{A'', \degf{C''}\trid A'''})^{-1}$,
	\item[$(iv)$] naturality of $\alpha_{A, - , ((((\degf{C}\fizb \degf{A'})\degf{C'})\fizb \degf{A''})\degf{C''})\trid A'''}$ applied to $\id_{\degf{C}\trid A'}\ot (\bL^2_{\degf{C}\fizb\degf{A'}, \degf{C'}})^{-1}_{A''}$,
	\item[$(v)$] naturality of $\alpha_{\degf{C}\trid A', - , ((((\degf{C}\fizb \degf{A'})\degf{C'})\fizb \degf{A''})\degf{C''})\trid A'''}$ applied to $(\bL^2_{\degf{C}\fizb\degf{A'}, \degf{C'}})^{-1}_{A''}$,
	\item[$(vi)$]  maps applied to different slots of the tensor product,
	\item[$(vii)$] because of \eqref{eq:gamma-T2-diagram},
	\item[$(viii)$] naturality of $\alpha_{A, - , ((((\degf{C}\fizb \degf{A'})\degf{C'})\fizb \degf{A''})\degf{C''})\trid A'''}$ applied to $(\gamma_{A',\degf{C'}\trid A''}^{\degf{C}})^{-1}$,
	\item[$(ix)$]  maps applied to different slots of the tensor product,
	\item[$(x)$] naturality of $\alpha_{\degf{C}\trid A', (\degf{C}\fizb \degf{A'})\trid(\degf{C'}\trid A'') ,-}$ applied to \newline $(\bL^2_{\degf{C}\fizb (\degf{A'}(\degf{C'}\tridb \degf{A''})),(\degf{C'}\fizb \degf{A''})\degf{C''}})^{-1}_{A'''}$,
	\item[$(xi)$] because of \eqref{action-eq1},
	\item[$(xii)$] because of \eqref{eq:gamma-alpha-diagram},
	\item[$(xiii)$] naturality of $(\gamma^{\degf{C}\fizb \degf{A'}}_{\degf{C'}\trid A'', -})^{-1}$ applied to $(\bL^2_{\degf{C'}\fizb \degf{A''}, \degf{C''}})^{-1}_{A'''}$,
	\item[$(xiv)$] naturality of $(\gamma^{\degf{C}}_{A',-})^{-1}$ applied to $\id_{\degf{C'}\trid A} \ot (\bL^2_{\degf{C'}\fizb \degf{A''}, \degf{C''}})^{-1}_{A'''}$,
	\item[$(xv)$] naturality of $(\gamma^{\degf{C}}_{A',-})^{-1}$ applied to $(\gamma^{\degf{C'}}_{A'', \degf{C''}\trid A'''})^{-1}$. 
\end{enumerate}

\begin{figure}[htbp]
	\centering
	\rotatebox{90}{
		\begin{adjustbox}{width= \textheight,totalheight=\textwidth,keepaspectratio}
		\begin{tikzcd}[ampersand replacement=\&]
			\&\& \objpentagonOne \\
			\\
			\\
			\&\&\& \objpentagonTen \\
			\\
			\& \objpentagonTwo \\
			\&\& \objpentagonFour \&\& \objpentagonEleven \\
			\\
			\\
			\objpentagonSix \&\&\& \objpentagonNine \\
			\& \objpentagonThree \\
			\&\& \objpentagonFive \&\& \objpentagonTwelve \\
			\\
			\\
			\\
			\& \objpentagonSeven \\
			\objpentagonTwentySix \&\&\& \objpentagonTwentyFive \\
			\\
			\\
			\\
			\\
			\\
			\objpentagonFourteen \&\& \objpentagonEight \\
			\\
			\\
			\\
			\& \objpentagonSixteen \&\& \objpentagonTwentyFour \\
			\\
			\\
			\\
			\\
			\\
			\&\& \objpentagonTwentyThree \&\& \objpentagonThirteen \\
			\\
			\\
			\\
			\& \objpentagonFifteen \&\& \objpentagonNineteen \\
			\\
			\\
			\&\& \objpentagonTwentyTwo \\
			\\
			\\
			\& \objpentagonTwenty \&\& \objpentagonEighteen \\
			\\
			\\
			\&\&\&\& \objpentagonSeventeen \\
			\&\& \objpentagonTwentyOne
			\arrow["\objpentmapOneToTwo"{description}, curve={height=18pt}, from=1-3, to=6-2]
			\arrow["\objpentmapOneToFour"{description}, from=1-3, to=7-3]
			\arrow["\objpentmapTwoToThree"{description}, from=6-2, to=11-2]
			\arrow["\objpentmapThreeToFive"{description}, curve={height=30pt}, from=11-2, to=12-3]
			\arrow["\objpentmapFourToFive"{description}, from=7-3, to=12-3]
			\arrow["\objpentmapTenToNine"{description}, from=4-4, to=10-4]
			\arrow["\objpentmapFiveToNine"{description}, curve={height=30pt}, from=12-3, to=10-4]
			\arrow["\objpentmapFourToTen"{description}, curve={height=-30pt}, from=7-3, to=4-4]
			\arrow["\objpentmapTenToEleven"{description}, curve={height=-30pt}, from=4-4, to=7-5]
			\arrow["\objpentmapElevenToTwelve"{description}, from=7-5, to=12-5]
			\arrow["\objpentmapNineToTwelve"{description}, curve={height=24pt}, from=10-4, to=12-5]
			\arrow["\objpentmapTwoToSix"{description}, curve={height=30pt}, from=6-2, to=10-1]
			\arrow["\objpentmapThreeToSeven"{description}, from=11-2, to=16-2]
			\arrow["\objpentmapSixToSeven"{description}, curve={height=30pt}, from=10-1, to=16-2]
			\arrow["\objpentmapSixToTwentysix"{description}, from=10-1, to=17-1]
			\arrow["\objpentmapTwentysixToFourteen"{description}, from=17-1, to=23-1]
			\arrow["\objpentmapSevenToFourteen"{description}, curve={height=-30pt}, from=16-2, to=23-1]
			\arrow["\objpentmapNineToTwentyfive"{description}, from=10-4, to=17-4]
			\arrow["\objpentmapFiveToEight"{description}, from=12-3, to=23-3]
			\arrow["\objpentmapEightToTwentyfive"{description}, curve={height=30pt}, from=23-3, to=17-4]
			\arrow["\objpentmapSevenToEight"{description}, curve={height=30pt}, from=16-2, to=23-3]
			\arrow["\objpentmapTwentyfiveToTwentyfour"{description}, from=17-4, to=27-4]
			\arrow["\objpentmapEightToTwentythree"{description}, from=23-3, to=33-3]
			\arrow["\objpentmapTwentythreeToTwentyfour"{description}, curve={height=30pt}, from=33-3, to=27-4]
			\arrow["\objpentmapSevenToSixteen"{description}, from=16-2, to=27-2]
			\arrow["\objpentmapSixteenToTwentythree"{description}, curve={height=30pt}, from=27-2, to=33-3]
			\arrow["\objpentmapSixteenToFifthteen"{description}, from=27-2, to=37-2]
			\arrow["\objpentmapFourteenToFifthteen"{description}, curve={height=30pt}, from=23-1, to=37-2]
			\arrow["\objpentmapTwentythreeToTwentytwo"{description}, from=33-3, to=40-3]
			\arrow["\objpentmapTwentyfourToNineteen"{description}, from=27-4, to=37-4]
			\arrow["\objpentmapTwentytwoToNineteen"{description}, curve={height=30pt}, from=40-3, to=37-4]
			\arrow["\objpentmapFifthteenToTwenty"{description}, from=37-2, to=43-2]
			\arrow["\objpentmapTwentyToTwentyone"{description}, curve={height=30pt}, from=43-2, to=47-3]
			\arrow["\objpentmapTwentytwoToTwentyone"{description}, from=40-3, to=47-3]
			\arrow["\objpentmapNineteenToEighteen"{description}, from=37-4, to=43-4]
			\arrow["\objpentmapTwentyoneToEighteen"{description}, curve={height=30pt}, from=47-3, to=43-4]
			\arrow["\objpentmapTwelveToThirdteen"{description}, from=12-5, to=33-5]
			\arrow["\objpentmapNineteenToThirdteen"{description}, curve={height=30pt}, from=37-4, to=33-5]
			\arrow["\objpentmapEighteenToSeventeen"{description}, curve={height=30pt}, from=43-4, to=46-5]
			\arrow["\objpentmapThirdteenToSeventeen"{description}, from=33-5, to=46-5]
			\arrow["\rm{{(i)}}"{description}, draw=none, from=6-2, to=7-3]
			\arrow["\rm{{(ii)}}"{description}, draw=none, from=7-3, to=10-4]
			\arrow["\rm{{(iii)}}"{description}, draw=none, from=4-4, to=12-5]
			\arrow["\rm{{(iv)}}"{description}, draw=none, from=10-1, to=11-2]
			\arrow["\rm{{(v)}}"{description}, draw=none, from=11-2, to=23-3]
			\arrow["\rm{{(vi)}}"{description}, draw=none, from=12-3, to=17-4]
			\arrow["\rm{{(vii)}}"{description}, draw=none, from=10-4, to=33-5]
			\arrow["\rm{{(viii)}}"{description}, draw=none, from=17-1, to=16-2]
			\arrow["\rm{{(ix)}}"{description}, draw=none, from=23-1, to=27-2]
			\arrow["\rm{{(x)}}"{description}, draw=none, from=16-2, to=33-3]
			\arrow["\rm{{(xi)}}"{description}, draw=none, from=23-3, to=27-4]
			\arrow["\rm{{(xii)}}"{description}, draw=none, from=37-2, to=40-3]
			\arrow["\rm{{(xiii)}}"{description}, draw=none, from=33-3, to=37-4]
			\arrow["\rm{{(xiv)}}"{description}, draw=none, from=40-3, to=43-4]
			\arrow["\rm{{(xv)}}"{description}, draw=none, from=37-4, to=46-5]
		\end{tikzcd}
		\end{adjustbox}
	}
	\caption{}
	\label{fig:pentagono-bicrossed-lhs-objetos}
\end{figure}
\end{proofw}

\subsubsection*{Proof of Proposition \ref{prop: subcategorias}}
\begin{proofw}
	We only check the statement for $\widetilde{\cA}$ as the one for $\widetilde{\cC}$ is similar. We only need to show that $\bF^2$ makes the corresponding diagram of  \cite[Definition 2.4.1]{EGNO-book} commute. The diagram is in Figure \ref{fig:inclusion-A}. The subdiagrams commute for the following reasons:
 \begin{enumerate}[leftmargin=*,label=(\roman*)]
     \item[$(i)$] Definition of $\alpha_{A\bowtie \uno, A'\bowtie \uno, A''\bowtie \uno}$,
     \item[$(ii)$] naturality of $\eta^{\degf{A''}}_{-,\uno}$ applied to $\eta_0^{\degf{A'}}$,
     \item[$(iii)$] because of \eqref{diagram:MP-18},
     \item[$(iv)$] naturality of $\alpha_{-,\uno\fiz\degf{A''},\uno}$ applied to $\eta_0^{\degf{A'}\degf{A''}}$,
     \item[$(v)$] because of \eqref{action-eq2},
     \item[$(vi)$] naturality of  $\gamma_{A',-}^e$  applied to  $\bL^{0}_{A''}$,
     \item[$(vii)$] because of \eqref{diagram:MP-15},
     \item[$(viii)$] because of \cite[Prop. 2.2.4]{EGNO-book},
     \item[$(ix)$] naturality of $\alpha_{A,-,e\trid A''}$ applied to $(\bL^0)_{A'}$,
     \item[$(x)$] because of \eqref{diagrama-e-Sonia-1},
     \item[$(xi)$] because of \eqref{diagram:MP-18},
     \item[$(xii)$] naturality of $\ell_{-}$ applied to $\eta_0^{\degf{A''}}$,
     \item[$(xiii)$] naturality of $\alpha_{A,A', -}$ applied to $(\bL^0)_{A''}$,
     \item[$(xiv)$] because of \eqref{diagram:MP-18}, and
 \end{enumerate}
 diagrams $(xv)$-$(xxv)$ commute trivially as they are just morphisms applied in different slots of the tensor product.

  \begin{figure}[htbp]
	\centering
	\hspace{50pt}\rotatebox{90}{
		\begin{adjustbox}{width=\textheight,totalheight=\textwidth,keepaspectratio}
\begin{tikzcd}[ampersand replacement=\&]
	{((A\ot e\trid A)\ot e\trid A'')\bowtie ((\uno\fiz\degf{A'}\ot \uno)\fiz\degf{A''}\ot\uno)} \&\&\&\&\&\& {(A\ot e\trid(A'\ot e\trid A'') )\bowtie (\uno\fiz \degf{A'}\degf{A''} \ot (\uno\fiz \degf{A''}\ot \uno))} \\
	\&\& {(A\ot (e\trid A'\ot e\trid(e\trid A'')))\bowtie((\uno\fiz\degf{A'}\ot \uno)\fiz\degf{A''}\ot\uno)} \&\& {(A\ot e\trid(A'\ot e\trid A'') )\bowtie (((\uno\fiz \degf{A'})\fiz\degf{A''} \ot \uno\fiz \degf{A''})\ot \uno)} \\
	\& {(A\ot (e\trid A'\ot e\trid A''))\bowtie((\uno\fiz\degf{A'}\ot \uno)\fiz\degf{A''}\ot\uno)} \&\& {(A\ot e\trid(A'\ot e\trid A'') )\bowtie ((\uno\fiz\degf{A'}\ot \uno)\fiz\degf{A''}\ot\uno)} \&\& {(A\ot e\trid(A'\ot e\trid A'') )\bowtie ((\uno\fiz \degf{A'}\degf{A''} \ot \uno\fiz \degf{A''})\ot \uno)} \\
	\\
	\\
	\&\&\&\& {(A\ot e\trid(A'\ot e\trid A'') )\bowtie ((\uno\fiz\degf{A''} \ot \uno\fiz \degf{A''})\ot \uno)} \\
	\& {(A\ot (e\trid A'\ot e\trid A''))\bowtie((\uno\ot \uno)\fiz\degf{A''}\ot\uno)} \&\&\&\& {(A\ot e\trid(A'\ot e\trid A'') )\bowtie ((\uno \ot \uno\fiz \degf{A''})\ot \uno)} \\
	{((A\ot e\trid A')\ot e\trid A'')\bowtie ((\uno\ot \uno)\fiz\degf{A''}\ot\uno)} \&\& {(A\ot (e\trid A'\ot e\trid(e\trid A'')))\bowtie(\uno\fiz\degf{A''}\ot\uno)} \& {(A\ot e\trid(A'\ot e\trid A'') )\bowtie ((\uno\ot \uno)\fiz\degf{A''}\ot\uno)} \\
	\&\&\&\&\&\& {(A\ot e\trid(A'\ot e\trid A'') )\bowtie (\uno \ot (\uno\fiz \degf{A''}\ot \uno))} \\
	\\
	\\
	\&\&\& {(A\ot e\trid(A'\ot e\trid A''))\bowtie(\uno\fiz\degf{A''}\ot\uno)} \\
	\\
	{((A\ot e\trid A')\ot e\trid A'')\bowtie (\uno\fiz\degf{A''}\ot\uno)} \& {(A\ot (e\trid A'\ot e\trid A''))\bowtie(\uno\fiz\degf{A''}\ot\uno)} \&\&\&\&\& {(A\ot e\trid(A'\ot e\trid A'') )\bowtie (\uno\fiz \degf{A'}\degf{A''} \ot (\uno\ot \uno))} \\
	\\
	\&\&\&\& {(A\ot e\trid(A'\ot e\trid A''))\bowtie(\uno\ot\uno)} \& {(A\ot e\trid(A'\ot e\trid A''))\bowtie(\uno\ot(\uno\ot\uno))} \\
	\\
	\\
	\&\& {(A\ot e\trid(A'\ot A''))\bowtie(\uno\fiz\degf{A''}\ot\uno)} \& {(A\ot e\trid(A'\ot e\trid A''))\bowtie((\uno\fiz\degf{A'})\fiz\degf{A''}\ot\uno)} \& {(A\ot e\trid(A'\ot e\trid A'') )\bowtie (\uno\fiz \degf{A'}\degf{A''} \ot \uno)} \\
	\&\&\&\&\& {(A\ot e\trid(A'\ot e\trid A'') )\bowtie (\uno\fiz \degf{A'}\degf{A''} \ot (\uno\ot \uno))} \\
	{((A\ot A')\ot e\trid A'')\bowtie (\uno\fiz\degf{A''}\ot\uno)} \& {(A\ot (A'\ot e\trid A''))\bowtie(\uno\fiz\degf{A''}\ot\uno)} \\
	\\
	\&\&\&\&\&\& {(A\ot e\trid(A'\ot A'') )\bowtie (\uno\fiz \degf{A'}\degf{A''} \ot (\uno\fiz \degf{A''}\ot \uno))} \\
	{((A\ot A')\ot A'')\bowtie (\uno\fiz\degf{A''}\ot\uno)} \\
	\&\&\&\&\& {(A\ot e\trid(A'\ot A'') )\bowtie (\uno\fiz \degf{A'}\degf{A''} \ot (\uno\ot \uno))} \\
	\\
	{((A\ot A')\ot A'')\bowtie (\uno\ot\uno)} \&\&\& {(A\ot e\trid(A'\ot A'') )\bowtie ((\uno\fiz \degf{A'})\fiz\degf{A''} \ot \uno)} \\
	\&\& {(A\ot (A'\ot A''))\bowtie(\uno\fiz\degf{A''}\ot\uno)} \\
	\\
	\\
	\&\&\&\& {(A\ot e\trid(A'\ot A'') )\bowtie (\uno\fiz \degf{A'}\degf{A''} \ot \uno)} \\
	\&\&\& {(A\ot (A'\ot A'') )\bowtie ((\uno\fiz \degf{A'})\fiz\degf{A''} \ot \uno)} \\
	\\
	\&\& {(A\ot (A'\ot A'') )\bowtie (\uno \ot \uno)} \\
	\\
	\\
	{((A\ot A')\ot A'')\bowtie \uno} \\
	\&\& {(A\ot (A'\ot A'') )\bowtie \uno} \& {(A\ot (A'\ot A'') )\bowtie (\uno\fiz \degf{A'}\degf{A''} \ot \uno)}
	\arrow["{\alpha_{A\bowtie\uno,A'\bowtie\uno,A''\bowtie\uno}}", curve={height=-30pt}, from=1-1, to=1-7]
	\arrow["{\rm{(i)}}"{description}, draw=none, from=1-1, to=1-7]
	\arrow["{\alpha_{A,e\trid A',e\trid A''}\bowtie\id}"', curve={height=18pt}, from=1-1, to=3-2]
	\arrow["{\id\bowtie(\eta_0^{\degf{A'}}\ot \uno)\fiz \degf{A''}\ot \id_\uno}"{description}, from=1-1, to=8-1]
	\arrow["{\rm{(xv)}}"{description, pos=0.7}, curve={height=30pt}, draw=none, from=1-1, to=14-2]
	\arrow["{\id\bowtie\eta_0^{\degf{A'}\degf{A''}}\ot\id_{\uno\fiz\degf{A''}}\ot\id_\uno}"{description}, from=1-7, to=9-7]
	\arrow["{\id\bowtie\id_{\uno\fiz\degf{A'}\degf{A''}}\ot \eta_0^{\degf{A''}}\ot \id_\uno}"{description, pos=0.9}, curve={height=-30pt}, from=1-7, to=14-7]
	\arrow["{\id_A \ot e\trid(\id_{A'}\ot(\bL^0)^{-1}_{A''})\bowtie\id}"{description, pos=0.9}, curve={height=-30pt}, from=1-7, to=23-7]
	\arrow["{\id_{A}\ot(\gamma^e_{A',e\trid A''})^{-1}\bowtie \id}"', curve={height=18pt}, from=2-3, to=3-4]
	\arrow["{\id\bowtie(\bR^2_{\degf{A''},\degf{A'}})_\uno \ot \id_{\uno\fiz\degf{A''}}\ot \id_\uno}"', curve={height=18pt}, from=2-5, to=3-6]
	\arrow["{\id\bowtie (\eta^{\degf{A'}}_0\fiz\degf{A''}\ot \id_{\uno\fiz\degf{A''}})\ot \id_\uno}"{description}, from=2-5, to=6-5]
	\arrow["{\rm{(iii)}}"', curve={height=-12pt}, draw=none, from=2-5, to=7-6]
	\arrow["{\id_A \ot \id_{e\trid A'}\ot (\bL^2_{e,e})^{-1}\bowtie \id}"', curve={height=18pt}, from=3-2, to=2-3]
	\arrow["{\id\bowtie(\eta_0^{\degf{A'}}\ot \uno)\fiz \degf{A''}\ot \id_\uno}"{description}, from=3-2, to=7-2]
	\arrow["{\rm{(xvi)}}"{description}, curve={height=-30pt}, draw=none, from=3-2, to=12-4]
	\arrow["{\id\bowtie \eta^{\degf{A''}}_{\uno\fiz \degf{A'},\uno}\ot \id_\uno}"', curve={height=18pt}, from=3-4, to=2-5]
	\arrow["{\rm{(ii)}}"{description}, curve={height=12pt}, draw=none, from=3-4, to=6-5]
	\arrow["{\id\bowtie (\eta^{\degf{A'}}_0\ot \id_\uno)\fiz\degf{A''}\ot \id_\uno}"{description}, from=3-4, to=8-4]
	\arrow["{\id\bowtie\alpha_{\uno\fiz \degf{A'}\degf{A''},\uno\fiz\degf{A''},\uno}}"', curve={height=18pt}, from=3-6, to=1-7]
	\arrow["{\id\bowtie\eta_0^{\degf{A'}\degf{A''}}\ot\id_{\uno\fiz\degf{A''}}\ot\id_\uno}"{description}, from=3-6, to=7-6]
	\arrow["{\rm{(iv)}}"{description}, draw=none, from=3-6, to=9-7]
	\arrow["{\id \bowtie \eta_0^{\degf{A''}}\ot \id_{\uno \fiz \degf{A''}}\ot\id_{\uno}}", curve={height=-18pt}, from=6-5, to=7-6]
	\arrow["{\id \bowtie \ell_\uno\fiz\degf{A''}\ot\id_{\uno}}"{description}, from=7-2, to=14-2]
	\arrow["{\id\bowtie\alpha_{\uno, \uno\fiz\degf{A''},\uno}}"', curve={height=18pt}, from=7-6, to=9-7]
	\arrow["{\id\bowtie\ell_{\uno\fiz \degf{A''}}\ot\id_\uno}"{description}, curve={height=-18pt}, from=7-6, to=12-4]
	\arrow["{\id \bowtie \ell_\uno\fiz\degf{A''}\ot\id_{\uno}}"{description}, from=8-1, to=14-1]
	\arrow["{\,\id_A \ot (\gamma^e_{A',e\trid A''})^{-1}\bowtie \id}", curve={height=24pt}, from=8-3, to=12-4]
	\arrow["{\id\bowtie \eta^{\degf{A''}}_{\uno,\uno}\ot \id_\uno}"{description}, curve={height=18pt}, from=8-4, to=6-5]
	\arrow["{\id\bowtie\ell_{\uno}\fiz\degf{A''}\ot \id_{1}}"{description}, from=8-4, to=12-4]
	\arrow["{\id\bowtie\ell_{\uno\fiz \degf{A''} \ot \uno}}"{description}, curve={height=-30pt}, from=9-7, to=12-4]
	\arrow["{\rm{(viii)}}"{description}, curve={height=6pt}, draw=none, from=9-7, to=12-4]
	\arrow["{\rm{(xix)}}"{description}, curve={height=30pt}, draw=none, from=9-7, to=14-7]
	\arrow["{\id\bowtie\id_\uno \ot \eta_0^{\degf{A''}}\ot\id_\uno}"{description}, from=9-7, to=16-6]
	\arrow["{\rm{(vii)}}"'{pos=0.7}, curve={height=18pt}, draw=none, from=12-4, to=6-5]
	\arrow["{\id\bowtie\eta_0^{\degf{A''}}\ot\id_\uno}"{description}, from=12-4, to=16-5]
	\arrow["{\rm{(xii)}}"{description}, curve={height=12pt}, draw=none, from=12-4, to=16-6]
	\arrow["{\rm{(xi)}}"{description, pos=0.6}, curve={height=24pt}, draw=none, from=12-4, to=19-5]
	\arrow["{\alpha_{A,e\trid A',e\trid A''}\bowtie\id}"{description}, curve={height=18pt}, from=14-1, to=14-2]
	\arrow["{\id_A\ot (\bL^0)^{-1}_{A'}\ot \id_{e\trid A''}\bowtie \id}"{description}, from=14-1, to=21-1]
	\arrow["{\rm{(ix)}}"{description, pos=0.6}, curve={height=6pt}, draw=none, from=14-1, to=21-2]
	\arrow["{\id_A\ot\id_{e\trid A'}\ot(\bL^2_{e,e})^{-1}_{A''}\bowtie \id}", curve={height=-24pt}, from=14-2, to=8-3]
	\arrow["{\id_A\ot\id_{e\trid A'}\ot e\trid(\bL^0)_{A''}\bowtie \id}"', curve={height=24pt}, from=14-2, to=8-3]
	\arrow["{\rm{(v)}}"{marking, allow upside down}, draw=none, from=14-2, to=8-3]
	\arrow["{\rm{(vi)}}"{description}, curve={height=30pt}, draw=none, from=14-2, to=12-4]
	\arrow["{\id_A \ot (\gamma^e_{A',A''})^{-1}\bowtie \id}"{description}, from=14-2, to=19-3]
	\arrow["{\id_A\ot (\bL^0)^{-1}_{A'}\ot \id_{e\trid A''}\bowtie \id}"{description}, from=14-2, to=21-2]
	\arrow["{\rm{(x)}}"{description}, curve={height=-18pt}, draw=none, from=14-2, to=28-3]
	\arrow["\id"{description}, from=14-7, to=20-6]
	\arrow["{\rm{(xxii)}}"{description, pos=0.7}, draw=none, from=14-7, to=25-6]
	\arrow["{\rm{(xviii)}}"{description}, draw=none, from=16-5, to=20-6]
	\arrow["{\id\bowtie\id_\uno \ot \ell_{\uno}}"{description}, curve={height=-18pt}, from=16-6, to=16-5]
	\arrow["{\id_A\ot e\trid(\id_{A'}\ot \bL^0_{A''})\bowtie \id}"{description}, from=19-3, to=12-4]
	\arrow["{\rm{(xvii)}}"{description}, curve={height=24pt}, draw=none, from=19-3, to=19-4]
	\arrow["{\id_A\ot (\bL^0)^{-1}_{A'\ot A''}\bowtie \id}"{description}, from=19-3, to=28-3]
	\arrow["{\id\bowtie\eta_0^{\degf{A'}}\fiz\degf{A''}\ot\id_\uno}"{description}, from=19-4, to=12-4]
	\arrow["{\id\bowtie (\bR^2_{\degf{A'},\degf{A''}})_{\uno}\ot \id_\uno}"{description}, curve={height=-30pt}, from=19-4, to=19-5]
	\arrow["{\rm{(xxi)}}"{description}, draw=none, from=19-4, to=31-5]
	\arrow["{\id\bowtie\eta_0^{\degf{A'}\degf{A''}}\ot\id_\uno}"{description}, from=19-5, to=16-5]
	\arrow["{\id\bowtie\eta_0^{\degf{A'}\degf{A''}}\ot\id_\uno}"{description}, from=20-6, to=16-6]
	\arrow["{\id\bowtie\id_{\uno\fiz\degf{A'}\degf{A''}}\ot \ell_\uno}"{description}, curve={height=24pt}, from=20-6, to=19-5]
	\arrow["{\alpha_{A, A',e\trid A''}\bowtie\id}"{description}, curve={height=18pt}, from=21-1, to=21-2]
	\arrow["{\id_A\ot \id_{A'}\ot (\bL^0)^{-1}_{A''}\bowtie \id}"{description}, from=21-1, to=24-1]
	\arrow["{\rm{(xiii)}}"{description, pos=0.4}, curve={height=30pt}, draw=none, from=21-1, to=28-3]
	\arrow["{\id_A\ot \id_{A'}\ot (\bL^0)^{-1}_{A''}\bowtie \id}"{description}, curve={height=12pt}, from=21-2, to=28-3]
	\arrow["{\id\bowtie\id_{\uno\fiz\degf{A'}\degf{A''}}\ot \eta_0^{\degf{A''}}\ot\id_\uno}"{description}, from=23-7, to=25-6]
	\arrow["{\id\bowtie \eta_0^{\degf{A''}}\ot\id_\uno}"{description}, from=24-1, to=27-1]
	\arrow["{\alpha_{A, A',A''}\bowtie\id}"{description}, curve={height=30pt}, from=24-1, to=28-3]
	\arrow["{\rm{(xxiii)}}"{description}, curve={height=30pt}, draw=none, from=24-1, to=34-3]
	\arrow["{\id\bowtie\id_{\uno\fiz\degf{A'}\degf{A''}}\ot \ell_\uno}"{description}, curve={height=30pt}, from=25-6, to=31-5]
	\arrow["{\alpha_{A, A',A''}\bowtie\id}"{description}, curve={height=30pt}, from=27-1, to=34-3]
	\arrow["{\id \bowtie \ell_\uno}"{description}, from=27-1, to=37-1]
	\arrow["{\rm{(xxv)}}"{description}, curve={height=30pt}, draw=none, from=27-1, to=38-3]
	\arrow["{\id\bowtie (\eta_0^{\degf{A'}})\fiz\degf{A''}\ot\id_\uno}"{description}, from=27-4, to=19-3]
	\arrow["{\id_A\ot e\trid(\id_{A'}\ot \bL^0_{A''})\bowtie \id}"{description}, from=27-4, to=19-4]
	\arrow["{\id\bowtie (\bR^2_{\degf{A'},\degf{A''}})_{\uno}\ot \id_\uno}"', curve={height=30pt}, from=27-4, to=31-5]
	\arrow["{\id\ot (\bL^0_{A'\ot A''})^{-1}\bowtie \id}"{description}, from=27-4, to=32-4]
	\arrow["{\rm{(xxiv)}}"{description, pos=0.6}, curve={height=-30pt}, draw=none, from=27-4, to=38-4]
	\arrow["{\id\bowtie \eta_0^{\degf{A''}}\ot\id_\uno}", from=28-3, to=34-3]
	\arrow["{\id_A \ot e\trid(\id_{A'}\ot \bL^0_{A''})\bowtie\id}"{description}, from=31-5, to=19-5]
	\arrow["{\id\ot (\bL^0_{A'\ot A''})^{-1}\bowtie \id}"{description}, from=31-5, to=38-4]
	\arrow["{\rm{(xx)}}"{description}, draw=none, from=32-4, to=19-3]
	\arrow["{\id\bowtie (\eta_0^{\degf{A'}})\fiz\degf{A''}\ot\id_\uno}"{description}, from=32-4, to=28-3]
	\arrow["{\id\bowtie (\bR^2_{\degf{A'},\degf{A''}})_{\uno}\ot \id_\uno}"{description}, curve={height=30pt}, from=32-4, to=38-4]
	\arrow["{\rm{(xiv)}}"{description}, draw=none, from=34-3, to=32-4]
	\arrow["{\id \bowtie \ell_\uno}"{description}, from=34-3, to=38-3]
	\arrow["{((A\ot A')\ot A'')\bowtie \uno}"{description}, curve={height=30pt}, from=37-1, to=38-3]
	\arrow["{\id\bowtie \eta_0^{\degf{A'}\degf{A''}}\ot\id_\uno}"{description}, curve={height=-30pt}, from=38-4, to=34-3]
\end{tikzcd}
		\end{adjustbox}
	}
	\caption{}
	\label{fig:inclusion-A}
\end{figure}

 \end{proofw}

\subsubsection*{Proof of Proposition \ref{prop:pivotal-bicrossed-product}}
\begin{proofw}
We have to check that $p_{A\bowtie C\ot A'\bowtie C'}=p_{A\bowtie C}\ot p_{A'\bowtie C'}$, for all $A, A'\in \Irr(\cA), C, C'\in \Irr(\cC)$. We have that, 
\begin{equation}
\begin{split}
    p_{A\bowtie C\ot A'\bowtie C'}= & r_{A^{**}\ot |C|\trid A'^{**}}(\id_{A^{**}\ot |C|\trid A'^{**}}\ot (\gamma^e_0)^{-1}) (p^\cA_{A\ot |C|\trid A'}\ot e\trid \id_\uno)\\ \label{eq:pivotal1}
    &(\id_{A\ot |C|\trid A'}\ot \gamma_0^e)r^{-1}_{A\ot |C|\trid A'}\bowtie \ell_{C^{**}\fiz |A'|\ot C'^{**}}((\eta_0^e)^{-1}\ot \id_{C^{**}\fiz |A'|\ot C'^{**}})\\
    & (\id_\uno\fiz e\ot p^\cC_{C\fiz |A'|\ot C'})(\eta_0^e\ot \id_{C\fiz |A'|\ot C'})\ell^{-1}_{C\fiz |A'|\ot C'},
    \end{split}
\end{equation}
    and, 
\begin{equation}
    \begin{split}
        p_{A\bowtie C}\ot p_{A'\bowtie C'} = & r_{A^{**}}(\id_{A^{**}}\ot (\gamma_0^e)^{-1})(p^\cA_A\ot e\trid \id_{\uno})(\id_A\ot \gamma^e_0)r_A^{-1}\ot \\ \label{eq:pivotal2}
        & |C|\trid r_{A'^{**}}(\id_{A'^{**}}\ot (\gamma_0^e)^{-1})(p^\cA_{A'}\ot e\trid \id_{\uno_\cA})(\id_{A'}\ot \gamma_0^e)r^{-1}_{A'}\bowtie \\ 
        & \ell_{C^{**}}((\eta_0^e)^{-1}\ot \id_{C^{**}})(\id_\uno\fiz e\ot p^\cC_C)(\eta^e_0\ot\id_C)\ell_C^{-1}\fiz |A'|\ot \\
        & \ell_{C'^{**}}((\eta_0^e)^{-1}\ot \id_{C'^{**}})(\id_\uno\fiz e\ot p^\cC_{C'})(\eta_0^e\ot \id_{C'})\ell_{C'}^{-1}.
    \end{split}
\end{equation}

By Lemma \ref{lemma:commutativity-both-sides-bowtie}, we can check both sides of $\bowtie$ separately. We explicitly check that the left-hand side of the $\bowtie$ in \eqref{eq:pivotal1} is equal to the left-hand side of the $\bowtie$ in \eqref{eq:pivotal2}. Checking the right-hand side is a similar computation. In Figure \ref{fig:pivotal}, $(i)$ commutes by the naturality of $r$, and $(ii), (iv), (v), (vi)$, and $(vii)$ trivially commute. We rewrite $(iii)$ in Figure \ref{fig:pivotal-2}. Then we have that $(a)$ commutes because $p^\cA$ is the pivotal structure of $\cA$, $(b)$ and $(e)$  commute from the naturality of $r$, and $(c), (d), (f), (g)$, and $(h)$ trivially commutes.

\begin{figure}[htbp]
    \centering
 \hspace{50pt}\rotatebox{90}{
   \begin{adjustbox}{width=580pt}
\begin{tikzcd}[ampersand replacement=\&]
    A\ot |C|\trid A'\ar[rr, "r^{-1}_{A\ot |C|\trid A'}"]\ar[dd,"r_A^{-1}\ot |C|\trid r_{A'}^{-1}"]\ar[ddrr,"p_{A\ot |C|\trid A'}"] \&   \& (A\ot |C|\trid A')\ot \uno\ar[rr,"\id\ot \gamma_0^e"]\ar[rrdd,"p^\cA_{A\ot |C|\trid A'}\ot \id_{\uno}"] \&  \& (A\ot |C|\trid A')\ot e\trid \uno\ar[ddrr,"p^\cA_{A\ot |C|\trid A'}\ot e\trid \id_{\uno}"] \& \& \\
    \& \& (i) \&\& (ii)\& \& \\
   (A\ot \uno)\ot(|C|\trid (A'\ot \uno))\ar[dd,pos=.7,"\id\ot \gamma_0^e\ot |C|\trid (\id\ot \gamma_0^e)"]\ar[rrdd,"p_A^\cA\ot\id\ot|C|\trid (p^\cA_{A'}\ot \id)"] \& \& A^{**}\ot |C|\trid A'^{**}\ar[rr,"r^{-1}_{A^{**}\ot |C|\trid A'^{**}}"]\ar[rrrrdddd,"\id"] \&\& (A^{**}\ot |C|\trid A'^{**})\ot \uno\ar[rr,"\id\ot\gamma_0^e"]\ar[ddrr,"\id"] \&\& (A^{**}\ot |C|\trid A'^{**})\ot e\trid \uno\ar[dd,"\id\ot (\gamma_0^e)^{-1}"]\\
    \&\& (iii) \&\& (iv)\&  \& \hspace{-2cm}(v) \\
    (A\ot e\trid \uno) \ot |C|\trid (A'\ot e\trid \uno)\ar[dd,pos=.4, "p^\cA_A\ot e\trid \id\ot|C|\trid(p^{\cA}_{A'}\ot e\trid \id)"] \& (vi)  \& (A^{**}\ot \uno) \ot |C|\trid (A'^{**}\ot \uno)\ar[ddll,"\id\ot\gamma_0^e\ot |C|\trid (\id\ot \gamma_0^e)"]\ar[rrdd,"\id"] \& \&\& \& (A^{**}\ot |C|\trid A'^{**})\ot\uno\ar[dd,"r_{A^{**}\ot |C|\trid A'^{**}}"]\\
    \& \& (vii)\&\&\& \& \\
    (A^{**}\ot e\trid \uno) \ot |C|\trid (A'^{**}\ot e\trid \uno)\ar[rrrr,"\id\ot(\gamma_0^e)^{-1}\ot|C|\trid (\id\ot(\gamma_0^e)^{-1})"] \&\& \&\& (A^{**}\ot \uno) \ot |C|\trid (A'^{**}\ot \uno)\ar[rr,"r_{A^{**}}\ot |C|\trid r_{A'^{**}}"] \& \& A^{**}\ot |C|\trid A'^{**}
    \end{tikzcd}
\end{adjustbox}}
 \caption{}
\label{fig:pivotal}
\end{figure}

\begin{figure}[htbp]
    \centering
 \hspace{50pt}\rotatebox{90}{
   \begin{adjustbox}{width=560pt}
\begin{tikzcd}[ampersand replacement=\&]
    A\ot |C|\trid A'\ar[ddddr,"r_A^{-1}\ot\id_{|C|\trid A'}"]\ar[dddddddd,"r^{-1}_A\ot |C|\trid r^{-1}_{A'}"]\ar[rrrr,"p^{\cA}_{A\ot |C|\trid A'}"]\ar[rrdd,"p^\cA_A\ot\id_{|C|\trid A'}"] \&\& \&\&  A^{**}\ot |C|\trid A'^{**}\\
    \&\& (a)\&\& \\
    \& (b) \&  A^{**}\ot |C|\trid A'^{**}\ar[rruu,"\id_{A^{**}}\ot |C|\trid p^\cA_{A'}"]\ar[dd,"r^{-1}_{A^{**}}\ot \id_{|C|\trid A'}"] \& (c) \& \\
     \hspace{2cm}(f)\& \&\&\& \hspace{-2cm}(g)\\
    \& A\ot\uno\ot|C|\trid A'\ar[r,"p_A^\cA\ot\id"]\ar[ldddd,"\id_{A\ot \uno}\ot|C|\trid r_{A'}^{-1}"] \& (A^{**}\ot\uno)\ot |C|\trid A'\ar[dd,"\id_{A^{**}\ot\uno}\ot |C|\trid r_{A'}^{-1}"]\ar[r,"\id\ot |C|\trid p^\cA_{A'}"] \& (A^{**}\ot\uno)\ot |C|\trid A'^{**}\ar[ruuuu,"r_{A^{**}\ot\id_{|C|\trid A^{**}}}"] \& \\
    \& (d)\&\& (e)\&\\
    \&\& (A^{**}\ot\uno)\ot |C|\trid (A'\ot \uno)\ar[rrdd,swap, "\id_{A^{**}\ot \uno}\ot |C|\trid (p^\cA_{A'}\ot \id_\uno)"] \&\& \\
    \&\& (h) \&\&\\
   A\ot\uno\ot|C|\trid (A'\ot \uno)\ar[rruu,swap,"p_A^{\cA}\ot\id_\uno\ot\id_{|C|\trid (A'\ot \uno)}"]\ar[rrrr,"p^\cA_A\ot\id_\uno\ot |C|\trid (p^\cA_{A'}\ot \id_\uno)"] \&\&  \&\& A^{**}\ot\uno\ot|C|\trid (A'^{**}\ot \uno)\ar[uuuuuuuu,"r_{A^{**}}\ot |C|\trid r_{A'^{**}}"]\ar[luuuu,"\id_{A^{**}\ot\uno}\ot |C|\trid r_{A'^{**}}"]
\end{tikzcd}
\end{adjustbox}}
\caption{}
\label{fig:pivotal-2}
\end{figure}
\end{proofw}

\subsubsection*{Proof of Proposition \ref{prop:equiv-to-deligne}}
\begin{proofw}
    Let be $\bF: \cA \boxtimes\cC\to \cA\bowtie\cC$, $F(A\boxtimes C)=A\bowtie C$ and 
    \begin{align*}
        \bF^2_{A\boxtimes C, A'\boxtimes C'}&: A\otimes \degf{C}\trid A'\bowtie C\fiz \degf{A'}\otimes C'\to A\otimes A'\bowtie C\otimes C'
    \end{align*} given by $\id_A\otimes (u_{\degf{C}})_{A'}\bowtie (v_{\degf{A'}})_C\otimes \id_{C'}$. To check that $\bF$ is monoidal, the left hand of $\bowtie$ side is the diagram from Figure \ref{fig:equivalence-trivial-action-to-delingue}, where $(i)$ commutes from the naturality of $\alpha_{A,-,|C||C'|\trid A''}$; $(ii)$ commutes from the naturality of $\alpha_{A, A',-}$; $(iii)$ and $(iv)$ trivially commutes; $(v)$ commutes from the naturality of $\gamma^{|C|}_{A',-}$; $(vi)$ commutes because $u_{|C|}$ is a monoidal natural transformation and $(vii)$ commutes since $u$ is a monoidal natural transformation. The right-hand side is analogous.
\begin{figure}[htbp]
    \centering
 \hspace{50pt}\rotatebox{90}{
   \begin{adjustbox}{width=\textheight,totalheight=\textwidth,keepaspectratio}
\begin{tikzcd}[ampersand replacement=\&]
(A\otimes \degf{C}\trid A')\otimes (\degf{C}\fiz \degf{A'})\degf{C'}\trid A'' \ar[rrr, "\alpha_{A, \degf{C}\trid A', \degf{C}\degf{C'}\trid A''}"]\ar[dd,"\id_A\ot (u_{\degf{C}})_{A'}"] \& \& \& A\ot (\degf{C}\trid A'\ot \degf{C}\degf{C'}\trid A'')\ar[rrr, "\id_A\ot\id_{\degf{C}\trid A'}\otimes (\mathbf{L}^2_{\degf{C}\degf{C'}})_{A''}^{-1}"]\ar[dd,"\id_A\ot (u_{\degf{C}})_{A'}\ot \id"] \& \& \& A\ot (\degf{C}\trid A'\ot \degf{C}\trid (\degf{C'}\trid A''))\ar[dd,"\id_A\ot(\gamma^{\degf{C}}_{A', \degf{C'}\trid A''})^{-1}"]\ar[ddddlll,swap, "\id_A\ot (u_{\degf{C}})_{A'}\ot \id"] \ar[ddddddlll,"\id_A\ot\id_{|C|\trid A'}\ot |C|\trid (u_{|C'|})_{A''}"] \\
\&\& (i) \&\&\& (iii)\&\\
(A\ot A' )\ot \degf{C}\degf{C'}\trid A''\ar[rrr,"\alpha_{A, A', \degf{C}\degf{C'}\trid A''}"]\ar[dddddddd,"\id_A\ot (u_{|C||C'|})_{A''}"]\&\&\&  A\ot (A' \ot \degf{C}\degf{C'}\trid A'')\ar[dd,swap,"\id_A\ot\id_{A'}\ot (\mathbf{L}^2_{|C|,|C'|})_{A''}^{-1}"] \ar[dddddddd,bend right=80,swap,"\id_{A\ot A'}\ot (u_{|C||C'|})_{A''}"] \&\&\& A\ot \degf{C}\trid (A'\ot \degf{C'}\trid A'')\ar[dd,"\id_A\ot \degf{C}\trid (\id_{A'}\ot (u_{\degf{C'}})_{A''})"]\\
\&\& \&\&\&\&\\
 \&\&\& A\ot (A'\ot |C|\trid (|C'|\trid A''))\ar[dddd,bend right=75,pos=.1,"\id_A\ot\id_{A'}\ot|C|\trid (u_{|C'|})_{A''}"] \&\& (v) \& A\ot \degf{C}\trid (A' \ot A'')\ar[ddddddlll, "\id_A\ot (u_{|C|})_{A'\ot A''}"]\\
\& (ii) \& \& (iv)\&\&\&\\
\&\&\& A\ot (|C|\trid A'\ot|C|\trid A'')\ar[rrruu,swap,"\id_A\ot (\gamma_{A',A''}^{|C|})^{-1}"]\ar[dd,"\id_A\ot (u_{|C|})_{A'}\ot\id_{|C|\trid A''}"]\&(vi)\&\&\\
\&\&\&\hspace{-7cm}(vii)\&\&\&\\
\&\&\& A\ot (A' \ot |C|\trid A'')\ar[dd, swap,"\id_A\ot\id_{A'}\ot (u_{|C|})_{A''}"]\&\&\&\\
\&\&\&\&\&\&\\
(A\ot A')\ot A''\ar[rrr,"\alpha_{A, A', A''}"] \&\&\& A\ot (A' \ot A'') \&\&\& 
\end{tikzcd}
\end{adjustbox}
}
    \caption{}
    \label{fig:equivalence-trivial-action-to-delingue}
\end{figure}

    \end{proofw}

 \clearpage 

\bibliographystyle{abbrv}
\bibliography{arxiv-version-2-On bicrossed product of fusion categories and exact factorizations}

\begin{thebibliography}{10}

\bibitem{generalization-G-nonsemisimple}
T.~Basak and S.~Gelaki.
\newblock Exact factorizations and extensions of finite tensor categories.
\newblock {\em Journal of Algebra}, 631:46--71, 2023.

\bibitem{EGNO-book}
P.~Etingof, S.~Gelaki, N.~D. Nikshych, and V.~Ostrik.
\newblock {\em Tensor Categories}.
\newblock Mathematical surveys and monographs. American Mathematical Society,
  2015.

\bibitem{ENO-no-published}
P.~Etingof, D.~Nikshych, and V.~Ostrik.
\newblock private communication.

\bibitem{ENO}
P.~Etingof, D.~Nikshych, and V.~Ostrik.
\newblock On fusion categories.
\newblock {\em Annals of Mathematics}, 162:581--642, 2005.

\bibitem{eno2}
P.~Etingof, D.~Nikshych, and V.~Ostrik.
\newblock Fusion categories and homotopy theory.
\newblock {\em Quantum Topol.}, 3:209--273, 2010.

\bibitem{FK}
J.~Frohlich and T.~Kerler.
\newblock {\em Quantum groups, quantum categories and quantum field theory}.
\newblock Springer Berlin, Heidelberg, 6 1993.

\bibitem{G-exact-factorization}
S.~Gelaki.
\newblock Exact factorizations and extensions of fusion categories.
\newblock {\em Journal of Algebra}, 480:505--518, 2017.

\bibitem{GN}
S.~Gelaki and D.~Nikshych.
\newblock Nilpotent fusion categories.
\newblock {\em Advances in Mathematics}, 217(3):1053--1071, 2008.

\bibitem{kac}
G.~Kac.
\newblock Extensions of groups to ring groups.
\newblock {\em Mathematics of the USSR-Sbornik}, 5(3):451, 1968.

\bibitem{Mackey}
G.~Mackey.
\newblock Products of subgroups and projective multipliers.
\newblock {\em Colloquia Math. Soc. Janos Bolyai 5, Hilbert Space Operators,
  Tihany (Hungary)}, pages 401--413, 1970.

\bibitem{ms}
G.~Moore and N.~Seiberg.
\newblock {Classical and quantum conformal field theory}.
\newblock {\em Communications in Mathematical Physics}, 123(2):177--254, 1989.

\bibitem{Na_faithful}
S.~Natale.
\newblock Faithful simple objects, orders and gradings of fusion categories.
\newblock {\em Algebraic \& Geometric Topology}, 13:1489--1511, 2011.

\bibitem{ostrik}
V.~Ostrik.
\newblock Module categories, weak hopf algebras and modular invariants.
\newblock {\em Transform. Groups}, 8(2):177--206, 2003.

\bibitem{rt}
N.~Reshetikhin and V.~G. Turaev.
\newblock Invariants of 3-manifolds via link polynomials and quantum groups.
\newblock {\em Inventiones mathematicae}, 103(1):547--597, 1991.

\bibitem{Takeuchi-matched}
M.~Takeuchi.
\newblock Matched pairs of groups and bismash products of hopf algebras.
\newblock {\em Communications in Algebra}, 9(8):841--882, 1981.

\bibitem{TY}
D.~Tambara and S.~Yamagami.
\newblock Tensor categories with fusion rules of self-duality for finite
  abelian groups.
\newblock {\em Journal of Algebra}, 209(2):692--707, 1998.

\bibitem{turaev-viro}
V.~Turaev and O.~Viro.
\newblock State sum invariants of 3-manifolds and quantum 6j-symbols.
\newblock {\em Topology}, 31(4):865--902, 1992.

\end{thebibliography}

\end{document}